\pdfoutput=1
\RequirePackage{ifpdf}
\ifpdf 
\documentclass[pdftex]{sigma}
\else
\documentclass{sigma}
\fi

\numberwithin{equation}{section}
\usepackage{bbm}
\usepackage{enumerate}

 \newcommand{\Lie}{\mathcal{L}}

 \newcommand{\N}{\mathbb{N}}
 \newcommand{\Z}{\mathbb{Z}}
 \newcommand{\Q}{\mathbb{Q}}
 \newcommand{\R}{\mathbb{R}}

 \newcommand{\TT}{\mathcal{T}}
 
 \newcommand{\GG}{\mathcal{G}}

 \newcommand{\id}{\mathrm{id}}

\newtheorem{Theorem}{Theorem}[section]
\newtheorem{maintheorem}{Theorem}

\newtheorem{maincorollary}[maintheorem]{Corollary}

\newtheorem{Lemma}[Theorem]{Lemma}
\newtheorem{Proposition}[Theorem]{Proposition}

\theoremstyle{definition}

\newtheorem{Remark}[Theorem]{Remark}
\newtheorem{Example}[Theorem]{Example}

\begin{document}

\allowdisplaybreaks

\newcommand{\arXivNumber}{2601.16804}

\renewcommand{\thefootnote}{}

\renewcommand{\PaperNumber}{116}

\FirstPageHeading

\ShortArticleName{Length Spectrum Rigidity and Flexibility of Spheres of Revolution with One Equator}

\ArticleName{Length Spectrum Rigidity and Flexibility\\ of Spheres of Revolution with One Equator\footnote{This paper is a~contribution to the Special Issue on Geometry and Dynamics in memory of Will Merry. The~full collection is available at \href{https://sigma-journal.com/Merry.html}{https://sigma-journal.com/Merry.html}}}

\Author{Alberto ABBONDANDOLO~$^{\rm ac}$ and Marco MAZZUCCHELLI~$^{\rm bc}$}

\AuthorNameForHeading{A.~Abbondandolo and M.~Mazzucchelli}

\Address{$^{\rm a)}$~Ruhr-Universit\"at Bochum, Fakult\"at f\"ur Mathematik,\\
\phantom{$^{\rm a)}$~}Geb\"aude IB 3/65, 44801 Bochum, Germany}
\EmailD{\mail{alberto.abbondandolo@rub.de}}

\Address{$^{\rm b)}$~Sorbonne Universit\'e, Universit\'e Paris Cit\'e, CNRS, IMJ-PRG,\\
\phantom{$^{\rm b)}$}~75005 Paris, France}
\EmailD{\mail{marco.mazzucchelli@imj-prg.fr}}
\Address{$^{\rm c)}$~Institute for Advanced Study, School of Mathematics,\\
\phantom{$^{\rm c)}$~}1 Einstein Drive, Princeton, NJ~08540, USA}

\ArticleDates{Received January 28, 2026, in final form October 27, 2026; Published online November 16, 2026}

\Abstract{We define a~notion of marked length spectrum for $S^1$-symmetric Riemannian metrics on the two-sphere having only one equator. We prove that isospectral metrics in this class have conjugate geodesic flows. Under a~further $\Z_2$-symmetry assumption, we show that the marked length spectrum determines the metric. Finally, we show that every isospectral class of metrics contains a~unique $\Z_2$-symmetric metric and give an explicit description of this isospectral class as an infinite-dimensional convex set, generalizing the known description of $S^1$-symmetric Zoll metrics. This paper contains also two appendices, in which we provide an elementary proof of the fact that a~$C^2$ real-valued function on an interval is determined by the set of tangent lines to its graph, and we classify a~class of $S^1$-invariant contact forms on three-manifolds.}

\Classification{53D25; 37J35; 37C27}


\renewcommand{\thefootnote}{\arabic{footnote}}
\setcounter{footnote}{0}

\section{Introduction}\label{s:intro}

An important question in dynamics is to understand how much information is encoded in the periodic orbits of a~system. In the case of geodesic flows on a~negatively curved orientable closed Riemannian surface $(M,g)$, this question is completely answered by the celebrated rigidity theorem of Otal~\cite{ota90} and Croke~\cite{cro90}: the map that assigns to each non-trivial free homotopy class of closed loops in $M$ the length of the unique closed geodesic in that class determines the metric $g$ up to pullback by a~diffeomorphism homotopic to the identity.

This metric rigidity phenomenon crucially relies on the hyperbolicity of the geodesic flow. It holds more generally for Anosov geodesic flows on orientable closed surfaces~\cite{glp25}, but has no analogue on surfaces such as the two-sphere, where uniform hyperbolicity cannot occur. For instance, on the two-sphere we have a~large class of \textit{Zoll metrics}, i.e., Riemannian metrics all of whose geodesics are closed and of equal length, see~\cite{bes78,gui76,zol03}. Any two Zoll metrics with the same common geodesic length are indistinguishable in terms of the data encoded by their closed geodesics, but are in general not isometric. However, it is worth noting that here dynamical rigidity still holds: the geodesic flows of any two Zoll metrics on the two-sphere with the same common geodesic length are smoothly conjugate~\cite[Appendix B]{abhs17}.

Motivated by the example of Zoll metrics, in this paper, we undertake a~case study of a~special class of metrics on the two-sphere.

We consider the unit sphere $S^2 \subset \R^3$ equipped with the $S^1$-action given by rotations about the $z$-axis, and denote by $\mathcal{G}$ the set of $S^1$-invariant smooth Riemannian metrics on $S^2$ that possess a~unique equator, i.e., a~single $S^1$-invariant unoriented closed geodesic. Spheres of revolution in~$\R^3$ with a~unique equator are in this class, but there are metrics $g\in \mathcal{G}$ such that $\bigl(S^2,g\bigr)$ does not embed isometrically into $\R^3$ as a~sphere of revolution.

Any unoriented closed geodesic $\gamma$ that is neither the equator nor a~meridian (i.e., a~closed geodesic passing through the two fixed points of the $S^1$-action) has a~non-zero winding number~${p \in \N}$ around the $z$-axis and intersects the equator transversely $2q$ times, for some $q \in \N$ with $\gcd(p,q) = 1$. Such a~geodesic is said to be of \textit{type} $(p,q)$. In the terminology of Angenent~\cite{ang05}, a~closed geodesic of type $(p,q)$ is a~$(q,p)$-\emph{satellite} of the equator.
For a~fixed coprime pair $(p,q) \in \N \times \N$, we denote by $\mathcal{L}_g(p,q)$ the subset of $(0,+\infty)$ consisting of the lengths of all closed geodesics of type $(p,q)$. This set may be empty, finite, infinite, or even uncountable.

By \textit{marked length spectrum} of a~metric $g\in \mathcal{G}$, we here mean the data given by the length~$\ell$ of the equator and the set-valued function $\mathcal{L}_g$. We say that two metrics $g_1$ and $g_2$ in $\mathcal{G}$ are \textit{isospectral} if they have the same marked length spectrum. Note that in this definition of marked length spectrum we do not keep track of the length of the meridians, and we do not record how many $S^1$-families of closed geodesics of type $(p,q)$ have length $\ell\in \mathcal{L}_g(p,q)$. It follows from Theorems~\ref{thmA} and~\ref{thmB} below that these data are nevertheless determined by the marked length spectrum.

Given $g \in \mathcal{G}$, we denote by $T^1_g S^2$ the unit tangent bundle of $\bigl(S^2,g\bigr)$ and recall that $T^1_g S^2$ has a~co-oriented contact structure which is induced by the canonical contact structure of the spherical cotangent bundle $ST^*S^2$ using the isomorphism $TS^2 \cong T^*S^2$ induced by the metric. The geodesic flow $\phi^t_g$ on $T^1_g S^2$ is the Reeb flow of a~contact form $\alpha_g$ defining this contact structure, which goes under the name of Hilbert form.

The $S^1$-action on $S^2$ lifts to a~free $S^1$-action on $T^1_g S^2$. Denote by $\Gamma_g$ the subset of $T^1_g S^2$ consisting of the two orbits of the geodesic flow corresponding to the equator. These are precisely the two orbits of the $S^1$-action on $T^1_g S^2$ which are also orbits of the geodesic flow.

If the geodesic flows of two metrics $g_1$ and $g_2$ in $\mathcal{G}$ are conjugate by an $S^1$-equivariant and orientation-preserving homeomorphism mapping lifts of meridians to lifts of meridians, then the metrics are isospectral (see Proposition~\ref{p:coniso} and the paragraph following it). Conversely, we have the following dynamical rigidity result.

\begin{maintheorem}
\label{thmA}
Let $g_1, g_2 \in \mathcal{G}$ be smooth $($resp.\ analytic$)$ isospectral metrics. Then there exists a~smooth $($resp.\ analytic$)$ $S^1$-equivariant contactomorphism
\[
h\colon \ T^1_{g_1}S^2 \setminus \Gamma_{g_1} \rightarrow T^1_{g_2}S^2 \setminus \Gamma_{g_2}
\]
which conjugates the geodesic flows of $g_1$ and $g_2$, i.e.,
$h\circ \phi^t_{g_1} = \phi^t_{g_2} \circ h$, $\forall t\in \R$.
Furthermore,
\begin{enumerate}[$(i)$]\itemsep=0pt
\item If the curvature of $g_1$ along the equator is positive, then the same holds for $g_2$, and the above map $h$ can be chosen in such a~way that it
extends to a~smooth $($resp.\ analytic$)$ conjugacy from $T^1_{g_1} S^2$ to $T^1_{g_2} S^2$.
\item If the curvature of $g_1$ along the equator does not vanish to infinite order, then the same holds for $g_2$, and the above map $h$ can be chosen in such a~way that it extends to a~continuous conjugacy from $T^1_{g_1} S^2$ to $T^1_{g_2} S^2$.
\end{enumerate}
\end{maintheorem}

In particular, we have a~continuous conjugacy as in statement (ii) when one of the isospectral metrics is analytic. Actually, this continuous conjugacy exists under the more general assumption that the curvature of both isospectral smooth metrics is non-negative near the equator, see Proposition~\ref{p:non-neg-curv} below. However, the order of vanishing of the curvature at the equator is invariant under the isospectrality equivalence relation (see Proposition~\ref{p:order} below), whereas the weaker condition of having non-negative curvature near the equator is not, as can be deduced from Theorem~\ref{thmB} below. This is why we are stating (ii) in the above form.

An explicit example (see Proposition~\ref{p:notsmooth}) shows that if the curvature at the equator vanishes, then the conjugacy may fail to extend smoothly to the two orbits corresponding to the equator. If the curvature near the equator is somewhere negative, then a~crucial stability property of the equator which is used in order to prove that the conjugacy we construct extends continuously to these two orbits can fail (see Section~\ref{s:unstable}). This suggests that there could exist two isospectral metrics violating the assumption of (ii) for which no continuous conjugacy between the corresponding geodesic flows exists. However, we do not know how to construct such examples.

We now turn to the question of determining the isospectral class of a~metric in~$\mathcal{G}$. Any metric~$g$ in~$\mathcal{G}$ can be uniquely written as
$
g = {\rm d}\sigma^2 + r(\sigma)^2 {\rm d}\theta^2$,
where $\theta \in \R/2\pi \Z$ denotes the longitude angle on $S^2$, and $\sigma$ is the $g$-length parameter along the meridians starting from the south pole $S=(0,0,-1)$. Here, $r$ is a~smooth function on $[0, m]$ which is positive on $(0, m)$, vanishes at $0$ and $m$, and has a~unique critical point (necessarily a~maximum), corresponding to the equator. The quantity $m > 0$ denotes the length of the meridian arc from one pole to the other. The condition that $g$ is smooth (resp.\ analytic) is equivalent to the fact that $r$ extends to a~smooth (resp.\ analytic) $2m$-periodic odd function on $\R$ such that $r'(0) = 1 = -r'(m)$. We refer to the function $r \colon [0, m] \rightarrow \R$ as the \textit{profile function} of~$g$. For instance, the profile function of the round sphere with curvature $1$ is given by $r(\sigma) = \sin \sigma$ on the interval $[0, \pi]$.

In the next result, we characterize the profile functions of isospectral metrics in $\mathcal{G}$ or -- equivalently -- of metrics in $\mathcal{G}$ with conjugate geodesic flows in the sense of Theorem~\ref{thmA}.

\begin{maintheorem}
\label{thmB}
Let $r_1 \colon [0, m_1] \rightarrow \R$ and $r_2 \colon [0, m_2] \rightarrow \R$ be the profile functions of metrics $g_1, g_2 \in \mathcal{G}$. Then $g_1$ and $g_2$ are isospectral if and only if for every $\rho \geq 0$ we have
\[
\mathrm{length} \left( \{ \sigma \in [0, m_1] \mid r_1(\sigma) \geq \rho \} \right) = \mathrm{length} \left( \{ \sigma \in [0, m_2] \mid r_2(\sigma) \geq \rho \} \right),
\]
where $\mathrm{length}(I)$ denotes the length of the interval $I\subset \R$.
\end{maintheorem}

In particular, choosing $\rho=0$ in the above identity, we infer that $m_1=m_2$, and hence isospectral metrics have meridians of the same length $2m_1=2m_2$.
 A~metric $g \in \mathcal{G}$ is said to be \textit{$\Z_2$-symmetric} if it is symmetric with respect to the reflection $(x, y, z) \mapsto (x, y, -z)$. By the $S^1$-symmetry, this is equivalent to requiring $g$ to be symmetric with respect to the antipodal involution $(x,y,z) \mapsto (-x,-y,-z)$ or -- equivalently -- to induce a~smooth metric on the projective plane $\R \mathrm{P}^2$.
 A~first immediate consequence of Theorem~\ref{thmB} is the following result, giving us metric rigidity for $\Z_2$-symmetric metrics in $\mathcal{G}$.

\begin{maincorollary}\label{corC}
Any two isospectral $\Z_2$-symmetric metrics in $\mathcal{G}$ coincide.
\end{maincorollary}

The latter assertion should be seen as a~version for arbitrary metrics in $\GG$ of the well known fact that on $\R \mathrm{P}^2$ -- unlike $S^2$ -- there is a~unique Zoll metric whose geodesics have prescribed length (see~\cite{pri09} and references therein).

Denote by $\mathcal{G}_*$ the subset of $\mathcal{G}$ consisting of metrics whose curvature at the equator does not vanish to infinite order. All analytic metrics in $\mathcal{G}$ belong to $\mathcal{G}_*$. A~second consequence of Theorem~\ref{thmB} is the following metric flexibility result, which shows that the space of metrics in $\mathcal{G}_*$ that are isospectral to a~given one is extremely large -- essentially as large as the space of Zoll metrics in $\mathcal{G}$ -- and carries a~natural structure of an infinite-dimensional convex set.

\begin{maincorollary}\label{corD}
For every smooth $($resp.\ analytic$)$ metric $g\in \mathcal{G}_*$ there exists a~unique smooth $($resp.\ analytic$)$ $\Z_2$-symmetric metric $g_s\!\in\! \mathcal{G}_*$ which is isospectral to $g$. Moreover, if ${r_s \colon [0, m] \!\rightarrow\! \R}$ is the profile function of a~smooth $($resp.\ analytic$)$ $\Z_2$-symmetric metric $g_s\in \mathcal{G}_*$, then the smooth $($resp.\ analytic$)$ metrics $g\in \mathcal{G}_*$ which are isospectral to $g_s$ are precisely those whose profile function $r$ is given by
$
r := r_s \circ \varphi$,
where $\varphi \colon [0, m] \rightarrow [0, m]$ is the inverse of the diffeomorphism
\[
[0, m] \rightarrow [0, m], \qquad \tau \mapsto \tau + \psi(\tau),
\]
with $\psi \colon \R \rightarrow \R$ an arbitrary smooth $($resp.\ analytic$)$ odd function satisfying $\psi(m - \tau) = \psi(\tau)$ for every $\tau \in \R$, $\psi'(0) = 0$, and
$|\psi'| < 1$.
\end{maincorollary}

Using the parameter $\tau = \varphi(\sigma) \in [0, m]$ and expressing the $2m$-periodic even function $\psi'$ as a~function of $\cos\bigl( \frac{\pi}{m} \cdot \bigr)$, we deduce that the metrics which are isospectral to the $\Z_2$-symmetric metric $g_s$ as above are precisely those of the form
\[
g = \left(1 + f\left(\cos \left({ \frac{\pi}{m}} \tau\right)\right)\right)^2 {\rm d}\tau^2 + r_s(\tau)^2 {\rm d}\theta^2,
\]
where $f\colon [-1,1] \to (-1,1)$ is a~smooth (resp.\ analytic) odd function vanishing at $1$. In the special case $m = \pi$ and $r_s(\sigma) = \sin \sigma$, we recover the classical formula for $S^1$-invariant Zoll metrics on $S^2$ with all geodesics of length $2\pi$ (see~\cite[Corollary~4.16]{bes78}).

It would be possible to extend the above corollary to metrics in $\mathcal{G}$ whose curvature at the equator vanishes to infinite order, but the formulation would be less neat because if $g_s\in \mathcal{G}\setminus \mathcal{G}_*$ is $\Z_2$-symmetric, then the metrics whose profile functions are described by Corollary~\ref{corD} are still isospectral to $g_s$, but there might be more smooth metrics isospectral to $g_s$, as in this case the function $r$ defined above can be smooth also if $\varphi$ is not smooth. Moreover, for an arbitrary~${g\in \mathcal{G}\setminus \mathcal{G}_*}$, the unique $\Z_2$-symmetric metric $g_s$ which is isospectral to $g$ may not be~$C^4$ at the equator (see Remark~\ref{r:smooth?}).

It is interesting to compare the above results with the known spectral rigidity theorems for the Laplacian on $S^1$-symmetric Riemannian two-spheres. The flexibility statement of Corollary~\ref{corD} should be contrasted with the following rigidity result of Zelditch~\cite{zel98} for the class of analytic metrics in $\mathcal{G}$ of ``simple type'' (where the latter assumption includes the condition that for each positive number $L$, there exists at most one $S^1$-family of unoriented closed geodesics of length $L$, thereby excluding Zoll metrics): within this class, the eigenvalues of the Laplacian determine the metric up to isometry. Zelditch's rigidity result has been extended to a~more general class of $S^1$-symmetric metrics on $S^2$ by Dryden, Macedo, and Sena-Dias~\cite{dms16}, who showed that within this class the metric is determined by the spectrum of the Laplacian together with the weights that are induced by the $S^1$-action on each eigenspace. Under the same $\Z_2$-symmetry assumption of Corollary~\ref{corC}, any $S^1$-symmetric metric on $S^2$ is determined by the eigenvalues of the Laplacian, as proven by Br\"uning and Heintze~\cite{bh84}.

{\bf Proof ideas and organization of the paper.} In Section~\ref{s:general}, we review some basic facts about geodesic flows of $S^1$-symmetric Riemannian spheres and their description as $S^1$\nobreakdash-equivariant Reeb flows. In Section~\ref{s:birkhoff}, we recall that the geodesic flow of a~metric in $\mathcal{G}$ has a~global Birkhoff section given by the Birkhoff annulus, consisting of
the unit vectors which are based at the equator and point to one of the two hemispheres, and we study the properties of the corresponding first return time and first return map. In the subsequent Section~\ref{s:genfun}, we show how the latter two maps can be conveniently expressed in terms of a~function of one variable, which we refer to as the generating function. The generating function determines the normal form of the geodesic flow in the complement of the two orbits corresponding to the equator. The material of these three sections partly builds on~\cite{abhs21}.

In Section~\ref{s:mls}, we discuss a~homological description of the type of a~closed geodesic and show that the existence of an $S^1$-equivariant and orientation-preserving conjugacy between two geodesic flows mapping the lifts of meridians to lifts of meridians implies that the two metrics have the same marked length spectrum.

The first part of Theorem~\ref{thmA} -- stating the existence of a~smooth conjugacy between the geodesic flows of two isospectral metrics outside of the two orbits given by the equator -- is proven in Section~\ref{s:away}. The proof consists in showing that the marked length spectrum determines the generating function. This ultimately hinges on the fact that a~smooth real function of one variable is uniquely determined by the set of tangent lines to its graph. In~\cite{ric97}, Richardson uses tools from geometric measure theory to prove the latter statement for $C^{1,1}$ functions. In Appendix~\ref{s:appA}, we provide an elementary proof for $C^2$ functions.

Other approaches for the proof of the first part of Theorem~\ref{thmA} are possible. For instance, one could work with the Hamiltonian formulation on $T^* S^2$ and use action-angle coordinates on the open subset where the Hamiltonian induced by the metric and the first integral determined by the $S^1$-symmetry, which goes under the name of Clairaut's first integral, are independent, and then use a~general statement about spectral reconstruction of integrable systems from Cieliebak's PhD thesis~\cite{cie96} (after removing a~genericity assumption from it). Alternatively, one could build on the fact that the cotangent disk bundle induced by a~metric in $\mathcal{G}$ minus the fiber of one of the poles is symplectomorphic to a~toric domain, see~\cite{frv26}. The approach based on Birkhoff's annulus we are adopting here turns out to be useful also for addressing the other questions studied in this paper.\looseness=1

The smooth and analytic extension of the conjugacy at the two orbits corresponding to the equator, in the case in which the curvature along it is positive, is discussed in Section~\ref{s:regular}. The proof here builds on a~general classification result for a~class of $S^1$-invariant contact forms on a~three-manifold, which we discuss in Appendix~\ref{s:appB} and whose proof uses normal forms and Moser's homotopy argument.

In Section~\ref{s:continuous}, we prove that if the curvature of two isospectral metrics in $\mathcal{G}$ is non-negative near their equators, then the conjugacy between their geodesic flows extends continuously to the whole unit tangent bundle. This proof is based on an explicit computation of the longitude-shift and length along the geodesic flow, which leads to stability estimates over long time intervals of the orbits corresponding to the equator.

The same computations are used in Section~\ref{s:abel} in order to express the first return time and first return map to the Birkhoff annulus in terms of the profile function of the metric. This expression involves the Abel transform, and the injectivity of this classical integral transformation leads to the proof of Theorem~\ref{thmB}. A~first consequence of this theorem is the fact that the order of vanishing of the curvature at the equator is determined by the marked length spectrum. This fact, together with the steps discussed above, concludes the proof of Theorem~\ref{thmA}.

Corollary~\ref{corC} follows immediately from Theorem~\ref{thmB}. The proof of Corollary~\ref{corD} is also elementary, but requires more work and is presented in Section~\ref{s:isospectral_classes}.

In Section~\ref{s:notsmooth}, we construct an example of a~one-parameter family $\{g_{\varepsilon}\}$ of isospectral analytic metrics in $\mathcal{G}$, emanating from a~$\Z_2$-symmetric one $g=g_0$ with vanishing curvature at the equator, such that no $S^1$-equivariant orientation-preserving conjugacy between the geodesic flow of $g_{\varepsilon}$ and the one of $g$ has a~smooth extension to the orbits corresponding to the equator. The existence of this example is related to the following phenomenon in two-dimensional Hamiltonian dynamics: if two convex analytic Hamiltonians $H_1$ and $H_2$ on the plane have a~strict minimum at the origin with $H_1(0)=H_2(0)=0$ and for every $c>0$ the Hamiltonian orbits given by the circles $H^{-1}_1(c)$ and $H^{-1}_2(c)$ have the same period, then the corresponding Hamiltonian flows on $\R^2\setminus \{0\}$ are analytically conjugate -- as it can be seen by using action-angle coordinates -- but when the critical point at the origin is degenerate it is possible that no conjugacy extends smoothly to the origin. The classification up to smooth and analytic conjugacy of planar Hamiltonian systems with a~singularity of the kind we find here is discussed in a~recent paper of Martynchuk and V\~{u} \smash{Ng\d{o}c}~\cite{mvn}, and in the construction of our example we use some of their results.

In Section~\ref{s:unstable}, we construct examples of smooth metrics in $\mathcal{G}$ whose curvature is somewhere negative near the equator -- this requires the curvature to vanish to infinite order at the equator -- for which the stability property of the equator which is proven in Section~\ref{s:continuous} under the assumption of non-negative curvature near the equator fails.

\section[Geodesic flows on S\^{}1-symmetric spheres]{Geodesic flows on $\boldsymbol{S^1}$-symmetric spheres}
\label{s:general}

In this section, we fix some notation and recall some basic facts about geodesic flows on $S^2$ induced by an $S^1$-symmetric smooth (resp.\ analytic) metric $g$. Here, we identify $S^1$ with $\R/2\pi \Z$. We denote by $N:=(0,0,1)$ and $S:=(0,0,-1)$ the \textit{north} and \textit{south poles} of $S^2$, which are the fixed points of the $S^1$-action, and use geodesic polar coordinates $(\sigma,\theta)$ emanating from~$S$, i.e., we identify $S^2 \setminus \{S,N\}$ with the cylinder $(0,m) \times S^1$ where $\sigma\in (0,m)$ denotes the distance
from~$S$ measured with respect to the metric $g$ and $\theta\in S^1$ is the longitudinal angle on $S^2 \setminus \{S,N\}$, $m$ being the length of the meridian arc from $S$ to $N$. With respect to these coordinates, $g$ has the form\looseness=-1
\begin{gather}
\label{e:form-g}
g = {\rm d}\sigma^2 + r(\sigma)^2 {\rm d}\theta^2,
\end{gather}
where $r\colon[0,m] \rightarrow [0,+\infty)$ is a~function which is positive on $(0,m)$, vanishes at $0$ and $m$, and has a~smooth (resp.\ analytic) extension to $\R$ as a~$2m$-periodic odd function with $r'(0)=1=-r'(m)$. We refer to $r$ as the \textit{profile function} of the metric $g$.

The $S^1$-family of closed curves given by the intersection of $S^2$ with planes containing the $z$-axis consists of closed geodesics of length $2m$, which are called \textit{meridians}. The closed curves defined by fixing $\sigma$ in $(0,m)$ are called \textit{parallels} and are geodesics if and only if $\sigma$ is a~critical point of $r$. These special parallels are called \textit{equators}.

We denote by
$
T^1_g S^2 := \bigl\{ v\in TS^2\mid g(v,v) = 1 \bigr\}
$
the unit tangent bundle of $S^2$ and by ${\pi\colon T^1_g S^2 \rightarrow S^2}$ the footpoint projection. We shall omit the subscript $g$ whenever the context makes clear the metric we are working with. The same observation applies to the other $g$\nobreakdash-dependent objects we are going to introduce.

The global coordinates $(\sigma,\theta)\in (0,m)\times S^1$ on $S^2\setminus \{S,N\}$ extend to global coordinates $(\sigma,\theta,\beta)\in (0,m)\times S^1 \times S^1$ on $\pi^{-1}\bigl(S^2 \setminus \{S,N\}\bigr)$, where $\beta$ denotes the angle of a~unit tangent vector with the parallel it is based at, oriented positively. In other words, $(\sigma,\theta,\beta)$ corresponds to the unit tangent vector
\begin{gather}
\label{e:glob_coord}
\frac{\cos \beta}{r(\sigma)} \partial_\theta+\sin \beta \partial_{\sigma} \in T^1 S^2
\end{gather}
based at $(\sigma,\theta)\in S^2$. We will refer to these coordinates $(\sigma,\theta,\beta)$ on $T^1 S^2$ minus the fibers at the poles as \textit{geodesic polar coordinates} on the unit tangent bundle of the sphere.

The Hilbert contact form $\alpha = \alpha_g$ on $T^1 S^2$, evaluated at $v=(\sigma,\theta,\beta)\in T^1 S^2$, is given by
\begin{align}
\label{e:contact_form}
\alpha_v (\cdot) = g(v, {\rm d}\pi_v (\cdot)) = \sin \beta {\rm d}\sigma + r(\sigma)\cos \beta {\rm d}\theta.
\end{align}
The corresponding volume form on $T^1 S^2$ is
\begin{align}
\label{e:volume}
\alpha \wedge {\rm d}\alpha = r(\sigma) {\rm d}\sigma \wedge {\rm d}\theta \wedge {\rm d}\beta.
\end{align}
The geodesic flow
$
\phi^t = \phi^t_g \colon \ T^1 S^2 \rightarrow T^1 S^2$, $ t\in \R$,
is the flow of the Reeb vector field $R=R_g$ associated with the contact form $\alpha$ by means of the equations $\imath_R {\rm d}\alpha =0$ and $\alpha(R)=1$. In geodesic polar coordinates, $R$ can be written as
\begin{align}
\label{e:Reeb}
 R =\sin \beta \partial_{\sigma} + \frac{\cos \beta}{r(\sigma)} \partial_\theta + \frac{r'(\sigma)\cos\beta}{r(\sigma)} \partial_\beta.
\end{align}

The differential of the $S^1$-action on $S^2$ defines an analytic $S^1$-action on $TS^2$, which leaves~$T^1 S^2$ invariant and is readily seen to be free on $T^1 S^2$.
The generator $V=V_g$ of this $S^1$-action on~$T^1 S^2$ has the form $V=\partial_{\theta}$ outside of the fibers at the poles, and at the fibers of the poles is a~non-vanishing vertical vector field (i.e., a~vector field in the kernel of the differential of $\pi$). The corresponding quotient projection
$
\Pi\colon T^1 S^2 \rightarrow W = W_g:= T^1 S^2/S^1
$
is a~smooth (resp.\ analytic) principal $S^1$-bundle. Changing the metric produces isomorphic $S^1$-bundles. In the case of the standard round metric $g_0$, we have the identification
\[
T^1_{g_0} S^2 \cong \mathrm{SO}(3), \qquad (x,v) \mapsto
(x \ v \ x\times v), \qquad \forall x\in S^2, \ v\in T_x S^2,
\]
where $\times$ denotes the vector product in $\R^3$, and the $S^1$-action on $T^1 S^2$ corresponds to the action of $S^1 \cong \mathrm{SO}(2)$ -- seen as the subgroup of $\mathrm{SO}(3)$ consisting of the rotations fixing the $z$-axis~-- by left multiplication. The quotient by this action can be identified to the 2-sphere by the map
$
\mathrm{SO}(2) \backslash \mathrm{SO}(3) \cong S^2$, $ [A] \mapsto A e_3$,
and hence the base $W$ of the principal $S^1$-bundle $\Pi$ is diffeomorphic to $S^2$. The interested reader may verify that the Euler number of this $S^1$-bundle is 2.

The function $C = C_g:= \alpha(V)$ is known as \textit{Clairaut's first integral}. The fact that the contact form $\alpha$ is $S^1$-invariant implies that $C$ is a~first integral of the geodesic flow, meaning that it is constant along every orbit $t\mapsto\phi^t(v)$. The function $C$ vanishes on the fibers at the poles, and its expression outside of them is
$
C(\sigma,\theta,\beta)= r(\sigma) \cos\beta$.
The function $C$ is invariant under the $S^1$-action and its critical points form the set
$
\bigl\{ (\sigma,\theta,\beta) \in T^1 S^2 \mid r'(\sigma) = 0,\, \beta \in \{ 0,\pi\} \bigr\}$,
consisting of two $S^1$-orbits for each equator of $g$. All the regular level sets $C^{-1}(\rho)$ consist of unions of two-dimensional tori which are invariant under the geodesic flow. The torus $C^{-1}(0)$ is given by all the orbits which project to meridians.

Denote by $P_{\sigma} \subset S^2$ the parallel corresponding to the value $\sigma\in (0,m)$. The invariance of Clairaut's first integral implies the following result (see~\cite[Lemma 1.1]{abhs21} for the easy proof).

\begin{Lemma}
\label{l:behaviour-geodesics}
Let $\gamma\colon \R \rightarrow S^2$ be a~geodesic parametrized by unit speed which is not a~meridian. Then exactly one of the following two alternative conditions holds:
\begin{enumerate}[$(i)$]\itemsep=0pt
\item for $t \rightarrow -\infty$ and $t\rightarrow +\infty$ the geodesic $\gamma$ is asymptotic to two possibly coinciding equators~$P_{\sigma_-}$ and $P_{\sigma_+}$ with $r(\sigma_-) = r(\sigma_+) = |C(\dot\gamma)|$;
\item there exist numbers $0 < \sigma_1 < \sigma_2 < m$ such that $r(\sigma_1) = r(\sigma_2) = |C( \dot\gamma)| < r(\sigma)$ for every $\sigma \in (\sigma_1, \sigma_2)$, $r'(\sigma_1) > 0$, $r'(\sigma_2) < 0$, $\gamma$ is confined to the strip $\bigcup_{\sigma \in[\sigma_1,\sigma_2]} P_{
\sigma}$ and alternately touches both parallels $P_{\sigma_1}$ and $P_{\sigma_2}$ tangentially infinitely many times.
\end{enumerate}
\end{Lemma}

\section{The Birkhoff annulus}
\label{s:birkhoff}

We now assume that the smooth (resp.\ analytic) $S^1$-invariant metric $g$ on $S^2$ belongs to $\mathcal{G}$, meaning that it has just one equator. Equivalently, the profile function $r\colon [0,m] \rightarrow [0,+\infty)$ has positive derivative on $[0,\sigma_{\max})$, achieves its maximum $r_{\max}$ at $\sigma_{\max}\in (0,m)$, and has negative derivative on $(\sigma_{\max},m]$.

The \textit{Birkhoff annulus} is the set $\Sigma=\Sigma_g$ of vectors in $T^1 S^2$ which are based at the equator and point to the northern hemisphere. In geodesic polar coordinates, it has the form
\[
\Sigma = \bigl\{ (\sigma,\theta,\beta) \mid \sigma=\sigma_{\max}, \, \theta\in S^1,\, \beta\in (0,\pi)\bigr\}.
\]
Since $\sin \beta>0$ on $\Sigma$, \eqref{e:Reeb} implies that $\Sigma$ is transverse to the geodesic flow. For every $v=(\sigma_{\max},\theta,\beta) \in \Sigma$, there exists $t>0$ such that $\phi^t(v)\in \Sigma$. Indeed, this is clear if $\beta = \frac{\pi}{2}$, because in this case the geodesic $\gamma$ corresponding to the orbit $t\mapsto \phi^t(v)$ is a~meridian, and the claim holds with $t=2m$. If $\beta\in (0,\pi) \setminus \bigl\{ \frac{\pi}{2} \bigr\}$, then the geodesic $\gamma$ is not a~meridian and satisfies
\[
0< |C(\dot\gamma)| =r(\sigma_{\max}) |\cos \beta| < r_{\max},
\]
so alternative (ii) must hold in Lemma~\ref{l:behaviour-geodesics} because there are no equators $P_{\sigma}$ with $r(\sigma)=|C(\dot\gamma)|$. Statement (ii) tells us that $\gamma$ oscillates in the strip $\{\sigma_1 \leq \sigma \leq \sigma_2\}$, where $0 < \sigma_1 < \sigma_{\max} < \sigma_2 < m$
are such that $r(\sigma_1)=r(\sigma_2) = |C(\dot\gamma)|$, so the forward orbit of $v$ meets $\Sigma$ again.
 A~similar argument, still involving Lemma~\ref{l:behaviour-geodesics}, shows that every orbit of the geodesic flow other than the two orbits defined by the equator and forming the set $\Gamma=\Gamma_g = \partial \Sigma$ meets $\Sigma$.

In particular, there exists a~first return time $\tau\colon \Sigma \rightarrow (0,+\infty)$ defined by
\[
\tau(v) := \min \bigl\{ t>0 \mid \phi^t(v) \in\Sigma \bigr\},
\]
which is a~smooth (resp.\ analytic) function. We denote by $\varphi\colon \Sigma \rightarrow \Sigma$ the first return map, i.e., the smooth (resp.\ analytic) diffeomorphism
$
\varphi(v) := \phi^{\tau(v)}(v)$.
Denote by $\lambda$ the restriction of the contact form $\alpha$ to $\Sigma$. The fact that $\tau$ and $\varphi$ are the first return time and map of a~Reeb flow implies the following identity
\begin{gather}
\label{e:liouville}
{\rm d}\tau = \varphi^* \lambda - \lambda.
\end{gather}
Indeed, for $v\in \Sigma$ and $Y\in T_v \Sigma$ we compute
\begin{align*}
\bigl(\varphi^* \lambda\bigr)_v (Y) &= \lambda_{\varphi(v)} \bigl( {\rm d}\varphi_v (Y) \bigr) = \alpha_{\varphi(v)} \bigl( {\rm d}\phi^{\tau(v)}_v (Y) + {\rm d}\tau_v(Y) R(\varphi(v)) \bigr) \\ &= \alpha_v(Y) + {\rm d}\tau_v (Y) = \lambda_v (Y) + {\rm d}\tau_v (Y),
\end{align*}
where we have used the fact that the Reeb flow of $\alpha$ preserves $\alpha$.

We conclude this section by discussing the behavior of $\varphi$ and $\tau$ near the boundary $\Gamma$ of $\Sigma$. Denote by $T=T_g$ the 2-torus consisting of the vectors $v\in T^1 S^2$ which are based at the equator. Note that $T$ is the union of the pairwise disjoint sets $\Sigma$, \smash{$\widehat{\Sigma}$}, and \smash{$\Gamma = \partial \Sigma = \partial \widehat{\Sigma}$}, where \smash{$\widehat{\Sigma} = \widehat{\Sigma}_g$} is the other Birkhoff annulus of the equator, consisting of vectors in $T^1 S^2$ which are based at the equator and point to the southern hemisphere, i.e.,
\[
\widehat{\Sigma} = \bigl\{ (\sigma,\theta,\beta) \mid \sigma=\sigma_{\max},\, \theta\in S^1, \, \beta\in (-\pi,0)\bigr\}.
\]
By considering the first return map and time to $\widehat{\Sigma}$, we can extend $\varphi$ and $\tau$ to smooth (resp.\ analytic) maps
$
\varphi\colon T \setminus \Gamma \rightarrow T \setminus \Gamma$, $ \tau\colon T \setminus \Gamma \rightarrow (0,+\infty)$,
which are actually the second return map and time to $T \setminus \Gamma$. The behavior of $\varphi$ and $\tau$ near $\Gamma$ depends on whether the curvature of $g$ along the equator is positive or zero (by $S^1$-invariance this curvature is of course constant, and cannot be negative because we are assuming that there is only one equator) as explained in the next lemma.

\begin{Lemma}\label{l:extension}
If the curvature of $g$ along the equator is positive, then $\tau$ and $\varphi$ extend smoothly $($resp.\ analytically$)$ to $T$. These extensions are defined as follows: if $v\in \Gamma$, then $\tau(v)$ is the second conjugate instant to $\pi(v)$ along the equator and $\varphi(v)=\phi^{\tau(v)}(v)$. If the curvature along the equator vanishes, then $\tau$ diverges to $+\infty$ at $\Gamma$.
\end{Lemma}

The above result is well known (see, e.g., \cite[Lemma 4.3]{sch15} or~\cite[Proposition 3.2]{abhs17} for statements of this kind) but we include a~proof for the sake of completeness.

\begin{proof}
Assume that the curvature of $g$ along the equator is positive or, equivalently, that the equator has conjugate points. Let $\dot\gamma(0)\in \Gamma$, where $\gamma$ is a~parametrization by arc-length of the equator. Given $\beta\in \bigl(-\frac{\pi}{2},\frac{\pi}{2}\bigr)$, we denote by $\gamma_{\beta}\colon \R \rightarrow S^2$ the geodesic such that $\gamma_{\beta}(0)=\gamma(0)$ and~$\dot{\gamma}_{\beta}(0)$ makes an angle $\beta$ with $\dot\gamma(0)$. The vector field
$
J(t):= \partial_{\beta}|_{\beta=0} \gamma_{\beta}(t)
$
is a~non-trivial Jacobi field orthogonal to $\dot\gamma(t)$ and hence of the form
$
J(t) = j(t) \partial_{\sigma}
$
for some smooth (resp.\ analytic) function $j\colon \R \rightarrow \R$ vanishing at $0$. Denote by $t_0$ the second positive zero of $j$, i.e., the second conjugate instant to $\gamma(0)$ along $\gamma$.

For $\beta\neq 0$, $\dot{\gamma}_{\beta}(0)$ belongs to $T\setminus \Gamma$. Let
$\rho(\beta):= \tau(\dot\gamma_{\beta}(0))$
be the second positive instant $t$ for which $\gamma_{\beta}(t)$ lies on the equator. Therefore, $\rho(\beta)$ tends to $t_0$ for $\beta\rightarrow 0$, and we can extend~$\rho$ continuously to $0$ by setting $\rho(0):= t_0$. We wish to show that this extension of $\rho$ is smooth (resp.\ analytic) in a~neighborhood of $0$.

In the usual coordinate system $(\sigma,\theta)$ on the sphere minus the poles, we can write
\[
\gamma_{\beta}(t) = (\sigma_{\beta}(t),\theta_{\beta}(t)),
\]
so that
$
j(t) = \partial_{\beta}|_{\beta=0} \sigma_{\beta}(t)$.
Since the function $(\beta,t) \mapsto \sigma_{\beta}(t) - \sigma_{\max}$ is smooth (resp.\ analytic) and vanishes for $\beta=0$, the function
\[
u(\beta,t) :=
\begin{cases}
\dfrac{\sigma_{\beta}(t) - \sigma_{\max}}{\beta} & \mbox{if } \beta\neq 0, \\
j(t) & \mbox{if } \beta= 0
\end{cases}
\]
is smooth (resp.\ analytic) on $\bigl(-\frac{\pi}{2},\frac{\pi}{2}\bigr) \times \R$. By construction,
$
u(\beta,\rho(\beta)) = 0
$
for every $\beta$ in~${\left(-\frac{\pi}{2},\frac{\pi}{2}\right)}$. Since
\[
\partial_t u (0,\rho(0)) = \partial_t u (0,t_0) = j'(t_0) \neq 0,\]
because the Jacobi vector field $J$ is non-trivial and vanishes at $t_0$, the implicit function theorem implies that $\rho$ is smooth (resp.\ analytic) near $0$.
Taking also its $S^1$-invariance into account, we conclude that $\tau$ has a~smooth (resp.\ analytic) extension to $T$. Therefore, the map $v\mapsto \phi^{\tau(v)}(v)$ is smooth (resp.\ analytic) on $T$ as well.

When the curvature at the equator vanishes, the equator has no conjugate points and the function $\rho(\beta)$ diverges to $+\infty$ for $\beta\rightarrow 0$, and so does $\tau$ at $\Gamma$.
\end{proof}

\section{The generating function}\label{s:genfun}

In this section, we show how the first return time $\tau$ and the first return map $\varphi$ to the Birkhoff annulus $\Sigma$ can be expressed in terms of a~function of one variable. We follow~\cite{abhs21}, but adopting different normalization conventions.

It is convenient to use global coordinates $(\theta,\eta)\in S^1 \times (-1,1)$ on $\Sigma$, by setting $\eta= \cos \beta$. In other words, we identify the pair $(\theta,\eta)\in S^1 \times (-1,1)$ with the vector
\begin{gather}\label{e:cosi}
\frac{\eta}{r_{\max}} \partial_{\theta} + \sqrt{1-\eta^2} \partial_{\sigma} \in \Sigma,
\end{gather}
see~\eqref{e:glob_coord}. By~\eqref{e:contact_form}, the restriction $\lambda$ of the contact form $\alpha$ to $\Sigma$ has the form
$
\lambda = r_{\max} \eta {\rm d}\theta$,
and~\eqref{e:liouville} implies that $\varphi$ preserves the area form \smash{$\frac{1}{r_{\max}} {\rm d}\lambda = {\rm d}\eta \wedge {\rm d}\theta$}.

By the $S^1$-symmetry, $\tau(\theta,\eta)=\tau(\eta)$ is independent of the angular variable $\theta$, and $\varphi(\theta_1+\theta_2,\eta)=\varphi(\theta_1,\eta)+(\theta_2,0)$. The invariance of the Clairaut first integral implies that the second component of $\varphi(\theta,\eta)$ is $\eta$, and we conclude that
$\varphi$ has the form
\begin{align}
\label{e:return_map}
\varphi(\theta,\eta)=(\theta+f(\eta),\eta),
\end{align}
for some smooth (resp.\ analytic) function $f\colon (-1,1)\to\R$. Here, $f$ is defined up to addition of an integer multiple of $2\pi$. The fact that $\varphi$ fixes the vectors which are tangent to meridians implies that $f(0)\in 2\pi \Z$, and we make $f$ unique by the normalization
\begin{gather}\label{e:norm}
f(0)=2\pi.
\end{gather}
In order to discuss the geometric meaning of this normalization, we denote by $ \Theta(\theta,\eta) = \Theta(\eta)$ the total angular variation along the geodesic arc emanating from a~vector in $\Sigma$ up to the next encounter with $\Sigma$. More precisely, given a~vector $v\in \Sigma$ with coordinates $(\theta_0,\eta_0)$ with~${\eta_0\in (-1,1) \setminus \{0\}}$, we write in geodesic polar coordinates
\[
\phi^t(v) = (\sigma(t),\theta(t),\beta(t)) \qquad \forall t\in \R,
\]
where $\sigma(0)=\sigma_{\max}$, $\theta(0)=\theta_0$, $\eta_0 = \cos \beta(0)$, and $t\mapsto\theta(t)$ is continuous and real-valued.
We define
$
\Theta(\eta_0) := \theta(\tau(\eta_0)) - \theta_0$.
This defines a~real-valued function $\Theta$ on $(-1,1)\setminus \{0\}$, and the normalization~\eqref{e:norm} implies
$
f(\eta) = \Theta(\eta)$, $\forall \eta\in (0,1)$.
Indeed, by~\eqref{e:return_map} the left and right hand-side of the above identity differ by an integer multiple of $2\pi$, so the above identity follows from the fact that $\Theta(\eta)$ converges to $2\pi$ for $\eta\downarrow 0$ (see the proof of~\cite[Lemma 4.1]{abhs21}). Similarly~${
f(\eta) =\Theta(\eta) + 4\pi}$, $ \forall \eta\in (-1,0)$,
because $\Theta(\eta)$ converges to $-2\pi$ for $\eta\uparrow 0$.

The form~\eqref{e:form-g} of the metric implies that $g$ is symmetric with respect to reflections with respect to planes containing the $z$-axis, i.e., with respect to involutions of the form $(\sigma,\theta) \mapsto (\sigma,\theta_0-\theta)$. Therefore, the restriction of $\varphi$ to the $\varphi$-invariant half of the Birkhoff annulus consisting of points $(\theta,\eta)$ with $\eta\in [0,1)$ determines its restriction to the other half -- the one where~${\eta\in (-1,0]}$. Said otherwise, the restriction of $f$ to $[0,1)$ determines it on the whole~${(-1,1)}$.
Indeed, our choice of normalization~\eqref{e:norm} implies that the function $\eta \mapsto f(\eta) - 2\pi$ is odd.

By the expressions for $\varphi$, $\tau$, and $\lambda$, equation~\eqref{e:liouville} can be rewritten as
\[
\tau'(\eta) {\rm d}\eta = r_{\max} \bigl( \eta {\rm d} (\theta + f(\eta)) - \eta {\rm d}\theta \bigr) = r_{\max} \eta f'(\eta) {\rm d}\eta \qquad \forall \eta\in (-1,1),
\]
that is,
$
\tau'(\eta) = r_{\max} \eta f'(\eta) $, $ \forall \eta\in (-1,1)$.
Denoting by $F\colon(-1,1) \rightarrow \R$ a~primitive of $f$, we have
\[
\frac{{\rm d}}{{\rm d} \eta} \bigl( \eta F'(\eta) - F(\eta) \bigr) = \eta F''(\eta) = \eta f'(\eta),
\]
and hence the function $\eta\mapsto r_{\max} (\eta F'(\eta) - F(\eta))$ differs from $\tau$ by a~constant. By choosing the primitive $F$ of $f$ appropriately, we can ensure that
\[
\tau(\eta) = r_{\max} \bigl( \eta F'(\eta) - F(\eta) \bigr) \qquad \forall \eta \in (-1,1).
\]
Indeed, since the value of $\tau$ at vectors in $\Sigma$ which are tangent to meridians is $2m$, the above identity holds provided that $F$ is the primitive of $f$ such that $F(0)=-\frac{2m}{r_{\max}}$.

Finally, we discuss the behavior of $F$ near $1$ and $-1$. If the curvature of $g$ along the equator vanishes, then Lemma~\ref{l:extension} tells us that $\tau(v)$ diverges to $+\infty$ at $\Gamma=\partial \Sigma$, and we deduce that for~${\eta\rightarrow \pm 1}$ we have $\Theta(\eta) \rightarrow \pm \infty$ and hence $F'(\eta)=f(\eta)\rightarrow \pm \infty$.

If the curvature along the equator is positive, the same lemma gives us a~smooth (resp.\ analytic) extension of $\varphi$ to the torus $T$ consisting of vectors $v\in T^1 S^2$ based at the equator. Recalling that $\eta=\cos \beta$, this implies that
$
f(\cos \beta) = \tilde{f}(\beta)$, $ \forall \beta\in (0,\pi)$,
where ${\tilde{f}\colon\ \R \rightarrow \R}$ is a~smooth (resp.\ analytic) $2\pi$-periodic function, and hence $f$ has a~continuous extension to~${[-1,1]}$. Moreover, the $S^1$-symmetry implies that the function $\tilde{f}$ is even. Differentiating the above identity, we obtain
$
- f'(\cos \beta) \sin \beta = \tilde{f}'(\beta) $, $\forall \beta\in (0,\pi)$,
which, together with the fact that~$\tilde{f}'$ vanishes at $0$ and $\pi$ -- a~consequence of the $2\pi$-periodic function~$\tilde{f}$ being even -- implies that $f'$ has a~continuous extension to $[-1,1]$. We conclude that in the case in which the curvature along the equator is positive, the function $F$ has a~$C^2$ extension to $[-1,1]$.

We summarize what we have found in this section into the next lemma.

\begin{Lemma}\label{l:gen_fun}
There exists a~unique smooth $($resp.\ analytic$)$ function $F\colon (-1,1) \rightarrow \R$ such that
\begin{enumerate}[$(i)$]\itemsep=0pt
\item \smash{$F(0)=-\frac{2m}{r_{\max}}$}, $F'(0)=2\pi$;
\item the function $\eta \mapsto F(\eta) - 2\pi \eta$ is even;
\item $\varphi(\theta,\eta) = (\theta+F'(\eta),\eta)$ for every $(\theta,\eta)\in S^1 \times (-1,1) \cong \Sigma$;
\item $\tau(\theta,\eta) = r_{\max} ( \eta F'(\eta) - F(\eta))$ for every $(\theta,\eta)\in S^1 \times (-1,1) \cong \Sigma$;
\item $F'(\eta) = \Theta(\eta)$ for every $\eta\in (0,1)$ and $F'(\eta) = \Theta(\eta) + 4\pi$ for every $\eta\in (-1,0)$;
\item if the curvature along the equator is positive, then $F$ has a~$C^2$ extension to $[-1,1]$;
\item if the curvature along the equator vanishes, then $F'(\eta)$ diverges to $\pm \infty$ for $\eta \rightarrow \pm 1$.
\end{enumerate}
\end{Lemma}

We shall refer to the function $F\colon (-1,1) \rightarrow \R$ of the above lemma as the \textit{generating function} of the geodesic flow of $g$.

\begin{Remark}
One could show that $F$ always has a~continuous extension to $[-1,1]$, see~\cite[Lemma 4.3]{abhs21}. We shall not need this fact here.
\end{Remark}

Our final task in this section is to determine a~normal form for the geodesic flow on $T^1 S^2 \setminus \Gamma$, decomposing it into a~one-parameter family of linear flows on invariant tori.

\begin{Proposition}\label{p:normal-form}
The expression
\[
\Xi(\theta_1,\theta_2,\eta) := \left( \theta_1 - \frac{F'(\eta) - 2\pi}{2 \pi} \theta_2 \right) \cdot \phi^{\frac{\tau(v_\eta)}{2\pi} \theta_2} (v_{\eta}), \qquad \mbox{with } v_{\eta}:= (0,\eta) \in \Sigma,
\]
where the dot denotes the $S^1$-action on $T^1 S^2$, defines a~smooth diffeomorphism
\[
\Xi \colon\ S^1 \times S^1 \times (-1,1) \rightarrow T^1 S^2 \setminus \Gamma
\]
such that
\begin{gather}\label{e:normal-form}
\phi^t \circ \Xi (\theta_1,\theta_2,\eta) = \Xi \left( \theta_1 + \frac{F'(\eta) - 2\pi}{\tau(v_\eta)} t, \theta_2 + \frac{2\pi}{\tau(v_\eta)} t, \eta \right).
\end{gather}
Moreover, the pullback of the Hilbert contact form by $\Xi$ is
\begin{gather}\label{e:conta}
\Xi^* \alpha = r_{\max} \eta {\rm d}\theta_1 + r_{\max} \left( \eta - \frac{F(\eta)}{2\pi} \right) d \theta_2.
\end{gather}
\end{Proposition}

Note that by Lemma~\ref{l:gen_fun}\,(iv) the normal form~\eqref{e:normal-form} and the contact form~\eqref{e:conta} involve only the generating function $F$ and $r_{\max}$.

\begin{proof}We first check that $\Xi$ is well defined. The $2\pi$-periodicity with respect to $\theta_1$ is clear. So is the $2\pi$-periodicity with respect to $\theta_2$ on the level set $\eta=0$, because of~\eqref{e:norm} and the fact that~$v_0$ is tangent to a~meridian. For $\eta\neq 0$, the identities
\[
\phi^{\frac{\tau(v_\eta)}{2\pi} (\theta_2 + 2\pi)} (v_{\eta}) = \phi^{\tau(v_{\eta})} \circ \phi^{\frac{\tau(v_\eta)}{2\pi} \theta_2} (v_{\eta}) = \Theta(\eta) \cdot \phi^{\frac{\tau(v_\eta)}{2\pi} \theta_2} (v_{\eta}),
\]
and
\begin{align*}
\theta_1 - \frac{F'(\eta) - 2\pi}{2\pi} (\theta_2 + 2\pi)& = \theta_1 - \frac{F'(\eta) - 2\pi}{2\pi} \theta_2 - F'(\eta) + 2\pi\\ &= \theta_1 - \frac{F'(\eta) - 2\pi}{2\pi} \theta_2 - \Theta(\eta) \mod 2\pi,
\end{align*}
where the last step follows from Lemma~\ref{l:gen_fun}\,(v), imply the $2\pi$-periodicity in $\theta_2$, showing that $\Xi$ is well defined. The map $\Xi$ is smooth (resp.\ analytic) because so are $F$ and $\tau$.

We now check that $\Xi$ is bijective. Given $w\in T^1S^2 \setminus \Gamma$, we set \smash{$\eta:= \frac{C(w)}{r_{\max}}$}, where $C$ denotes Clairaut's first integral. Let $t$ be the unique number in $[0,\tau(v_\eta))$ such that $\phi^{-t}(w)$ belongs to $\Sigma$. Since $C\bigl(\phi^{-t}(w)\bigr) = C(w) = r_{\max} \eta = C(v_\eta)$, there exists a~unique $\theta\in S^1$ such that~${\phi^{-t}(w) = \theta\cdot v_\eta}$. The unique pre-image of $w$ with respect to $\Xi$ is the point
\[
\left(\theta + \frac{F'(\eta)-2\pi}{\tau(v_{\eta})} t, \frac{2\pi t}{\tau(v_\eta)}, \eta\right).
\]
Identity~\eqref{e:normal-form} is proven by the following two identities:
\begin{gather*}
\phi^t \Bigl( \Xi(\theta_1, \theta_2, \eta) \Bigr) = \left( \theta_1 - \frac{F'(\eta) - 2\pi}{2\pi} \theta_2 \right) \cdot \phi^{t + \frac{\tau(v_{\eta})}{2\pi} \theta_2} (v_\eta),
\\
\Xi \left( \theta_1 + \frac{F'(\eta) - 2\pi}{\tau(v_\eta)} t, \theta_2 + \frac{2\pi}{\tau(v_\eta)} t, \eta \right) \\
\qquad= \left( \theta_1 + \frac{F'(\eta) - 2\pi}{\tau(v_\eta)} t - \frac{F'(\eta) - 2\pi}{2\pi} \theta_2 - \frac{F'(\eta) - 2\pi}{2\pi} \frac{2\pi}{\tau(v_{\eta})} t \right) \cdot \phi^{\frac{\tau(v_{\eta})}{2\pi} \left( \theta_2 + \frac{2\pi}{\tau(v_{\eta})} t\right)} (v_\eta) \\
\qquad= \left( \theta_1 - \frac{F'(\eta) - 2\pi}{2\pi} \theta_2 \right) \cdot \phi^{t + \frac{\tau(v_{\eta})}{2\pi} \theta_2} (v_\eta).
\end{gather*}
There remains to prove~\eqref{e:conta}. Indeed, this will also prove that the smooth (resp.\ analytic) bijective map $\Xi$ is a~diffeomorphism because it pulls back volume forms to volume forms, as shown by the identity
\[
\Xi^* (\alpha \wedge {\rm d}\alpha) = \Xi^* \alpha \wedge {\rm d}\Xi^* \alpha = \frac{r_{\max}}{2\pi} \tau(v_\eta) {\rm d}\theta_1 \wedge {\rm d}\theta_2 \wedge {\rm d}\eta.
\]
Differentiation in $\theta_1$ and $\theta_2$ yields
\[
\partial_{\theta_1} \Xi = V \circ \Xi, \qquad
\partial_{\theta_2} \Xi = \frac{2\pi - F'(\eta)}{2\pi} V\circ \Xi + \frac{\tau(v_\eta)}{2\pi} R\circ \Xi,
\]
where we recall that $V$ and $R$ are the vector fields on $T^1 S^2$ generating the $S^1$-action and the geodesic flow. Using the identity $\alpha(V) = C = r_{\max} \eta$ on $\Sigma$ and Lemma~\ref{l:gen_fun}\,(iv), we obtain
\[
\alpha \bigl( \partial_{\theta_1} \Xi \bigr) = r_{\max} \eta, \qquad
\alpha \bigl( \partial_{\theta_2} \Xi \bigr) = \frac{2\pi - F'(\eta)}{2\pi} r_{\max} \eta + \frac{\tau(v_\eta)}{2\pi} = r_{\max} \left( \eta - \frac{F(\eta)}{2\pi} \right),
\]
and hence
\begin{gather}
\label{e:prov}
\Xi^* \alpha = r _{\max} \eta {\rm d}\theta_1+ r_{\max} \left( \eta - \frac{F(\eta)}{2\pi} \right) {\rm d}\theta_2 + U {\rm d}\eta,
\end{gather}
for some smooth function $U=U(\theta_1,\theta_2,\eta)$ yet to be determined. The one-form $\alpha$ is invariant under the $S^1$-action and the geodesic flow. Since
\[
\Xi_*(\partial_{\theta_1}) = V, \qquad \Xi_*\left( \frac{F'(\eta) - 2\pi}{\tau(v_\eta)}\partial_{\theta_1} + \frac{2\pi}{\tau(v_\eta)} \partial_{\theta_2} \right) = R,
\]
where the last identity follows from~\eqref{e:normal-form}, $\Xi^* \alpha$ is invariant under the flows generated by $\partial_{\theta_1}$ and~\smash{$ \frac{F'(\eta) - 2\pi}{\tau(v_\eta)}\partial_{\theta_1} + \frac{2\pi}{\tau(v_\eta)} \partial_{\theta_2}$}. Since these two vector fields are everywhere linearly independent, $\Xi^* \alpha$ is invariant under the full $S^1\times S^1$-action on its domain and hence the function $U$ depends only on $\eta$. We now evaluate at $\theta_1=\theta_2=0$, where $\Xi(0,0,\eta) = v_\eta$, and obtain
$
\partial_\eta \Xi (0,0,\eta) =\partial_\eta v_\eta$.
Since the vectors $v_\eta$ are all based at the same point, $\partial_\eta v_\eta$ is vertical and hence
\[
U(\eta) = \alpha ( \partial_\eta v_\eta ) = g ( v_\eta, {\rm d}\pi (\partial_\eta v_\eta) ) = 0.
\]
Therefore, $U$ is identically zero and~\eqref{e:conta} follows from~\eqref{e:prov}.
\end{proof}

\section{The marked length spectrum}
\label{s:mls}

Let $g$ be a~smooth (resp.\ analytic) metric on $S^2$ belonging to the set $\mathcal{G}$. Recall that an unoriented closed geodesic which is neither the equator nor a~meridian is said to have type $(p,q)$ if it winds~$p$ times around the $z$-axis while intersecting the equator $2q$ times. Here, $p$ is a~positive integer because the invariance of the Clairaut first integral implies that the coefficient of $\partial_{\theta}$ in the expression~\eqref{e:Reeb} for the vector field generating the geodesic flow never vanishes along orbits which do not project to meridians, and $q$ is a~positive integer because the assumption that $g$ has only one equator implies that alternative (ii) in Lemma~\ref{l:behaviour-geodesics} holds. The positive integers $p$ and~$q$ are necessarily coprime, see Lemma~\ref{l:rel-geo-per} below.

\begin{Remark}
Inspired by Poincar\'e's work~\cite{poi05}, in~\cite{ang05} Angenent studies the relative flat knot type of closed geodesics on a~Riemannian surface $M$. Fixing a~closed geodesic $\alpha \colon S^1 \rightarrow M$, he calls a~closed geodesic $\gamma\colon S^1 \rightarrow M$ an $(m,n)$-satellite of $\alpha$, where $m$, $n$ are coprime natural numbers, if $\gamma$ is smoothly isotopic through curves which have no self-tangencies and are transverse to $\alpha$ to the loop
$
t\mapsto \exp_{\alpha(nt)} ( \epsilon \sin (mt) \nu(nt) \bigr)$, $ t\in S^1$,
where $\epsilon\neq 0$ has small absolute value and $\nu$ denotes a~unit normal vector field along $\alpha$. It is easy to see that a~closed geodesic of type $(p,q)$ in our setting is a~$(q,p)$-satellite of the equator.
\end{Remark}

Any vector of the form $(\theta,0)\in \Sigma$ is a~fixed point of the first return map $\varphi$ and its orbit under the geodesic flow is the lift of a~meridian. The other periodic orbits of $\varphi$ are in one-to-one correspondence with the oriented closed geodesics other than the equator and the meridians, as shown by the next lemma.

\begin{Lemma}
\label{l:rel-geo-per}
Let $v=(\theta,\eta)\in \Sigma$ be a~periodic point of $\varphi$ with minimal period $q$ and $\eta\neq 0$. Then the unoriented geodesic $\gamma_v$ which is tangent to $v$ is closed, has type $(p,q)$, where
\begin{gather}
\label{e:rel-p-F}
p =
\begin{cases}
\displaystyle\frac{q F'(\eta)}{2\pi} & \mbox{if } \eta>0,\vspace{5pt}\\
\displaystyle 2q - \frac{q F'(\eta)}{2\pi} & \mbox{if } \eta<0,
\end{cases}
\end{gather}
is a~positive integer, and has length
$
\mathrm{length}(\gamma_v) = r_{\max} q \bigl( \eta F'(\eta) - F(\eta) \bigr)$.
Conversely, every unoriented closed geodesic which is neither the equator nor a~meridian is obtained in this way. Moreover, the positive integers $p$ and $q$ are coprime.
\end{Lemma}

\begin{proof}
Let $v=(\theta,\eta) \in \Sigma$, $\eta\neq 0$, be a~periodic point of $\varphi$ with minimal period $q$. Then the oriented geodesic $\hat{\gamma}_v$ emanating from $v$ is closed and meets the equator $2q$ times. Moreover, Lemma~\ref{l:gen_fun}\,(iii) implies that
$
\varphi^q(\theta,\eta) = (\theta + q F'(\eta), \eta)$,
so the fact that $(\theta,\eta)$ is a~fixed point of~$\varphi^q$ implies that
\begin{gather}
\label{e:interm}
q F'(\eta) = 2\pi p',
\end{gather}
for some $p'\in \Z$. Assume that $\eta>0$. By Lemma~\ref{l:gen_fun}\,(v), the integer $p'$ is positive, the geodesic~$\hat\gamma_v$ has winding number $p:=p'$ around the $z$-axis and hence the underlying unoriented geodesic~$\gamma_v$ is of type $(p,q)$. Moreover, the first alternative in~\eqref{e:rel-p-F} holds.

Now assume that $\eta<0$. By Lemma~\ref{l:gen_fun}\,(v), $\hat\gamma_v$ has a~negative winding number
\[
q \frac{\Theta(v)}{2\pi} = q \frac{F'(\eta)}{2\pi} - 2q = p' - 2q
\]
around the $z$-axis, and hence $\gamma_v$ is a~closed geodesic of type $(p,q)$ with $p=2q-p'>0$. Now, the second alternative in~\eqref{e:rel-p-F} holds.

In both cases, the length of $\gamma_v$ is, by Lemma~\ref{l:gen_fun}\,(iv),
\[
\mathrm{length}(\gamma_v) = q \tau(\theta,\eta) = r_{\max} q ( \eta F'(\eta) - F(\eta) ).
\]
Conversely, every unoriented closed geodesic which is neither the equator nor a~meridian is obtained in this way, because the corresponding two orbits of the geodesic flow meet $\Sigma$ at some~${(\theta,\eta)}$ with $\eta\neq 0$ which is necessarily a~periodic point of $\varphi$.

Finally, in the setting above assume that $p$ and $q$ are not coprime. Since $p'$ is either $p$ or~${2q-p}$, $p'$ and $q$ are not coprime either and we have $p'=r p_0$, $q=r q_0$ for some integers $r$, $p_0$, $q_0$ with $r>1$.
Then by~\eqref{e:interm}
\[
\varphi^{q_0}(\theta,\eta) = (\theta + q_0 F'(\eta), \eta) = ( \theta + 2\pi p_0, \eta) = (\theta,\eta),
\]
contradicting the assumption that $q$ is the minimal period of $(\theta,\eta)$. This shows that $p$ and $q$ are coprime.
\end{proof}

The Clairaut first integral $C\colon T^1 S^2 \rightarrow \R$ has image $[-r_{\max},r_{\max}]$ and its only critical points are its maximizers and minimizers, which form the subset $\Gamma$. Therefore, $T^1 S^2\setminus \Gamma$ is $S^1$-equivariantly diffeomorphic to $C^{-1}(0) \times (-r_{\max},r_{\max})$, where the
level set $C^{-1}(0)$ is the two-torus consisting of all $g$-unitary vectors which are tangent to meridians. Therefore, the first homology group $H_1\bigl(T^1 S^2\setminus \Gamma\bigr)$ with integer coefficients is $\Z^2$ and is generated by the homology classes~${S=S_g}$ and $M=M_g$ given respectively by an $S^1$-orbit and by an orbit of the geodesic flow in the meridians two-torus $C^{-1}(0)$. Indeed, this follows from the fact that these two simple curves in $C^{-1}(0)$ intersect transversally at exactly one point.

In the next lemma, we give a~homological interpretation of the type of a~closed geodesic.

\begin{Lemma}
\label{l:hom-inter}
Let $\gamma$ be an oriented closed geodesic which is neither the equator nor a~meridian and let $(p,q)$ be the type of the underlying unoriented closed geodesic. Then $\dot\gamma$, seen as an oriented closed loop in $T^1 S^2\setminus \Gamma$, has homology class
\[
[\dot\gamma]
=
\begin{cases}
(p-q) S + q M&\mbox{if } C(\dot\gamma)>0,\vspace{5pt}\\
(q-p) S + q M& \mbox{if } C(\dot\gamma)<0,
\end{cases}
\]
in $H_1\bigl(T^1 S^2\setminus \Gamma\bigr)$.
\end{Lemma}

\begin{proof}
Up to shifting time and employing the $S^1$-action, we may assume that $\dot\gamma(0)=v_\eta = (0,\eta) \in \Sigma$, for some $\eta\neq 0$. Write
\begin{gather}
\label{e:respbas}
[\dot\gamma] = a S + b M,
\end{gather}
with $a,b\in \Z$. Note that the isomorphism
\[
\Xi_* \colon\ H_1\bigl(S^1 \times S^1 \times (-1,1) \bigr) \rightarrow H_1 \bigl( T^1 S^2 \setminus \Gamma \bigr)
\]
induced by the diffeomorphism $\Xi$ of Proposition~\ref{p:normal-form} maps the homology class $\Theta_1$ defined by the oriented loop $\theta \mapsto (\theta,0,0)$ to $S$ and the homology class $\Theta_2$ defined by the oriented loop~${\theta \mapsto (0,\theta,0)}$ to $M$. Therefore, \eqref{e:respbas} yields
\begin{gather}
\label{e:dopoXi}
\bigl[\Xi^{-1} \circ \dot\gamma\bigr] = a \Theta_1 + b \Theta_2.
\end{gather}
On the other hand, since $\Xi(0,0,\eta) = v_\eta = \dot\gamma(0)$, \eqref{e:normal-form} implies
\[
\Xi^{-1} \circ \dot\gamma (t) = \left( \frac{F'(\eta) - 2\pi}{\tau(v_\eta)} t, \frac{2\pi}{\tau(v_\eta)} t, \eta \right).
\]
Since $\dot{\gamma}$ has minimal period $q \tau(v_\eta)$, comparison with~\eqref{e:dopoXi} gives us
\[
 a = q \left( \frac{F'(\eta)}{2\pi} - 1 \right), \qquad b= q.
\]
From~\eqref{e:rel-p-F}, we find
$
a = p-q $ if $\eta>0$, $ a = q-p$ if $\eta<0$,
and the conclusion follows from the fact that $C(\dot\gamma) = C(v_\eta)=r_{\max} \eta$ has the sign of $\eta$.
\end{proof}

Recall that the marked length spectrum of $g\in \mathcal{G}$ is given by the length $\ell$ of the equator and the set-valued function
\begin{gather*}
\mathcal{L}_g \colon \ \{ (p,q)\in \N\times \N \mid p, q \mbox{ coprime}\} \rightarrow \mathbbm{2}^{(0,+\infty)}, \qquad (p,q) \mapsto \left\{ \begin{array}{@{}c@{}} \mbox{lengths of all closed} \\ \mbox{geodesics of type } (p,q) \end{array} \right\},
\end{gather*}
and that two metrics are said to be isospectral if they have the same marked length spectrum.

\begin{Remark}
The above notion of marked spectrum does not take the length of the meridians into account. The homological interpretation of Lemma~\ref{l:hom-inter} suggests that meridians should be considered as closed geodesics of type $(1,1)$. By including the length of the meridians in the set $\mathcal{L}_g(1,1)$, one obtains a~slightly different definition of marked length spectrum, for which all the results of this paper would still hold. Actually, the pair of coefficients of the homology class~${[\dot\gamma]\in H_1\bigl(T^1 S^2 \setminus \Gamma\bigr)}$ with respect to the basis $\{S,M\}$ could be used as marker
instead of the type. The results of this paper would hold also with this different notion of marked length spectrum.
\end{Remark}

\begin{Proposition}
\label{p:coniso}
Let $h\colon T_{g_1}^1 S^2 \setminus \Gamma_{g_1} \rightarrow T_{g_2}^1 S^2 \setminus \Gamma_{g_2}$ be an $S^1$-equivariant orientation-preserving homeomorphism conjugating the geodesic flows of $g_1$ and $g_2$ and mapping lifts of meridians to lifts of meridians. Then $h$ maps $g_1$-geodesic flow orbits corresponding to an unoriented closed geodesic of type $(p,q)$ to $g_2$-geodesic flow orbits corresponding to an unoriented closed geodesic of type $(p,q)$. In particular, $\mathcal{L}_{g_1} = \mathcal{L}_{g_2}$.
\end{Proposition}

Note that in the above proposition we are not requiring $h$ to have a~continuous extension to $T^1 S^2$. An immediate corollary is that the existence of an $S^1$-equivariant and orientation-preserving homeomorphism from $T^1_{g_1} S^2$ to $T^1_{g_2} S^2$ conjugating the geodesic flows of $g_1$ and $g_2$ and mapping lifts of meridians to lifts of meridians implies that $g_1$ and $g_2$ are isospectral. Indeed, by $S^1$-equivariance such homeomorphism must map $\Gamma_{g_1}$ to $\Gamma_{g_2}$. In particular, this applies to any $S^1$-equivariant conjugacy which intertwines the contact forms $\alpha_1$ and $\alpha_2$, because the latter condition implies that the orientations and the torus of meridians are preserved.

\begin{proof}
Since $h$ maps lifts of meridians to lifts of meridians, it restricts to a~homeomorphism from $C_{g_1}^{-1}(0)$ to $C_{g_2}^{-1}(0)$. Since this restriction is equivariant with respect to the standard $S^1$\nobreakdash-actions and with respect to the $S^1$-actions defined by the two geodesic flows, the homomorphism
\[
h_* \colon\ H_1\bigl( T^1_{g_1} S^2 \setminus \Gamma_{g_1}\bigr) \cong H_1\bigl(C_{g_1}^{-1}(0)\bigr) \rightarrow H_1\bigl( T^1_{g_2} S^2 \setminus \Gamma_{g_2}\bigr) \cong H_1\bigl(C_{g_2}^{-1}(0)\bigr)
\]
maps $S_{g_1}$ to $S_{g_2}$ and $M_{g_1}$ to $M_{g_2}$. Indeed, this follows from the fact that, after identifying~${C_{g_1}^{-1}(0)}$ and $C_{g_2}^{-1}(0)$ to $S^1\times S^1$ by fixing starting points and using the $S^1$-actions, the fact that~${h\colon S^1 \times S^1 \rightarrow S^1 \times S^1}$ commutes with the $S^1 \times S^1$-action forces it to be a~translation.

Using also the assumption that $h$ is orientation preserving, we deduce that $h$ preserves the sign of Clairaut's first integral. The conclusion is now an immediate consequence of Lemma~\ref{l:hom-inter}.
\end{proof}

\section{Construction of the conjugacy away from the equator}
\label{s:away}

The next result is the main ingredient in the proof of Theorem~\ref{thmA}.

\begin{Proposition}
\label{p:equalF}
The marked length spectrum of $g\in \mathcal{G}$ determines the generating function $F\colon (-1,1) \rightarrow \R$ of the geodesic flow, and hence the first return time $\tau\colon \Sigma \rightarrow (0,+\infty)$, the first return map $\varphi\colon \Sigma \rightarrow \Sigma$ to the Birkhoff annulus of the equator, the normal form~\eqref{e:normal-form}, and the contact form~\eqref{e:conta}.
\end{Proposition}

\begin{proof}
From the properties of the Birkhoff annulus discussed in Section~\ref{s:birkhoff}, we deduce that there is a~one-to-one correspondence between oriented closed geodesics other than the equator and periodic orbits of the first return map $\varphi$ to the Birkhoff annulus $\Sigma$. More precisely, let~${v=(\theta,\eta) \in \Sigma}$ be a~periodic point of $\varphi$ with minimal period $q\in \N$. Then the geodesic $\gamma_v$ which is tangent to $v$ is closed and meets the equator $2q$ times. Moreover, Lemma~\ref{l:gen_fun}\,(iii) implies that
$
\varphi^q(\theta,\eta) = (\theta + q F'(\eta), \eta)$,
so the fact that $(\theta,\eta)$ is a~fixed point of $\varphi^q$ implies that
\begin{gather}
\label{e:F'}
F'(\eta) = 2\pi \frac{p}{q},
\end{gather}
for some $p\in \Z$. Assume that $\gamma_v$ is not a~meridian, i.e., that $\eta\neq 0$. Up to applying the reflection with respect to a~plane containing the $z$-axis, we may assume that $\eta\in (0,1)$. By Lemma~\ref{l:gen_fun}\,(v), the integer $p$ is positive, the geodesic $\gamma_v$ winds $p$ times around the $z$-axis and hence is of type~$(p,q)$. By Lemma~\ref{l:gen_fun}\,(iv), its length is
\begin{gather}
\label{e:length}
q \tau(\theta,\eta) = r_{\max} q \bigl( \eta F'(\eta) - F(\eta) \bigr).
\end{gather}
Conversely, every unoriented closed geodesic which is neither the equator nor a~meridian and is of type $(p,q)$ corresponds to a~periodic point $v=(\theta,\eta)$ of $\varphi$ with minimal period $q$ and with~${\eta\in (0,1)}$ and satisfying~\eqref{e:F'}. Together with the fact that the length of this geodesic is given by~\eqref{e:length}, this implies that
\begin{gather}
\label{e:mls}
\mathcal{L}_g(p,q) = \left\{ r_{\max} q ( \eta F'(\eta) - F(\eta) ) \mid \eta\in (0,1) \mbox{ such that } F'(\eta) = 2\pi \frac{p}{q} \right\},
\end{gather}
and hence
\[
\frac{1}{r_{\max} q} \mathcal{L}_g(p,q) = \left\{ \eta F'(\eta) - F(\eta) \mid \eta\in (0,1) \mbox{ such that } F'(\eta) = 2\pi \frac{p}{q} \right\}.
\]
The length $\ell$ of the equator determines $r_{\max}$ by the identity $\ell = 2\pi r_{\max}$, so this length together with the set-valued function $\mathcal{L}_g$ determine the following subset of $\R^2$:
\[
\mathcal{S} := \bigl\{ (F'(\eta), \eta F'(\eta) - F(\eta)) \mid \eta\in (0,1) \mbox{ such that } F'(\eta) \in 2\pi \Q \bigr\}.
\]

Given a~differentiable function $G\colon I \rightarrow \R$ on the interval $I\subset \R$, we set
\[
\mathcal{T}(G):= \{ (G'(\eta), \eta G'(\eta) - G(\eta)) \mid \eta \in I \}.
\]
Since the tangent line to the graph of $G$ at $(\eta,G(\eta))$ has the form
\[
\bigl\{ (x,y)\in \R^2 \mid y = G'(\eta) x - (\eta G'(\eta) - G(\eta)) \bigr\},
\]
the set $\mathcal{T}(G)$ can be identified with the set of tangent lines to the graph of $G$. With this identification, $\mathcal{S}$ is the set of lines which are tangent to the graph of $F|_{(0,1)}$ and whose slope is a~rational multiple of $2\pi$.

We claim that the following facts hold:
\begin{enumerate}[(a)]\itemsep=0pt
\item If $\mathcal{S}$ is bounded, then $F$ has a~$C^2$ extension to the interval $[-1,1]$, and the closure of $\mathcal{S}$ in~$\R^2$ is
$
\overline{\mathcal{S}} = \mathcal{T}(F|_{[0,1]})$.
\item If $\mathcal{S}$ is unbounded, then the closure of $\mathcal{S}$ in $\R^2$ is
$
\overline{\mathcal{S}} = \mathcal{T}(F|_{[0,1)})$.
\end{enumerate}

In the proof of (a) and (b), we can assume that $F'$ is not constant on $[0,1)$, which by Lemma~\ref{l:gen_fun}\,(ii) is equivalent to the fact that $F'$ is not constant on $(-1,1)$. Indeed, if $F'$ is constant then by Lemma~\ref{l:gen_fun}\,(i) it takes the value $F'(0)=2\pi$, so \smash{$F(\eta) = 2\pi \eta - \frac{2m}{r_{\max}}$} on $(-1,1)$ and hence \smash{$\mathcal{S}=\bigl\{\bigl(2\pi,\frac{2m}{r_{\max}}\bigr)\bigr\}$}. Therefore, $\mathcal{S}$ is bounded and statement (a) trivially holds.

Assume now that $F'$ is not constant. Then $F'([0,1))$ is contained in the closure of $F'((0,1))\cap 2\pi \Q$. Therefore, if $\mathcal{S}$ is bounded then $F'$ is bounded on $[0,1)$, which by Lemma~\ref{l:gen_fun}\,(vi) and (vii) implies that $F$ has a~$C^2$ extension to $[-1,1]$, and \smash{$\mathcal{T}(F|_{[0,1]})$} is a~compact set. If $\mathcal{S}$ is unbounded, then Lemma~\ref{l:gen_fun}\,(vii) implies that $F'(\eta)$ diverges to $+\infty$ for $\eta\rightarrow 1$, and hence \smash{$\mathcal{T}(F|_{[0,1)})$} is a~closed set. Let $\mathcal{T}_1$ be either \smash{$\mathcal{T}(F|_{[0,1]})$} or \smash{$\mathcal{T}(F|_{[0,1)})$}, depending on whether $\mathcal{S}$ is bounded or not. Then $\mathcal{T}_1$ is a~closed set containing $\mathcal{S}$, and hence $\overline{\mathcal{S}}\subset\mathcal{T}_1$. The opposite inclusion follows from the following elementary fact: if $\mathcal{Q}$ is a~dense subset of $\R$ and $G$ is a~continuously differentiable function on an interval $I$ which is not affine linear, then any tangent line to the graph of $G$ is the limit of a~sequence of tangent lines at points $(\eta_n,G(\eta_n))$ with $\eta_n$ in the interior of $I$ and~${G'(\eta_n)\in \mathcal{Q}}$. This concludes the proof of claims (a) and (b).

We now conclude the proof of the proposition. The marked length spectrum of $g$ determines the set $\mathcal{S}$ and hence its closure $\overline{\mathcal{S}}$. By claims (a) and (b), this closure is the set of tangent lines to the graph of either $F|_{[0,1]}$ or $F|_{[0,1)}$, depending on whether $\mathcal{S}$ is bounded or not.

The set of tangent lines to the graph of a~sufficiently regular real function on an interval $I$ uniquely determines the function. This fact is proven for $C^{1,1}$ functions on compact intervals by Richardson in~\cite[Theorem 4.8]{ric97}, using tools from geometric measure theory. In Appendix~\ref{s:appA} below, we provide an elementary proof for $C^2$ functions on arbitrary intervals. Therefore, the marked length spectrum determines the restriction of $F$ to either $[0,1]$ or $[0,1)$, and in either case it determines $F$ on $(-1,1)$ thanks to Lemma~\ref{l:gen_fun}\,(ii). Since by statements (iii) and (iv) in Lemma~\ref{l:gen_fun} the function $F$ determines $\tau$ and $\varphi$, and hence also~\eqref{e:normal-form} and~\eqref{e:conta}, this concludes the proof.
\end{proof}

\begin{Remark}\label{r:converse}
Conversely, identity~\eqref{e:mls} and Lemma~\ref{l:gen_fun}\,(i),~(iii) imply that $m$, $r_{\max}$, and the first return map $\varphi$ together determine the marked length spectrum of $g$.
\end{Remark}

We can now prove the first part of Theorem~\ref{thmA}, asserting that if two smooth (resp.\ analytic) metrics $g_1$ and $g_2$ in $\mathcal{G}$ are isospectral then the restrictions of the geodesic flows $\phi^t_{g_1}$ and $\phi^t_{g_2}$ to the complements of $\Gamma_{g_1}$ and $\Gamma_{g_2}$ are conjugated by a~smooth (resp.\ analytic) $S^1$-equivariant diffeomorphism $h$ such that $h^* \alpha_{g_2} = \alpha_{g_1}$.

By Proposition~\ref{p:equalF}, the isospectrality of $g_1$ and $g_2$ implies that their geodesic flows have the same generating function $F$. If
\[
\Xi_{g_1} \colon\ S^1 \times S^1 \times (-1,1) \rightarrow T^1_{g_1} S^2 \setminus \Gamma_{g_1}, \qquad \Xi_{g_2} \colon\ S^1 \times S^1 \times (-1,1) \rightarrow T^1_{g_2} S^2 \setminus \Gamma_{g_2}
\]
are the smooth (resp.\ analytic) diffeomorphisms of Proposition~\ref{p:normal-form}, \eqref{e:normal-form} implies that the pullbacks of the geodesic flows of $g_1$ and $g_2$ by $\Xi_{g_1}$ and $\Xi_{g_2}$ are the same flow on $S^1 \times S^1 \times (-1,1)$. Note also that $\Xi_{g_1}^* \alpha_{g_1} =
\Xi_{g_2}^* \alpha_{g_2}$ because of~\eqref{e:conta}. The map $h:= \Xi_{g_2} \circ \Xi_{g_1}^{-1}$ is the required $S^1$-equivariant smooth (resp.\ analytic) contactomorphism conjugating the geodesic flows $\phi^t_{g_1}$ and~$\phi^t_{g_2}$.

\begin{Remark}
\label{r:dyn-cont}
Equivalently, the map $h$ defined above can be seen as dynamical continuation of the identification between the Birkhoff annuli induced by their common coordinates $(\theta,\eta)\in S^1 \times (-1,1)$ (see~\eqref{e:cosi}): if $h_0\colon \Sigma_{g_1} \rightarrow \Sigma_{g_2}$ is such identification, then given $v$ in $T_{g_1}^1 S^2 \setminus \Gamma_{g_1}$ we have
\[
h(v) = \phi_{g_2}^{-t}\bigl( h_0 \bigl(\phi_{g_1}^t(v) \bigr) \bigr) \in T_{g_2}^1 S^2 \setminus \Gamma_{g_2},
\]
where $t$ is any real number such that $\phi_{g_1}^t(v) \in \Sigma_{g_1}$. This follows from the identity $\Xi(\theta_1,0,\eta) = \theta_1 \cdot v_{\eta} = (\theta_1,\eta)\in \Sigma$ and~\eqref{e:normal-form}.
\end{Remark}

\begin{Remark}
Geodesic flows satisfy $\phi^{-t}(-v) = - \phi^t(v)$. The conjugacy $h$, however, is in general not equivariant with respect to the flip map $v\mapsto -v$. To see this, consider two isospectral metrics~$g_1$, $g_2$ and a~vector $v\in T^1_{g_1} S^2\setminus \Gamma_{g_1}$ which is tangent to a~parallel. Then there exist~${t\in \R}$ and vectors $(\theta,\eta)$, $(\theta',-\eta)$ in $\Sigma$ such that $v = \phi^t_{g_1}(\theta,\eta)$ and $-v = \phi^t_{g_1}(\theta',-\eta)$. Then $h(v) = \phi^t_{g_2}(\theta,\eta)$ and
\[
h(-v) = \phi^t_{g_2}(\theta',-\eta) = R_{\theta'} \circ \phi^t_{g_2} (\theta',\eta) = R_{\theta'} \bigl( (\theta' - \theta) \cdot \phi^t_{g_2} (\theta,\eta) \bigr) = R_{\theta'} \bigl( (\theta' - \theta) \cdot h(v) \bigr),
\]
where $R_{\theta'}\colon T^1_{g_2} S^2 \rightarrow T^1_{g_2} S^2$ is the tangent lift of the reflection about the meridian with angle~$\theta'$. By applying the function $\beta\colon T^1_{g_2} \bigl(S^2 \setminus \{S,N\}\bigr) \rightarrow S^1$, giving us the angle with the positive direction of parallels, we obtain
$
\beta(h(-v)) = \pi - \beta(h(v))$.
On the other hand, if $h(-v)=-h(v)$ then $\beta(h(-v)) = \pi + \beta(h(v))$ and hence the above identity implies that $\beta(h(v)) \in \{0,\pi\}$, meaning that $h(v)$ must be tangent to a~parallel. However, $h$ need not map vectors tangent to parallels to vectors tangent to parallels. Indeed, if this is the case then the first return time to the union~${\Sigma\cup \widehat\Sigma}$ of the two Birkhoff annuli associated to the equator would be the same for both geodesic flows and the argument of Section~\ref{s:abel} would force the metrics $g_1$ and $g_2$ to coincide. We conclude that the conjugacy $h$ constructed here is flip equivariant only in the trivial case~${g_1=g_2}$.
\end{Remark}

\begin{Remark}
Let $\mathcal{Q}$ be a~dense subset of $\Q \cap (0,+\infty)$. Then the proof of Proposition~\ref{p:equalF} shows that the data given by the length of the equator and the restriction of $\mathcal{L}_g$ to the set of pairs of coprime natural numbers $(p,q)$ such that $\frac{p}{q}\in \mathcal{Q}$ is still sufficient to recover the generating function $F$. Therefore, Theorem~\ref{thmA} holds also with this restricted version of the marked length spectrum.
\end{Remark}

\section{Regular extension of the conjugacy at the equator}
\label{s:regular}

We now prove statement (i) of Theorem~\ref{thmA}.
Note that in the coordinates $(\sigma,\theta)$ on $S^2 \setminus \{N,S\}$, the curvature of a~metric $g$ in $\mathcal{G}$ with profile function $r$ is
\begin{gather}
\label{e:curvature}
K(\sigma,\theta) = - \frac{r''(\sigma)}{r(\sigma)}.
\end{gather}
Let $g_1$ and $g_2$ be isospectral smooth (resp.\ analytic) metrics in $\mathcal{G}$. We now assume that the curvature of $g_1$ at its equator is positive. By Proposition~\ref{p:equalF}, $g_1$ and $g_2$ have the same generating function $F$, and hence Lemma~\ref{l:gen_fun}\,(vi)--(vii) implies that the curvature of $g_2$ at its equator is positive as well. This proves the first assertion in statement~(i) of Theorem~\ref{thmA}.

There remains to prove that in this case there exists an $S^1$-equivariant smooth (resp.\ analytic) diffeomorphism $h\colon T^1_{g_1} S^2 \rightarrow T^1_{g_2} S^2$ such that $h^* \alpha_{g_2} = \alpha_{g_1}$. Here, $\alpha_{g_1}$ and $\alpha_{g_2}$ are the $S^1$-invariant Hilbert contact forms on $T^1_{g_1} S^2$ and $T^1_{g_2} S^2$. Indeed, such a~diffeomorphism intertwines the Reeb vector fields of $\alpha_{g_1}$ and $\alpha_{g_2}$ and hence conjugates the two geodesic flows.

Using the fact that the first return time to the Birkhoff annulus is bounded when the curvature at the equator is positive (see Lemma~\ref{l:extension}), the conjugacy of the last section can be shown to have a~continuous extension to $T^1_{g_1} S^2$, but this extension is not necessarily smooth. See Example~\ref{e:not-smooth} at the end of this section. In order to construct a~different smooth (resp.\ analytic) conjugacy, we shall use a~general statement about $S^1$-invariant contact forms which is discussed in Appendix~\ref{s:appB}.

We need to introduce some further notation. Recall that $V$ denotes the generator of the free $S^1$-action on $T^1 S^2$ and
$
\Pi \colon T^1 S^2 \rightarrow W
$
the corresponding quotient projection, where the quotient space $W$ is diffeomorphic to $S^2$. The Clairaut first integral $C = \alpha(V)$ is $S^1$-invariant and hence has the form $C=H\circ \Pi$, where $H$ is a~smooth (resp.\ analytic) real function on $W$. We refer to $H$ as the \textit{Hamiltonian} induced by the $S^1$-invariant contact form $\alpha_g$. The reason for this terminology is the following: by Proposition~\ref{p:symplectic_quotient}, $W$ admits a~unique symplectic form $\omega$ such that
\begin{align}
\label{e:omega}
\Pi^* \omega = \imath_V \alpha\wedge {\rm d}\alpha = r(\sigma) {\rm d}\beta \wedge {\rm d}\sigma,
\end{align}
(see~\eqref{e:volume}), and the $S^1$-invariant Reeb vector field $R$ generating the geodesic flow projects by~$\Pi$ to the Hamiltonian vector field $X_H$ of the function $H$ on the symplectic manifold $(W,\omega)$, where~$X_H$ is defined by the identity $\imath_{X_H} \omega = {\rm d}H$.

The critical set of $C$ is the set $\Gamma$ consisting of the two orbits of the $S^1$-action which are also orbits of the geodesic flow, corresponding to the equator. Using geodesic polar coordinates, we see that
$C=C(\sigma,\theta,\beta) = r(\sigma) \cos \beta$ achieves its maximum $r_{\max}$ on the orbit
\[
\Gamma^+ := \bigl\{ (\sigma,\theta,\beta) \mid \sigma=\sigma_{\max},\, \theta\in S^1, \, \beta=0 \bigr\},
\]
and its minimum $-r_{\max}$ on the orbit
\[
\Gamma^- := \bigl\{ (\sigma,\theta,\beta) \mid \sigma=\sigma_{\max},\, \theta\in S^1,\, \beta=\pi\bigr\}.
\]
Seen as points in $W$, $\Gamma^+$ and $\Gamma^-$ are the unique critical points of $H$. The geodesic polar coordinates induce global coordinates $(\sigma,\beta) \in (0,m) \times S^1$ on $W$ minus the two points corresponding to the $S^1$-orbits given by the fibers of $T^1 S^2$ at the poles. In these coordinates,
\begin{align}
\label{e:HH}
H(\sigma,\beta) = r(\sigma) \cos \beta, \qquad \Gamma^+ = (\sigma_{\max},0), \qquad \Gamma^- = (\sigma_{\max},\pi).
\end{align}
Note that the critical points $\Gamma^+$ and $\Gamma^-$ of $H$ are non-degenerate if and only if ${r''(\sigma_{\max}) \neq 0}$, corresponding to the case in which $g$ has positive curvature along the equator, see~\eqref{e:curvature}. Therefore, the latter assumption guarantees that the Hamiltonian $H$ is a~perfect Morse function.

Note that all orbits of $X_H$ are periodic: $\Gamma^+$ and $\Gamma^-$ are equilibrium points, and the circles~$H^{-1}(c)$ for $c\in J:= (-r_{\max},r_{\max})$ are non-constant periodic orbits. Given $c\in J$, we denote by $T(c)$ the period of the orbit $H^{-1}(c)$. Said otherwise: $T(c)$ is the minimal positive time the geodesic flow needs in order to bring a~vector $v\in T^1 S^2$ with $C(v)=c$ back to its orbit by the $S^1$-action. We refer to $T$ as the \textit{period function} induced by the $S^1$-invariant contact form $\alpha_g$.

The main result of Appendix~\ref{s:appB} asserts the following: two $S^1$-invariant smooth (resp.\ analytic) contact forms on the total space of a principal $S^1$-bundle over $S^2$ whose induced Hamiltonians on $S^2$ are perfect Morse functions are one the pull-back of the other by a smooth (resp.\ analytic) diffeomorphism provided that the two Hamiltonians have the same image and the two period functions coincide.
 See Theorem~\ref{mainappB} and the discussion preceding it for the precise statement.

Applying this theorem to the geodesic setting considered here, it is enough to show that the marked length spectrum of a~metric $g$ in $\mathcal{G}$ determines the interval $J=(\min H,\max H)$ and the period function $T\colon J \rightarrow (0,+\infty)$ introduced above.

Since $\min H=-r_{\max}$ and $\max H=r_{\max}$, the interval $J$ is determined by the length $\ell=2\pi r_{\max}$ of the equator. The image by $\Pi$ of the Birkhoff annulus $\Sigma$ is the smooth curve in $W$ given by the points $(\sigma,\beta)$ with $\sigma=\sigma_{\max}$ and $\beta\in (0,\pi)$. This curve joins $\Gamma^+$ and $\Gamma^-$ and meets each circle $H^{-1}(c)$ with~$c\in J$ exactly once. Therefore, for every $v\in \Sigma$ we have $\tau(v) = T(C(v))$, where~$\tau$ denotes the first return time to the Birkhoff annulus. By Proposition~\ref{p:equalF}, $\tau$ is determined by the marked length spectrum of $g$, and hence so is
 the period function $T\colon J \rightarrow (0,+\infty)$. So $J$ and $T$ are determined by the marked length spectrum of $g$, concluding the proof of statement~(i) in Theorem~\ref{thmA}.

\begin{Example}
\label{e:not-smooth}
Let $r\colon [0,m] \rightarrow \R$ be the profile function of a~smooth metric $g_1$ in $\mathcal{G}$ with positive curvature along the equator, and let $g_2$ be the metric with profile function ${\sigma\mapsto r(m-\sigma)}$. The reflection $(\sigma,\theta) \mapsto (m-\sigma,\theta)$ is an isometry between $g_1$ and $g_2$ and induces a~smooth conjugacy between the geodesic flows of the two metrics. However, the smooth conjugacy $h\colon T_{g_1}^1 S^2 \setminus \Gamma_{g_1} \rightarrow T_{g_2}^1 S^2 \setminus \Gamma_{g_2}$ which is obtained by dynamical continuation from the identification~${h_0\colon \Sigma_{g_1} \rightarrow \Sigma_{g_2}}$ of the Birkhoff annuli does not necessarily admit a~$C^2$ extension to the whole~$T^1_{g_1} S^2$.

Indeed, assume for the sake of simplicity that $r$ achieves its maximum at $\frac{m}{2}$, where $r''\left(\frac{m}{2}\right)<0$ because we are assuming that the curvature along the equator is positive. We claim that if~${r'''\left(\frac{m}{2}\right)\neq 0}$, then the conjugacy $h$ does not have a~$C^2$ extension over $\Gamma_{g_1}$.

As seen above, the geodesic flows $\phi^t_{g_1}$ and $\phi^t_{g_2}$ project to Hamiltonian flows on the quotients~$W_1$ and $W_2$ by the $S^1$-action. We can identify $W_1$ and $W_2$ with the same two-sphere $W$ by the use of common coordinates $(\sigma,\beta)$ on them, and with this choice the Hamiltonian and symplectic forms induced by $g_1$ and $g_2$ read
\begin{alignat*}{3}
& H_1(\sigma,\beta) = r(\sigma) \cos \beta, \qquad && H_2(\sigma,\beta) = r(m-\sigma) \cos \beta, & \\
& \omega_1 = r(\sigma) {\rm d}\beta \wedge {\rm d}\sigma, \qquad&& \omega_2 = r(m-\sigma) {\rm d}\beta\wedge {\rm d}\sigma.&
\end{alignat*}
Both $\Gamma_{g_1}$ and $\Gamma_{g_2}$ project to the set consisting of the two common critical points $\Gamma^+=\bigl(\frac{m}{2},0\bigr)$ and $\Gamma^-=\bigl(\frac{m}{2},\pi\bigr)$ of the Hamiltonians $H_1$ and $H_2$. Denote by $\psi^t_1$ and $\psi^t_2$ the Hamiltonian flows of $(H_1,\omega_1)$ and $(H_2,\omega_2)$. Then the smooth conjugacy $h\colon T_{g_1}^1 S^2 \setminus \Gamma_{g_1} \rightarrow T_{g_2}^1 S^2 \setminus \Gamma_{g_2}$ induces a~smooth conjugacy
\smash{$
\widehat{h} \colon W \setminus \bigl\{\Gamma^+,\Gamma^-\bigr\} \rightarrow W \setminus \bigl\{\Gamma^+,\Gamma^-\bigr\}
$}
from $\psi_1^t$ to $\psi_2^t$, which is the dynamical continuation of the identity map on the projection
\[
\widehat{\Sigma} := \left\{(\sigma,\beta)\in W \mid \sigma= { \frac{m}{2}},\, \beta\in (0,\pi) \right\}
\]
of the Birkhoff annuli. The map \smash{$\widehat{h}$} has an obvious continuous extension to $W$ which fixes $\Gamma^+$ and $\Gamma^-$. If the $S^1$-equivariant diffeomorphism $h$ has a~$C^2$ extension over $\Gamma_{g_1}$, then \smash{$\widehat{h}$} has a~$C^2$ extension over the two critical points $\Gamma^+$ and $\Gamma^-$. We shall instead see that \smash{$\widehat{h}$} does not have a~$C^2$ extension over $\Gamma^+$ (nor, by symmetry, over $\Gamma^-$).

The smooth involution $\jmath\colon W \rightarrow W$ mapping $(\sigma,\beta)$ to $(m-\sigma,-\beta)$ satisfies $\jmath^* H_2 = H_1$ and~${\jmath^* \omega_2 = \omega_1}$, and hence
$
\psi_2^t = \jmath \circ \psi_1^t \circ \jmath $, $\forall t\in \R$.
Given $v=\left(\frac{m}{2},\theta,\beta\right)\in \Sigma_{g_1}$, $\beta\in (0,\pi)$, let $\gamma_v\colon \R \rightarrow S^2$ be the geodesic of the metric $g_1$ with $\dot\gamma_v(0)=v$.
We denote by $\tau_N(\beta)$ the smallest positive number $t$ such that the point $\gamma_v(t)$ is again on the equator. The geodesic arc~${\gamma_v([0,\tau_N(\beta)])}$ is a~maximal geodesic arc contained in the northern hemisphere. Similarly, we denote by $\tau_S(\beta)$ the smallest positive number $t$ such that $\gamma_v(-t)$ is on the equator. The geodesic arc~${\gamma_v([-\tau_S(\beta),0])}$ is a~maximal geodesic arc contained in the southern hemisphere. For $\beta\rightarrow 0$, both $\tau_N(\beta)$ and $\tau_S(\beta)$ converge to the first conjugate instant along the equator (see Lemma~\ref{l:extension}). The assumption that $r'''\bigl(\frac{m}{2}\bigr)$ is not zero implies that $\tau_N(\beta)$ and $\tau_S(\beta)$ differ at first order for~${\beta\rightarrow 0}$. More precisely, there exists a~non-zero real number $a$ such that
\begin{gather}
\label{e:first-order}
\lim_{\beta \rightarrow 0} \bigl( \tau_N'(\beta) - \tau_S'(\beta) \bigr) = a.
\end{gather}
This can be proven by the first formula of Lemma~\ref{l:calcoli} in the next section.

Let $\beta\in (0,\pi)$. Since $\bigl(\frac{m}{2},\beta\bigr)$ belongs to $\widehat{\Sigma}$, we have $\widehat{h}\bigl(\frac{m}{2},\beta\bigr) = \bigl(\frac{m}{2},\beta\bigr)$ and hence
\begin{gather}
\label{e:contra}
\partial_{\beta\beta} \widehat{h} \bigl( { \frac{m}{2} }, \beta\bigr) = 0, \qquad \forall \beta\in (0,\pi).
\end{gather}
Now fix some $\beta\in (-\pi,0)$. By the definition of $\tau_S$ we have
\smash{$
\psi_1^{\tau_S(-\beta)} \bigl( { \frac{m}{2} }, \beta\bigr) = \bigl( { \frac{m}{2} }, -\beta\bigr) \in \widehat{\Sigma}$},
Using also the identity
\smash{$
\psi_1^{\tau_N(-\beta)} \bigl( { \frac{m}{2} }, - \beta\bigr) = \bigl( { \frac{m}{2} }, \beta\bigr)$},
we find
\begin{align*}
\widehat{h} \left( { \frac{m}{2} }, \beta\right) &= \psi_2^{-\tau_S(-\beta)} \left( { \frac{m}{2} }, -\beta\right) = \jmath \circ \psi_1^{-\tau_S(-\beta)} \left( { \frac{m}{2} }, \beta\right) \\ &= \jmath \circ \psi_1^{\tau_N(-\beta) -\tau_S(-\beta)} \circ \psi_1^{-\tau_N(-\beta)} \left( { \frac{m}{2} }, \beta\right) \\ &= \jmath \circ \psi_1^{\tau_N(-\beta) -\tau_S(-\beta)} \left( { \frac{m}{2} }, -\beta\right) = \psi_2^{\tau_N(-\beta) -\tau_S(-\beta)} \left( { \frac{m}{2} }, \beta\right).
\end{align*} A~first differentiation in $\beta$ gives us the identity
\[
\partial_{\beta} \widehat{h} \left( { \frac{m}{2} }, \beta\right) = -\bigl( \tau_N'(-\beta) - \tau_S'(-\beta) \bigr) X_2 \left( \widehat{h} \left( { \frac{m}{2} }, \beta \right) \right) + \partial_{\beta} \Psi \left( { \frac{m}{2} }, \beta\right),
\]
where \smash{$\Psi=\psi_2^{\tau_N(-\beta) - \tau_S(-\beta)}$} and
\[
X_2(\sigma,\beta) = \sin \beta \partial_{\sigma} - \frac{r'(m-\sigma)}{r(m-\sigma)} \cos \beta \partial_{\beta}
\]
denotes the Hamiltonian vector field of $(H_2,\omega_2)$. By~\eqref{e:first-order} and by the identity ${\rm d}X_2\left(\frac{m}{2},0\right) \partial_{\beta} = \partial_{\sigma}$, a~second differentiation in $\beta$ and a~limit for $\beta\uparrow 0$ produce
\[
\lim_{\beta\uparrow 0} \partial_{\beta\beta} \widehat{h} \left( { \frac{m}{2} }, \beta\right) = - 2 a \partial_{\sigma}.\]
Comparing the above limit with~\eqref{e:contra}, we conclude that \smash{$\widehat{h}$} does not have a~$C^2$ extension at $\Gamma^+$.
\end{Example}

\section{Continuous extension of the conjugacy at the equator}
\label{s:continuous}

The aim of this section is to prove the following result.

\begin{Proposition}
\label{p:non-neg-curv}
Assume that the metrics $g_1,g_2\in \mathcal{G}$ are isospectral and have non-negative curvature near their equators. Then the $S^1$-equivariant contactomorphism $h\colon T^1_{g_1} S^2 \setminus \Gamma_{g_1} \rightarrow T^1_{g_2} S^2 \setminus \Gamma_{g_2}$ which is constructed in Section~{\rm\ref{s:away}} extends to a~homeomorphism from $T^1_{g_1} S^2$ to $T^1_{g_2} S^2$ which conjugates the two geodesic flows.
\end{Proposition}

Assume that the curvature of $g\in \mathcal{G}$ does not vanish to infinite order at the equator $\{\sigma=\sigma_{\max}\}$. By~\eqref{e:curvature}, the function $r''$ does not vanish to infinite order at $\sigma_{\max}$. Using also the fact that $r$ has a~maximum at $\sigma_{\max}$, we deduce that the first natural number $k\geq 1$ such that the $k$-th derivative of $r$ at $\sigma_{\max}$ is not zero is even, and $r^{(k)}(\sigma_{\max}) < 0$. This implies that $r''$ is non-positive in a~neighborhood of $\sigma_{\max}$, and hence the curvature of $g$ is non-negative near its equator. Therefore, Proposition~\ref{p:non-neg-curv} implies statement (ii) of Theorem~\ref{thmA}, provided that we show that the condition of having curvature which does not vanish to infinite order at the equator is preserved by the isospectrality equivalence relation. The latter assertion follows from the more general Proposition~\ref{p:order}, which is presented at the end of the next section.

Let $g$ be a~metric in $\mathcal{G}$ with profile function $r\colon [0,m]\rightarrow [0,+\infty)$ achieving its maximum $r_{\max}$ at $\sigma_{\max}$. Then $r$ is strictly increasing on $[0,\sigma_{\max}]$ and strictly decreasing on $[\sigma_{\max},m]$, and we can consider the inverse functions
\begin{gather}
\sigma_S := r|_{[0,\sigma_{\max}]}^{-1} \colon\ [0,r_{\max}] \rightarrow [0,\sigma_{\max}],\qquad
\sigma_N := r|_{[\sigma_{\max},m]}^{-1} \colon\ [0,r_{\max}] \rightarrow [\sigma_{\max},m].\label{e:inverses}
\end{gather}
Let $v$ be a~vector in the Birkhoff annulus $\Sigma$ forming an angle $\beta_0 \in (0,\pi) \setminus \{\frac{\pi}{2}\}$ with the positive direction of the equator. In geodesic polar coordinates, we have $v=(\sigma_{\max},\theta_0,\beta_0)$, and we denote the orbit of $v$ under the geodesic flow by
$
\phi^t(v) = \bigl(\sigma(t),\theta(t),\beta(t) \bigr)$,
and the geodesic starting at $v$ by $\gamma(t) = \pi \circ \phi^t(v)$. Let $\tau_1=\tau_1(\beta_0)$ be the first positive instant $t$ at which $\gamma$ is tangent to a~parallel, i.e., the first $t>0$ such that $\cos \beta(t) = \pm 1$. In particular, the geodesic arc $\gamma([0,\tau_1])$ is contained in the northern hemisphere, and we obtain the following expressions for its length and the variation of $\theta$ along it.

\begin{Lemma}\label{l:calcoli}
For every $t_0\in [0,\tau_1]$, we have
\begin{gather*}
\mathrm{length}(\gamma|_{[0,t_0]}) = t_0 = - \frac{r_{\max}}{2} \int_{\frac{\cos^2 \beta_0}{\cos^2 \beta (t_0)}}^1 \frac{\sigma_N'\bigl(r_{\max} \sqrt{u}\bigr)}{\sqrt{u-\cos^2 \beta_0}} {\rm d}u, \\
\theta(t_0) - \theta(0) = - \frac{\cos \beta_0}{2} \int_{\frac{\cos^2 \beta_0}{\cos^2 \beta (t_0)}}^1 \frac{1}{u} \frac{\sigma_N'\bigl(r_{\max} \sqrt{u}\bigr)}{\sqrt{u-\cos^2 \beta_0}} {\rm d}u.
\end{gather*}
\end{Lemma}

\begin{proof}
The conservation of the Clairaut first integral gives us the identity
\begin{gather}
\label{e:cla}
r(\sigma(t)) \cos \beta(t) = r_{\max} \cos \beta_0, \qquad \forall t\in \R.
\end{gather}
For every $t\in [0,\tau_1]$, the quantity $\sin \beta(t)$ is non-negative, and by~\eqref{e:Reeb} and~\eqref{e:cla} we have
\[
\sigma'(t) = \sin \beta(t) = \sqrt{1 - \cos^2 \beta(t)} = \sqrt{ 1 - \frac{r_{\max}^2 \cos^2 \beta_0}{r(\sigma(t))^2} }.
\]
Therefore, the change of variable $s=\sigma(t)$ gives us
\begin{gather}
\label{e:tempo}
 \mathrm{length}(\gamma|_{[0,t_0]}) = t_0 = \int_0^{t_0}{\rm d}t = \int_{\sigma_{\max}}^{\sigma(t_0)} \left( 1 - \frac{r_{\max}^2 \cos^2 \beta_0}{r(s)^2} \right)^{-\frac{1}{2}}{\rm d}s.
\end{gather}
 A~second change of variable \smash{$u = \frac{r(s)^2}{r_{\max}^2}$}, which is equivalent to $s = \sigma_N\bigl(r_{\max} \sqrt{u}\bigr)$ because $\sigma(t_0)\in [\sigma_{\max},m]$, yields
\begin{align*}
 \mathrm{length}(\gamma|_{[0,t_0]}) &= \int_{1}^{\frac{\cos^2 \beta_0}{\cos^2 \beta (t_0)}} \left( 1 - \frac{\cos^2 \beta_0}{u} \right)^{-\frac{1}{2}} r_{\max} \frac{\sigma_N'\bigl(r_{\max} \sqrt{u}\bigr)}{2\sqrt{u}} {\rm d}u, \\ &=
 - \frac{r_{\max}}{2} \int_{\frac{\cos^2 \beta_0}{\cos^2 \beta (t_0)}}^1 \frac{\sigma_N'\bigl(r_{\max} \sqrt{u}\bigr)}{\sqrt{u-\cos^2 \beta_0}} {\rm d}u,
 \end{align*}
 proving the first identity. Thanks to~\eqref{e:Reeb} and~\eqref{e:cla}, we have
 \begin{align*}
 \theta(t_0) - \theta(0) &= \int_0^{t_0} \theta'(t){\rm d}t = \int_0^{t_0} \frac{\cos \beta(t)}{r(\sigma(t))}{\rm d}t = \int_0^{t_0} \frac{r(\sigma(t))\cos \beta(t)}{r(\sigma(t))^2}{\rm d}t \\ &= r_{\max} \cos \beta_0 \int_0^{t_0} \frac{1}{r(\sigma(t))^2}{\rm d}t.
 \end{align*}
 Using again the change of variable $s=\sigma(t)$, we find
 \[
 \int_0^{t_0} \frac{1}{r(\sigma(t))^2}{\rm d}t = \int_{\sigma_{\max}}^{\sigma(t_0)} \frac{1}{r(s)^2} \left( 1 - \frac{r_{\max}^2 \cos^2 \beta_0}{r(s)^2} \right)^{-\frac{1}{2}}{\rm d}s.
\]
By the change of variable \smash{$u = \frac{r(s)^2}{r_{\max}^2}$}, the above integral equals
\begin{gather*}
\int_1^{\frac{\cos^2 \beta_0}{\cos^2 \beta(t_0)}} \frac{1}{r_{\max}^2 u} \left( 1 - \frac{\cos^2 \beta_0}{u} \right)^{-\frac{1}{2}} r_{\max} \frac{\sigma_N'\bigl(r_{\max} \sqrt{u}\bigr)}{2\sqrt{u}} {\rm d}u \\
\qquad=
 - \frac{1}{2 r_{\max}} \int_{\frac{\cos^2 \beta_0}{\cos^2 \beta (t_0)}}^1 \frac{1}{u} \frac{\sigma_N'\bigl(r_{\max} \sqrt{u}\bigr)}{\sqrt{u-\cos^2 \beta_0}} {\rm d}u,
 \end{gather*}
 and we conclude that
 \[
 \theta(t_0) - \theta(0) = - \frac{\cos \beta_0}{2} \int_{\frac{\cos^2 \beta_0}{\cos^2 \beta (t_0)}}^1 \frac{1}{u} \frac{\sigma_N'\bigl(r_{\max} \sqrt{u}\bigr)}{\sqrt{u-\cos^2 \beta_0}} {\rm d}u.\tag*{\qed}
\]\renewcommand{\qed}{}
\end{proof}

When $\beta_0$ converges to $0$, the vector $v=(\sigma_{\max},\theta_0,\beta_0)\in \Sigma$ tends to the vector $w=(\sigma_{\max},\theta_0,0)$, which is tangent to the equator and whose orbit by the geodesic flow is
$
\phi^t(w) = \smash{\bigl( \sigma_{\max}, \theta_0 + { \frac{1}{r_{\max}}} t, 0 \bigr)}$.
The next lemma is an immediate consequence of the continuity of the geodesic flow if the first return time $\tau$ to the Birkhoff annulus $\Sigma$ is bounded, i.e., if the curvature of $g$ at the equator is positive. In general, it gives us a~stability property over possibly long time intervals of the orbit corresponding to the equator.

\begin{Lemma}
\label{l:stability}
If the curvature of $g$ near the equator is non-negative, then
\[
\lim_{\beta_0\downarrow 0} \max_{t\in [0,\tau(v)]} \left| \theta(t) - \theta(0) - { \frac{\cos \beta_0}{r_{\max}}} t \right| = 0.
\]
\end{Lemma}

\begin{proof}
Assume first that $t$ is in the interval $[0,\tau_1(\beta_0)]$, where $\tau_1(\beta_0)$ is the first positive $t$ such that $\cos \beta(t) = \pm 1$. By expressing both $\theta(t) - \theta(0)$ and $t$ by the integrals of Lemma~\ref{l:calcoli}, we obtain the identity
\[
\theta(t) - \theta(0) - \frac{\cos \beta_0}{r_{\max}} t = \frac{\cos \beta_0}{2} \ \int_{\frac{\cos^2 \beta_0}{\cos^2 \beta (t)}}^1 \left( 1 - \frac{1}{u} \right) \frac{\sigma_N'\bigl(r_{\max} \sqrt{u}\bigr)}{\sqrt{u-\cos^2 \beta_0}} {\rm d}u,
\]
and hence the upper bound
\begin{gather}
\label{e:ub}
\max_{t\in [0,\tau_1]} \bigl| \theta(t) - \theta(0) - { \frac{\cos \beta_0}{r_{\max}}} t \bigr| \leq \frac{|\cos \beta_0|}{2} \int_{\cos^2 \beta_0}^1 \left( 1 - \frac{1}{u} \right) \frac{\sigma_N'\bigl(r_{\max} \sqrt{u}\bigr)}{\sqrt{u-\cos^2 \beta_0}} {\rm d}u.
\end{gather}
Let $\sigma_1\in (\sigma_{\max},m)$ be such that the curvature $K(\sigma,\theta)$ is non-negative for every $\sigma\in [\sigma_{\max},\sigma_1]$, or equivalently, $r$ is concave on $[\sigma_{\max},\sigma_1]$. Let $\beta_1\in (0,\pi/2)$ be such that
\[
\sigma_N\bigl(r_{\max} \sqrt{u}\bigr) \in [\sigma_{\max},\sigma_1], \qquad \forall u\in \bigl[\cos^2 \beta_1,1\bigr].
\]
We claim that there exists $c>0$ such that
\begin{gather}
\label{e:claim}
0 \leq \left( 1 - \frac{1}{u} \right) \sigma_N'\bigl(r_{\max} \sqrt{u}\bigr) \leq c \qquad \forall u\in \bigl[\cos^2 \beta_1,1\bigr).
\end{gather}
Indeed, for $u$ in this interval we can set
$
s:=\sigma_N\bigl(r_{\max} \sqrt{u}\bigr)\in (\sigma_{\max},\sigma_1]$,
or, equivalently \smash{$u = \frac{r(s)^2}{r_{\max}^2}$}, and obtain
\[
\left( 1 - \frac{1}{u} \right) \sigma_N'\bigl(r_{\max} \sqrt{u}\bigr) = \left( 1 - \frac{r_{\max}^2}{r(s)^2} \right) \frac{1}{r'(s)} = \frac{r(s)+r_{\max}}{r(s)^2} \frac{r(s)-r(\sigma_{\max})}{r'(s)}.
\]
The function
\begin{gather}
\label{e:easy}
s\mapsto \frac{r(s)+r_{\max}}{r(s)^2}
\end{gather}
is bounded on the interval $[\sigma_{\max},\sigma_1]$, as here $r$ is bounded away from zero. By the mean value theorem, there exists $s_1\in (\sigma_{\max},s)$ such that
\[
\frac{r(s)-r(\sigma_{\max})}{r'(s)} = \frac{r'(s_1)}{r'(s)} (s-\sigma_{\max}).
\]
Since $r$ is concave on $[\sigma_{\max},\sigma_1]$, we have $r'(s_1) \geq r'(s)$. Together with the fact that $r'(s)$ is~negative, we deduce that
\smash{$
0 \leq \frac{r'(s_1)}{r'(s)} \leq 1$},
and hence
\[
0 \leq \frac{r(s) - r(\sigma_{\max})}{r'(s)} \leq s-\sigma_{\max} \leq \sigma_1 - \sigma_{\max}.
\]
Together with the fact that the function~\eqref{e:easy} is bounded, this proves the existence of a~bound of the form~\eqref{e:claim}, as claimed.

The inequalities~\eqref{e:ub} and~\eqref{e:claim} imply the following bound for every $\beta_0 \in (0,\beta_1]$:
\[
\max_{t\in [0,\tau_1]} \left| \theta(t) - \theta(0) - { \frac{\cos \beta_0}{r_{\max}}} t \right| \leq \frac{\cos \beta_0}{2} \cdot c \int_{\cos^2 \beta_0}^1 \frac{1}{\sqrt{u-\cos^2 \beta_0}} {\rm d}u = c \cos \beta_0 \sin \beta_0,
\]
and hence
\begin{gather}
\label{e:partial}
\lim_{\beta_0\downarrow 0} \max_{t\in [0,\tau_1(\beta_0)]} \left| \theta(t) - \theta(0) - { \frac{\cos \beta_0}{r_{\max}}} t \right| = 0.
\end{gather}
By the $S^1$-symmetry, $t=2 \tau_1(\beta_0)$ is the first positive time for which $\sigma(t)=\sigma_{\max}$, and we have
\begin{gather*}
\theta(\tau_1(\beta_0)+ t) = 2 \theta(\tau_1(\beta_0)) - \theta(\tau_1(\beta_0)-t),
\qquad \forall t\in [0,\tau_1(\beta_0)], \\
\phi^{2\tau_1(\beta_0)} (\sigma_{\max},\theta_0,\beta_0) = (\sigma_{\max},\theta(2\tau_1(\beta_0)),-\beta_0).
\end{gather*}
By the first identity, the maximum in~\eqref{e:partial} can be extended over the interval $[0,2\tau_1(\beta_0)]$. By the second identity, the analogous uniform limit over the interval $[2\tau_1(\beta_0),\tau(v)]$ follows from what we just proved by considering the metric with profile function $\sigma \mapsto r(m-\sigma)$.
\end{proof}

\begin{proof}[Proof of Proposition~\ref{p:non-neg-curv}]
Recall that the conjugacy
\smash{$
h\colon T_{g_1}^1 S^2 \setminus \Gamma_{g_1} \rightarrow T_{g_2}^1 S^2 \setminus \Gamma_{g_2}
$}
from Section~\ref{s:away} is obtained by dynamical continuation starting from the diffeomorphism
\[
\Sigma_{g_1} \rightarrow \Sigma_{g_2}, \qquad (\sigma_{\max,1},\theta,\beta) \mapsto (\sigma_{\max,2},\theta,\beta), \qquad \forall (\theta,\beta)\in S^1 \times (0,\pi),
\]
where $\sigma_{\max,1}$ and $\sigma_{\max,2}$ are the maximizers of the profile functions $r_1$ and $r_2$ of $g_1$ and $g_2$ (see Remark~\ref{r:dyn-cont}). The above restriction of $h$ to $\Sigma_{g_1}$ has an obvious regular extension to $\overline{\Sigma}_{g_1} = \Sigma_{g_1} \cup \Gamma_{g_1}$, and we use this extension to define $h$ on the whole $T_{g_1}^1 S^2$. This map is clearly bijective and conjugates the two geodesic flows, so we only need to prove its continuity at $\Gamma_{g_1}$.

Let $(v_j)$ be a~sequence in \smash{$T^1_{g_1} S^2$} converging to some $v\in \Gamma_{g_1}$, and assume without loss of generality that $v=\bigl(\sigma_{\max,1},\overline{\theta},0\bigr)$. We must prove that $h(v_j)$ converges to $\bigl(\sigma_{\max,2},\overline{\theta},0\bigr)$ in $T^1_{g_2} S^2$. Since the restriction of $h$ to $\Gamma_{g_1}$ is continuous, we may assume that $v_j$ is not in $\Gamma_{g_1}$ for every~$j$. Then
\smash{$
v_j = \phi_{g_1}^{t_j}(w_j)
$}
for some $w_j\in \Sigma_{g_1}$ and $t_j\in [0,\tau_{g_1}(w_j))$. Writing
$
w_j = (\sigma_{\max,1}, \theta_j, \beta_j)$ with~$ {(\theta_j,\beta_j) \in S^1 \times (0,\pi)}$,
the conservation of Clairaut's first integral of $g_1$ gives us
\[
r_{\max} \cos \beta_j = C_{g_1}(w_j) = C_{g_1}(v_j) \rightarrow C_{g_1}(v) = r_{\max} \qquad \mbox{for } j\rightarrow \infty.
\]
Here, we are denoting by $r_{\max}$ the maximum of $r_1$, which coincides with the maximum of $r_2$ because the equators of both metrics have the same length $2\pi r_{\max}$. The above limit implies that~$(\beta_j)$ converges to $0$. By definition,
\smash{$
h(v_j) = \phi_{g_2}^{t_j} ( h(w_j) )= \phi_{g_2}^{t_j} ( \sigma_{\max,2},\theta_j,\beta_j)$}.
The conservation of Clairaut's first integral of $g_2$ gives us
\[
r_2(\sigma(h(v_j))) \cos \beta (h(v_j)) = C_{g_2}(h(v_j)) = C_{g_2}(h(w_j)) = r_{\max} \cos \beta_j \longrightarrow r_{\max}
\]
for $j\rightarrow \infty$, and hence the $\sigma$-coordinate of $h(v_j)$ tends to $\sigma_{\max,2}$ and its $\beta$-coordinate tends to~$0$. By applying Lemma~\ref{l:stability} to the sequences $(w_j)\subset \Sigma_{g_1}$, $(h(w_j)) \subset \Sigma_{g_2}$, and $t_j \in [0, \tau_{g_1}(w_j)] = [0,\tau_{g_2}(h(w_j))]$, we obtain
\begin{gather*}
\lim_{j\to\infty}
\left( \theta(v_j) - \theta_j - \frac{\cos \beta_j}{r_{\max}} t_j \right) =0 \quad \operatorname{mod} 2\pi,\\
\lim_{j\to\infty}
\left(
\theta(h(v_j)) - \theta_j - \frac{\cos \beta_j}{r_{\max}} t_j \right) =0 \quad \operatorname{mod} 2\pi.
\end{gather*}
Therefore,
\smash{$
\lim_{j\to\infty}
\bigl(\theta(h(v_j)) - \overline{\theta} \bigr)
=
\lim_{j\to\infty}
\bigl(\theta(h(v_j)) - \theta(v_j) \bigr)
=0$},
and we conclude that $h(v_j)$ tends to $\bigl(\sigma_{\max,2}, \overline{\theta},0\bigr)$ in $T^1_{g_2} S^2$, as we wished to prove.
\end{proof}

\begin{Remark}
\label{r:control}
Let $v_{\beta}\in \Sigma$ be a~unit tangent vector based at the equator and forming an angle~${\beta\in (0,\pi)}$ with the positive direction of the equator. As discussed in Lemma~\ref{l:extension}, if the curvature of $g$ vanishes along the equator then $\tau(v_{\beta})$ diverges to $+\infty$ for $\beta\downarrow 0$. The higher the vanishing order of the curvature along the equator, the faster this divergence occurs. However, if the curvature is assumed to be non-negative in a~neighborhood of the equator, then this divergence has the universal upper bound
\begin{gather}
\label{e:controllato}
\tau(v_{\beta}) = o \left( \frac{1}{\beta} \right) \qquad \mbox{for } \beta\downarrow 0.
\end{gather}
This can be proven as follows. Using the notation introduced in this section, \eqref{e:tempo} gives us
\[
\tau_1(\beta) = \int_{\sigma_{\max}}^{\sigma_N( r_{\max} \cos \beta)} \frac{r(s)}{\sqrt{r(s)^2 - r_{\max}^2 \cos^2 \beta}}{\rm d}s.
\]
By the bounds $r(s)\leq r_{\max}$ and $r(s)+r_{\max} \cos \beta \geq r_{\max} \cos \beta$, we obtain
\begin{align*}
\tau_1(\beta) &= \int_{\sigma_{\max}}^{\sigma_N( r_{\max} \cos \beta)} \frac{r(s)}{\sqrt{r(s) + r_{\max} \cos \beta}} \frac{1}{\sqrt{r(s) - r_{\max} \cos \beta}}{\rm d}s \\ &\leq \sqrt{\frac{r_{\max}}{\cos \beta}} \int_{\sigma_{\max}}^{\sigma_N( r_{\max} \cos \beta)} \frac{1}{\sqrt{r(s) - r_{\max} \cos \beta}}{\rm d}s,
\end{align*}
where we are assuming $\beta\in \left(0,\frac{\pi}{2}\right)$. By the assumption on the sign of the curvature near the equator, $r$ is concave near $\sigma_{\max}$ and hence its graph lies above its chords, giving us the inequality
\[
r(s) \geq r_{\max} - \frac{r_{\max} (1-\cos \beta)}{\sigma_N(r_{\max} \cos \beta) - \sigma_{\max}} (s-\sigma_{\max}), \qquad \forall s \in [\sigma_{\max}, \sigma_N(r_{\max} \cos \beta)],
\]
for $\beta>0$ small enough, and hence the last integral is bounded from above by
\begin{gather*}
\frac{1}{\sqrt{r_{\max}(1-\cos \beta)}} \int_{\sigma_{\max}}^{\sigma_N( r_{\max} \cos \beta)} \frac{1}{\sqrt{1 - \frac{s - \sigma_{\max}}{\sigma_N(r_{\max} \cos \beta) - \sigma_{\max}}}}{\rm d}s \\
\qquad= \frac{\sigma_N(r_{\max} \cos \beta) - \sigma_{\max}}{\sqrt{r_{\max}(1-\cos \beta)}} \int_0^1 \frac{{\rm d} u}{\sqrt{1-u}} = 2 \frac{\sigma_N(r_{\max} \cos \beta) - \sigma_{\max}}{\sqrt{r_{\max}(1-\cos \beta)}}.
\end{gather*}
We deduce that
\[
\tau_1(\beta) \leq \frac{2}{\sqrt{\cos \beta}} \frac{\sigma_N(r_{\max} \cos \beta) - \sigma_{\max}}{\sqrt{1-\cos \beta}}
\]
for every $\beta>0$ small enough. Since the numerator on the right-hand side is infinitesimal for~${\beta\downarrow 0}$, we deduce that
\[
\tau_1(\beta) = o \left( \frac{1}{\sqrt{1-\cos \beta}} \right) = o \left( \frac{1}{\beta} \right) \qquad \mbox{for } \beta\downarrow 0,
\]
and hence~\eqref{e:controllato} holds, because $\tau(v_{\beta})$ equals $2 \tau_1(\beta)$ plus the analogous quantity for the metric with profile function $\sigma\mapsto r(m-\sigma)$. When the curvature is allowed to change sign arbitrarily close to the equator -- something which may happen only if the curvature has infinite vanishing order at the equator -- $\tau(v_{\beta})$ may diverge arbitrarily fast along suitable sequences $\beta_k\downarrow 0$ (see Remark~\ref{r:nocontrol}).
\end{Remark}

\section[Proof of Theorem~B and order of vanishing of the curvature]{Proof of Theorem~\ref{thmB} and order of vanishing of the curvature}
\label{s:abel}

Let $g$ be a~metric in $\mathcal{G}$ with profile function $r\colon [0,m] \rightarrow [0,+\infty)$, achieving its maximum~$r_{\max}$ at~$\sigma_{\max}$. We now wish to express the first return time $\tau\colon \Sigma \rightarrow (0,\infty)$ and first return map $\varphi\colon \Sigma \rightarrow \Sigma$ to the Birkhoff annulus of the equator of $g$ in terms of the profile function. More precisely, the formula will involve the inverses $\sigma_S$ and~$\sigma_N$ of the restriction of $r$ to $[0,\sigma_{\max}]$ and~${[\sigma_{\max},m]}$, as defined in~\eqref{e:inverses}.

Denote by $ L^1_{\mathrm{loc}}((0,1])$ the space of measurable real functions on $(0,1]$ which are summable on every compact subinterval $[a,1]$, $a>0$. We recall that the Abel transform is the linear operator
\[
\mathcal{A} \colon\ L^1_{\mathrm{loc}}((0,1]) \rightarrow L^1_{\mathrm{loc}}((0,1]), \qquad (\mathcal{A} f)(y) := \int_y^1 \frac{f(x)}{\sqrt{x-y}}{\rm d}x.
\]
This operator is injective. Indeed, its square is easily computed to be the injective operator
\[
\bigl(\mathcal{A}^2 f\bigr)(y) = \pi \int_y^1 f(x){\rm d}x.
\]
See~\cite{abe26}, and also~\cite[Theorem 2.2.3]{psu23} for a~modern presentation. The next result is an easy consequence of Lemma~\ref{l:calcoli}.

\begin{Proposition}
\label{p:abel}
If $v\in \Sigma$ has geodesic polar coordinates $(\sigma_{\max},\theta,\beta)$ with $\beta\neq\pi/2$, then
\begin{gather*}
\tau(v) = r_{\max} \mathcal{A} \bigl( u \mapsto ( \sigma_S'- \sigma_N') \bigl( r_{\max} \sqrt{u} \bigr) \bigr) \bigl(\cos^2 \beta\bigr),\qquad
\varphi(v) = (\sigma_{\max},\theta+ \Theta(\beta),\beta),
\end{gather*}
where
\[
\Theta(\beta) = \cos \beta \mathcal{A} \left( u \mapsto \frac{1}{u} ( \sigma_S'- \sigma_N') \bigl( r_{\max} \sqrt{u} \bigr) \right) \bigl(\cos^2 \beta\bigr).
\]
\end{Proposition}

We can now prove Theorem~\ref{thmB}. Let $g_1$ and $g_2$ be metrics in $\mathcal{G}$, and denote by $\sigma_{S,1}$, $\sigma_{N,1}$, $\sigma_{S,2}$, and $\sigma_{N,2}$ the partial inverses of their profile functions $r_1\colon [0,m_1] \rightarrow [0,+\infty)$ and $r_2\colon [0,m_2]\rightarrow [0,+\infty)$.

Assume that $g_1$ and $g_2$ are isospectral, and in particular $\max r_1 = \max r_2 = r_{\max}$. By Proposition~\ref{p:equalF}, the first return time functions $\tau_{g_1}$ and $\tau_{g_2}$ and first return maps $\varphi_{g_1}$ and $\varphi_{g_2}$ are the same, once the Birkhoff annuli $\Sigma_{g_1}$ and $\Sigma_{g_2}$ are canonically identified. In particular, the meridians of $g_1$ and $g_2$ have the same length, as this length is the value of $\tau$ at vectors with~${\beta=\frac{\pi}{2}}$, and we deduce that $m_1=m_2$. In the notation of Proposition~\ref{p:abel}, the maps $\varphi_{g_1}$ and~$\varphi_{g_2}$ are determined by the same function $\Theta=\Theta(\beta)$. By this proposition and the injectivity of the Abel transform, we have
\[
\frac{1}{u} ( \sigma_{S,1}'- \sigma_{N,1}') \bigl( r_{\max} \sqrt{u} \bigr) = \frac{1}{u} ( \sigma_{S,2}'- \sigma_{N,2}') \bigl( r_{\max} \sqrt{u} \bigr) \qquad \forall u\in (0,1),
\]
and hence
$
\sigma_{S,1}'- \sigma_{N,1}' = \sigma_{S,2}'- \sigma_{N,2}' $ on $ [0,r_{\max})$.
Since $(\sigma_{S,1}- \sigma_{N,1})(0) = - m_1 = - m_2 = (\sigma_{S,2}- \sigma_{N,2})(0)$, we deduce that
\begin{gather}
\label{e:equiinv}
\sigma_{S,1} - \sigma_{N,1} = \sigma_{S,2} - \sigma_{N,2} \qquad \mbox{on} \ [0,r_{\max}].
\end{gather}
Since
\[
\sigma_{N,j}(\rho) - \sigma_{S,j}(\rho) = \mathrm{length} ( \{ \sigma\in [0,m_j] \mid r_j(\sigma) \geq \rho \} ), \qquad j=1,2,
\]
we conclude that
\begin{gather}
\label{e:equipese}
\mathrm{length} ( \{ \sigma\in [0,m_1] \mid r_1(\sigma) \geq \rho \} ) = \mathrm{length} ( \{ \sigma\in [0,m_2] \mid r_2(\sigma) \geq \rho \} ), \qquad \forall \rho\geq 0.
\end{gather}
Conversely, assume that~\eqref{e:equipese} holds. This implies that $m_1=m_2$, that $r_1$ and $r_2$ have the same maximum $r_{\max}$, and that~\eqref{e:equiinv} holds. By Proposition~\ref{p:abel}, the first return maps $\varphi_{g_1}$ and $\varphi_{g_2}$ are the same. As observed in Remark~\ref{r:converse}, this implies that the metrics $g_1$ and $g_2$ are isospectral. This concludes the proof of Theorem~\ref{thmB}.

Corollary~\ref{corC} is an immediate consequence of Theorem~\ref{thmB}. Indeed, a~metric $g\in \mathcal{G}$ is $\Z_2$-symmetric if and only if its profile function $r\colon [0,m]\rightarrow [0,+\infty)$ satisfies $r(m-\sigma) = r(\sigma)$ for every $\sigma\in [0,m]$, and profile functions with this symmetry property are uniquely determined by the family of lengths of their superlevel sets $r^{-1}([\rho,+\infty))$.

Recall that the order of vanishing at $x_0$ of a~smooth function $f$ of one real variable is the smallest non-negative integer $k$ such that $f^{(k)}(x_0)$ is non-zero, and is infinite if no such $k$ exists. Order of vanishing zero means $f(x_0)\neq 0$. We conclude this section by showing that if $g\in \mathcal{G}$ then the order of vanishing of the curvature of $g$
at the equator $\{\sigma = \sigma_{\max}\}$ is determined by the marked length spectrum of $g$. Together with Proposition~\ref{p:non-neg-curv}, this implies statement (ii) of Theorem~\ref{thmA} and concludes the proof of that theorem.

\begin{Proposition}
\label{p:order}
The marked length spectrum of $g\in \mathcal{G}$ determines the order of vanishing of the curvature $K=K(\sigma,\theta)$ of $g$ at the equator. If such order is finite, then it is an even non-negative integer $2k$, and the marked length spectrum of $g$ determines the number
\smash{$
\frac{\partial^{2k}}{\partial \sigma^{2k}} K(\sigma_{\max},\theta)$},
which is positive.
\end{Proposition}

\begin{proof}
By~\eqref{e:curvature}, the order of vanishing of the curvature at the equator equals the order of vanishing of the function $R(\sigma):= r_{\max} - r(\sigma)$ at $\sigma_{\max}$ minus 2. Since $R$ has a~minimum at $\sigma_{\max}$ and vanishes there, its order of vanishing -- if finite -- is a~positive even integer $2h$, and $K$ has order of vanishing $2k=2h-2\geq 0$ with
\[
\partial^{2k}_{\sigma} K(\sigma_{\max},\theta) = \frac{1}{r_{\max}} \partial^{2h}_{\sigma} R(\sigma_{\max}),
\]
which is a~positive number. Since the value of $r_{\max}$ is encoded by the marked length spectrum of $g$, it is enough to show that the order of vanishing $2h\in 2\N \cup \{\infty\}$ of $R$ at $\sigma_{\max}$ and the value of its first non-vanishing derivative there are determined by the marked length spectrum.

The function $R$ has finite order of vanishing $2h$ if and only if the function
\smash{$
f(\sigma) := \frac{R(\sigma)}{(\sigma - \sigma_{\max})^{2h}}
$}
has a~finite non-zero limit for $\sigma\rightarrow \sigma_{\max}$, and in this case we have
\[
R^{(2h)}(\sigma_{\max}) = (2h)! \lim_{\sigma \rightarrow \sigma_{\max}} f(\sigma).
\]
Evaluating the identity
$
r_{\max} - r(\sigma) = f(\sigma) (\sigma-\sigma_{\max})^{2h}
$
at $\sigma=\sigma_S(\rho)$ and $\sigma=\sigma_{N}(\rho)$, where $\rho\in [0,r_{\max})$, we obtain
\[
\sigma_S(\rho) = \sigma_{\max} - \left( \frac{r_{\max} - \rho}{f(\sigma_S(\rho))} \right)^{\frac{1}{2h}}, \qquad
\sigma_N(\rho) = \sigma_{\max} + \left( \frac{r_{\max} - \rho}{f(\sigma_N(\rho))} \right)^{\frac{1}{2h}}.
\]
Therefore, for every $\rho\in [0,r_{\max})$ we have
\[
\frac{\mathrm{length}(\{\sigma\in [0,m] \mid r(\sigma) \geq \rho \})}{(r_{\max}-\rho)^{\frac{1}{2h}}} = \frac{\sigma_N(\rho) - \sigma_S(\rho)}{(r_{\max}-\rho)^{\frac{1}{2h}}} = f(\sigma_N(\rho))^{-\frac{1}{2h}} + f(\sigma_S(\rho))^{-\frac{1}{2h}},
\]
and hence
\[
\lim_{\rho \uparrow r_{\max}} \frac{\mathrm{length}(\{\sigma\in [0,m] \mid r(\sigma) \geq \rho\})}{(r_{\max}-\rho)^{\frac{1}{2h}} } = 2 \lim_{\sigma\rightarrow \sigma_{\max}} f(\sigma)^{-\frac{1}{2h}}.
\]
By Theorem~\ref{thmB}, the marked length spectrum of $g$ determines the limit on the left-hand side, and hence also the limit on the right-hand side. By our previous considerations, this proves that the marked length spectrum determines the order of vanishing $2h\in 2\N \cup \{\infty\}$ of the function $R$ at~$\sigma_{\max}$ and, when $h$ is finite, also the value of $R^{(2h)}(\sigma_{\max})$, as we wished to show.
\end{proof}

\section{Isospectral classes}
\label{s:isospectral_classes}

In this section, we discuss the proof of Corollary~\ref{corD}. Let $r\colon [0,m]\rightarrow [0,+\infty)$ be a~profile function. Recall that this means that $r$ is positive on $(0,m)$, vanishes at $0$ and $m$, satisfies~${r'(0)=1=-r'(m)}$, and is the restriction of a~$2m$-periodic and odd smooth function on $\R$. The profile function $r$ is here said to be \textit{unimodal} if it has a~unique critical point $\sigma_{\max}\in (0,m)$ with $r(\sigma_{\max})=r_{\max} = \max r$. Unimodal profile functions are precisely the profile functions of the metrics in $\mathcal{G}$.

The \textit{symmetric rearrangement} of the unimodal profile function $r$ is the unique function $r_s\colon [0,m] \rightarrow [0,+\infty)$ such that $r_s(m-\sigma) = r_s(\sigma)$ for every $\sigma\in [0,m]$, \smash{$r_s|_{[0,\frac{m}{2}]}$} is monotonically increasing, and
\begin{gather}
\label{e:same_superlevel_size}
\mathrm{length}\bigl( r_s^{-1} ([\rho,+\infty)) \bigr) = \mathrm{length}\bigl( r^{-1} ([\rho,+\infty)) \bigr), \qquad \forall \rho\geq 0.
\end{gather}
See~\cite[Chapter 3]{ll01} for general facts about symmetric rearrangements. Equivalently, $r_s = r\circ \varphi^{-1}$, where $\varphi\colon [0,m] \rightarrow [0,m]$ is the bijective function
\begin{align}
\label{e:diffeo}
\varphi(\sigma)
=
\begin{cases}
\displaystyle\frac{m}{2} - \frac{1}{2} \mathrm{length} \bigl( r^{-1}([r(\sigma),+\infty)) \bigr) & \mbox{for } 0\leq \sigma\leq \sigma_{\max},\vspace{5pt} \\
\displaystyle\frac{m}{2} + \frac{1}{2} \mathrm{length} \bigl( r^{-1}([r(\sigma),+\infty)) \bigr) & \mbox{for } \sigma_{\max} \leq \sigma \leq m.
\end{cases}
\end{align}

Recall also that $\mathcal{G}_*$ denotes the set of smooth metrics in $\mathcal{G}$ whose curvature does not vanish to infinite order at the equator. Equivalently, $\mathcal{G}_*$ consists of the metrics in $\mathcal{G}$ with profile function~${r\colon [0,m] \rightarrow [0,+\infty)}$ such that the function $r- r_{\max}$ has finite vanishing order at $\sigma_{\max}$. The next lemma shows in particular that the symmetric rearrangement of the profile function of a~metric in $\mathcal{G}_*$ is a~unimodal profile function as well.

\begin{Lemma}
\label{l:rearrangements}
Let $r\colon [0,m] \rightarrow [0,+\infty)$ be a~smooth $($resp.\ analytic$)$ unimodal profile function with $r(\sigma_{\max}) = r_{\max} = \max r$. If $r-r_{\max}$ has finite vanishing order at $\sigma_{\max}$, then the map $\varphi$ defined in~\eqref{e:diffeo} is a~smooth $($resp.\ analytic$)$ diffeomorphism of $[0,m]$ with $\varphi'(0) =\varphi'(\sigma_{\max})=\varphi'(m)=1$ and extends to a~smooth $($resp.\ analytic$)$ diffeomorphism $\varphi\colon \R \rightarrow \R$ such that
\begin{gather}
\label{e:extension}
\varphi(\sigma+ 2m) = \varphi(\sigma) + 2m, \qquad \varphi(-\sigma) = - \varphi(\sigma), \qquad \forall \sigma\in \R.
\end{gather}
In particular, the symmetric rearrangement $r_s=r\circ \varphi^{-1}$ is a~smooth $($resp.\ analytic$)$ unimodal profile function as well.
\end{Lemma}

\begin{proof}
The functions
\smash{$
\sigma_S:=r|_{[0,\sigma_{\max}]}^{-1} $} and \smash{$ \sigma_N:=r|_{[\sigma_{\max},m]}^{-1}$}
restrict to diffeomorphisms from $[0,r_{\max})$ to either $[0,\sigma_{\max})$ or $(\sigma_{\max},m]$, such that $\sigma_S'(0)=-\sigma_N'(0)=1$.
Since $r$ admits a~smooth (resp.\ analytic) extension to an odd $2m$-periodic function, $\sigma_S$ and $\sigma_N-m$ admit smooth (resp.\ analytic) odd extensions to $(-r_{\max},r_{\max})$. Since
\[
\mathrm{length} \bigl( r^{-1}([\rho,+\infty)) \bigr) = \sigma_N(\rho)-\sigma_S(\rho), \qquad \forall \rho\in [0,r_{\max}],
\]
the function $\varphi$ has the form
\begin{align*}
\varphi(\sigma)
=
\left\{
 \begin{array}{@{}ll}
 \displaystyle\frac{m}{2} - \frac{1}{2} \bigl( \sigma_N(r(\sigma)) - \sigma_S(r(\sigma)) \bigr) & \mbox{for } 0 \leq \sigma\leq \sigma_{\max},\vspace{5pt} \\
 \displaystyle\frac{m}{2} + \frac{1}{2} \bigl( \sigma_N(r(\sigma)) - \sigma_S(r(\sigma)) \bigr) &\mbox{for } \sigma_{\max} \leq \sigma\leq m, \\
 \end{array}
\right.
\end{align*}
and hence admits an extension to a~homeomorphism of $\R$ which satisfies~\eqref{e:extension}, is smooth (resp.\ analytic) with positive derivative outside of $\pm \sigma_{\max} + 2m \Z$, and satisfies
$
\varphi'(0)=\varphi'(m)=1$.
There remains to prove that $\varphi$ is smooth (resp.\ analytic) in a~neighborhood of $\sigma_{\max}$ with $\varphi'(\sigma_{\max})=1$. In particular, the positivity of $\varphi'(\sigma_{\max})$ implies that $\varphi$ is a~diffeomorphism on the whole $\R$, thanks to~\eqref{e:extension}.

Since we are assuming that $r-r_{\max}$ has finite vanishing order at $\sigma_{\max}$, we can find $k\in \N$ such that $r'(\sigma_{\max}) = \dots = r^{(2k-1)}(\sigma_{\max}) = 0$ and $r^{(2k)} (\sigma_{\max}) < 0$. Then there is a~smooth (resp.\ analytic) increasing diffeomorphism $\nu$ from a~small interval $(-\epsilon,\epsilon)$ to an open interval containing $\sigma_{\max}$ such that $\nu(0) = \sigma_{\max}$ and
\begin{gather}
\label{e:gen-morse}
r(\nu(u)) = r_{\max} - u^{2k} \qquad \forall u\in (-\epsilon,\epsilon).
\end{gather}
Indeed, since the non-negative function $r_{\max} - r$ vanishes with order $2k$ at $\sigma_{\max}$, Taylor's formula with integral remainder gives us
\[
r_{\max} - r(\sigma) = (\sigma-\sigma_{\max})^{2k} f(\sigma),
\]
for some smooth (resp.\ analytic) positive function $f$ on a~neighborhood of $\sigma_{\max}$, and, choosing~$\nu$ to be the inverse of the local diffeomorphism at $\sigma_{\max}$ given by
\smash{$
\sigma \mapsto f(\sigma)^{\frac{1}{2k}} (\sigma-\sigma_{\max})$},
we deduce that~\eqref{e:gen-morse} holds for $\epsilon>0$ small enough. Then we have
\begin{align*}
 \varphi(\nu(u)) =\frac{m + \nu(u)-\nu(-u)}2, \qquad \forall u\in (-\epsilon,\epsilon),
\end{align*}
and we infer that $\varphi$ is smooth (resp.\ analytic) on a~neighborhood of $\sigma_{\max}$. By differentiating the latter identity, we obtain
\begin{align*}
 \varphi'(\nu(u))\nu'(u) =\frac{\nu'(u)+\nu'(-u)}2,
\end{align*}
and by evaluating this expression at $u=0$, we conclude that $\varphi'(\sigma_{\max})=1$. This concludes the proof of the properties of $\varphi$. By these properties, $r_s = r\circ \varphi^{-1}$ is a~smooth (resp.\ analytic) unimodal profile function.
\end{proof}

\begin{Remark}\label{r:smooth?}
If $r\colon [0,m] \rightarrow [0,+\infty)$ is a~unimodal profile function such that $r - \max r$ vanishes to infinite order at the maximum point of $r$, then its symmetric rearrangement $r_s$ is smooth on~${[0,m] \setminus \{\frac{m}{2}\}}$ and satisfies
\[
\lim_{\sigma \rightarrow \frac{m}{2}} \frac{r_s(\sigma) - r_s\left(\frac{m}{2}\right)}{\left(\sigma - \frac{m}{2}\right)^k}=0 \qquad \forall k\in \N,
\]
but is in general not smooth at $\frac{m}{2}$. Explicit computations show that it is of class $C^3$, where the proof of the fact that $r_s'''(\sigma)$ tends to zero for $\sigma\rightarrow \frac{m}{2}$ uses the following inequality which is due to Glaeser~\cite{gla63b} (see also~\cite{die70} for this formulation): if the smooth function $f\colon \R \rightarrow [0,+\infty)$ vanishes at zero together with its first and second derivatives, then
\[
f'(x)^2 \leq 2 f(x) \cdot \max_{|y|\leq 2 |x|} |f''(y)|
\]
for every $x\in \R$. However, the fourth derivative of $r_s$ need not converge to zero at $\frac{m}{2}$ and could even diverge. An explicit example of a~smooth unimodal profile $r$ whose symmetric rearrangement is not $C^4$ was produced for us by Chat GPT Pro.
\end{Remark}

We shall also need the following easy observation.

\begin{Lemma}
\label{l:easy}
Let $r_s\colon [0,m] \rightarrow [0,+\infty)$ be a~unimodal profile function such that $r_s(m-\sigma) = r_s(\sigma)$ for every $\sigma\in [0,m]$, and let $r=r_s\circ \varphi$ for some orientation-preserving homeomorphism~${\varphi\colon [0,m] \rightarrow [0,m]}$. Writing $\varphi^{-1}(\tau) = \tau + \psi(\tau)$, we have that $r_s$ is the symmetric rearrangement of $r$ if and only if $\psi(m-\tau) = \psi(\tau)$ for every $\tau \in [0,m]$.
\end{Lemma}

\begin{proof}
Given $\rho\in [0,\max r_s]$, let $\tau\in \bigl[0,\frac{m}{2}\bigr]$ be such that $r_s(\tau)=r_s(m-\tau)=\rho$. Then $r_s^{-1}([\rho,+\infty))=[\tau,m-\tau]$ and
\[
r^{-1}([\rho,+\infty)) = \varphi^{-1} \bigl(r_s^{-1} ([\rho,+\infty)\bigr)) = \varphi^{-1} ([\tau,m-\tau]) = [\tau + \psi(\tau), m - \tau + \psi(m-\tau)].
\]
Therefore, the intervals $r_s^{-1}([\rho,+\infty))$ and $r^{-1}([\rho,+\infty))$ have the same length if and only if $\psi(m-\tau)=\psi(\tau)$. We conclude that $r_s$ is the symmetric rearrangement of $r$ if and only if the latter identity holds for every $\tau\in \bigl[0,\frac{m}{2}\bigr]$, which is equivalent to the fact that it holds for every~${\tau\in [0,m]}$.
\end{proof}

We can now prove Corollary~\ref{corD}. If $r$ is the profile function of a~smooth (resp.\ analytic) metric~${g\in \mathcal{G}_*}$, then the order of vanishing of $r-\max r$ at the maximum point of $r$ is finite, and hence Lemma~\ref{l:rearrangements} implies that its symmetric rearrangement $r_s$ is the profile function of a~smooth (resp.\ analytic) $\Z_2$-symmetric metric $g_s\in \mathcal{G}_*$. By~\eqref{e:same_superlevel_size} and Theorem~\ref{thmB}, $g_s$ is isospectral to $g$. The uniqueness of $g_s$ follows from Corollary~\ref{corC}.

Now let $g_s$ be a~$\Z_2$-symmetric smooth (resp.\ analytic) metric in $\mathcal{G}_*$ with profile function $r_s\colon [0,m] \rightarrow [0,+\infty)$, and let $\Psi$ be the set of smooth (resp.\ analytic) $2m$-periodic odd functions~${\psi\colon \R \rightarrow \R}$ satisfying $\psi(m - \tau) = \psi(\tau)$ for every $\tau \in \R$, $\psi'(0) = 0$, and $|\psi'| < 1$.

If the smooth (resp.\ analytic) metric $g\in \mathcal{G}_*$ with profile function $r$ is isospectral to $g_s$ then Theorem~\ref{thmB} implies that $r_s$ is the symmetric rearrangement of $r$, and by Lemma~\ref{l:rearrangements} we have~${r=r_s\circ \varphi}$ where $\varphi\colon \R \rightarrow \R$ is a~smooth (resp.\ analytic) diffeomorphism with $\varphi'(0)=1$ and satisfying~\eqref{e:extension}. Writing $\varphi^{-1}(\tau) = \tau + \psi(\tau)$ where $\psi$ is $2m$-periodic and odd, Lemma~\ref{l:easy} implies that $\psi(m-\tau) = \psi(\tau)$ for every $\tau\in [0,m]$ and hence, by oddness and $2m$-periodicity, for every $\tau\in \R$. The fact that $\varphi'>0$ implies that $\psi'>-1$. Together with the identity~${\psi'(\tau) = - \psi'(m-\tau)}$, we deduce that $|\psi'|<1$. From $\varphi'(0)=1$, we obtain $\psi'(0)=0$, and we conclude that $\psi$ is in $\Psi$.

Conversely, for every $\psi\in \Psi$ the map $\tau \rightarrow \tau + \psi(\tau)$ is an orientation-preserving diffeomorphism of $\R$ and its inverse $\varphi\colon \R \rightarrow \R$ satisfies $\varphi([0,m]) = [0,m]$, $\varphi'(0)=\varphi'(m)=1$, and~\eqref{e:extension}. Then~${r=r_s\circ \varphi}$ is the profile function of a~smooth (resp.\ analytic) metric $g\in \mathcal{G}_*$, which by Lemma~\ref{l:easy} and Theorem~\ref{thmB} is isospectral to $g_s$. This concludes the proof of Corollary~\ref{corD}.

\section[Examples without smooth S\^{}1-equivariant conjugacy]{Examples without smooth $\boldsymbol{S^1}$-equivariant conjugacy}
\label{s:notsmooth}

Let $g\in \mathcal{G}$ be an analytic $\Z_2$-symmetric metric with associated profile function $r \colon [0, m] \rightarrow \R$.
Notice that the analyticity implies that $g$ is in $\mathcal{G}_*$.
Let $\psi\colon \R \rightarrow \R$ be an analytic $2m$-periodic odd function such that $\psi(m-\tau) = \psi(\tau)$ for every $\tau\in \R$, $\psi'(0)=0$, $|\psi'|\leq1$, and
\begin{gather}
\label{e:condonpsi}
\psi\left( { \frac{m}{2}} \right) = 0, \qquad \psi''\left( { \frac{m}{2}} \right) \neq 0.
\end{gather}
Given $\varepsilon\in (-1,1)$, let $\varphi_{\varepsilon}$ be the inverse of the diffeomorphism $\tau\mapsto \tau + \varepsilon \psi (\tau)$. By Corollary~\ref{corD}, $r_{\varepsilon} := r \circ \varphi_{\varepsilon}$ is the profile function of an analytic metric $g_{\varepsilon}$ which is isospectral to $g$. The path~${\varepsilon \mapsto g_{\varepsilon}}$ is analytic. Note that, if $\varepsilon \neq 0$, \eqref{e:condonpsi} implies that the identity $\varphi_\epsilon(m-\tau)=m-\varphi_\epsilon(\tau)$ does not hold for $\tau\neq\tfrac m2$ sufficiently close to $\tfrac m2$, and therefore $g_{\varepsilon}$ fails to be $\Z_2$-symmetric.\looseness=1

Theorem~\ref{thmA}\,(ii) implies that for every $\varepsilon\in (-1,1)$ the geodesic flow of $g$ is conjugate to the geodesic flow of $g_{\varepsilon}$ by an $S^1$-equivariant homeomorphism $h\colon T^1_{g} S^2 \rightarrow T^1_{g_{\varepsilon}} S^2$ whose restriction to the complement of $\Gamma_g$ is an analytic contactomorphism. The next result shows in particular that if the curvature of $g$ vanishes along the equator and $\varepsilon \neq 0$, then the conjugacy $h$ cannot be everywhere analytic.

\begin{Proposition}
\label{p:notsmooth}
Let $g$ be an analytic $\Z_2$-symmetric metric with vanishing curvature along the equator, and let $\{g_{\varepsilon}\}_{\varepsilon \in (-1,1)}$ be the analytic family of deformations of $g$ defined above. Then for every $\varepsilon \neq 0$ there is no smooth orientation-preserving $S^1$-equivariant conjugacy from the geodesic flow of $g$ to the one of $g_{\varepsilon}$.
\end{Proposition}

The remaining part of this section is devoted to the proof of the above proposition.
This proof uses the following result of Martynchuk and V\~u Ng\d{o}c~\cite{mvn} about diffeomorphisms which preserve a~Hamiltonian on the plane having a~critical point which, in Arnold's notation~\cite{arn76}, is an $A_{2k-1}$ singularity with $k\geq 2$.

\begin{Lemma}
\label{l:marvu}
Let $\omega_0$, $\omega_1$ be smooth symplectic forms on $\R^2$ and, given an integer $k\geq 2$, consider the function $K\colon \R^2 \rightarrow \R$ given by
\[
K(x,y) = x^{2k} + y^2.\]
Let $f\colon \R^2 \rightarrow \R^2$ be an orientation-preserving smooth diffeomorphism fixing the origin and such that $f^* \omega_1=\omega_0$ and $f^* K = K$. Then the following alternative holds:
\begin{enumerate}[$(i)$]\itemsep=0pt
\item The differential of $f$ at the origin has the form
\[
{\rm d}f(0) = \left( \begin{matrix} 1 & c \\ 0 & 1 \end{matrix} \right)
\]
for some $c\in \R$, and there exists a~smooth function $u$ on a~neighborhood of the origin such that
$
{\rm d}K \wedge {\rm d}u = \omega_1 - \omega_0$.
\item The differential of $f$ at the origin has the form
\[
{\rm d}f(0) = \left( \begin{matrix} -1 & c \\ 0 & -1 \end{matrix} \right)
\]
for some $c\in \R$, and there exists a~smooth function $u$ on a~neighborhood of the origin such that
$
{\rm d}K \wedge {\rm d}u = \omega_1 - \iota^* \omega_0$,
where $\iota\colon \R^2 \rightarrow \R^2$ is the involution $\iota(x,y) = (-x,-y)$.
\end{enumerate}
\end{Lemma}

\begin{proof}
Since $f$ is assumed to be an orientation-preserving diffeomorphism leaving the function~$K$ invariant, \cite[Lemma 2.4]{mvn} implies that the differential of $f$ at the origin has one of the two forms listed in (i) and (ii), and that there is a~smooth isotopy $\{f_t\}_{t\in [0,1]}$ such that $f_0=\mathrm{id}$, $f_t^* K = K$ for every $t\in [0,1]$, and $f_1$ is either $f$ -- if alternative (i) holds -- or $f\circ \iota$ -- if alternative (ii) holds. By~\cite[Proposition 2.3]{mvn}, the existence of a~smooth isotopy $f_t$ preserving $K$, with~${f_0=\mathrm{id}}$, and such that $f_1^* \omega_1 = \omega$, is equivalent to the existence of a~smooth local solution near the origin to the cohomological equation
$
 {\rm d}K \wedge {\rm d}u = \omega_1 - \omega$,
and the conclusion follows from the fact that in our case $\omega:=f_1^* \omega_1$ is either $\omega_0$ -- in case (i) -- or $\iota^* \omega_0$ -- in case (ii).
\end{proof}

We denote by $W_{\varepsilon}$ the quotient of $T^1_{g_{\varepsilon}} S^2$ by the $S^1$-action and by $H_{\varepsilon}$ and $\omega_{\varepsilon}$ the Hamiltonian and the symplectic form on the quotient two-sphere $W_{\varepsilon}$ induced by the $S^1$-invariant Hilbert contact form $\alpha_{g_{\varepsilon}}$, as in Section~\ref{s:regular}. In the coordinates $(\sigma,\beta)$ on $W_{\varepsilon}$ which are induced by the geodesic polar coordinates on $T^1_{g_{\varepsilon}} S^2$, by~\eqref{e:HH} and~\eqref{e:omega} we have
\begin{gather}
\label{e:down}
H_{\varepsilon}(\sigma,\beta) = r_{\varepsilon}(\sigma) \cos \beta, \qquad \omega_{\varepsilon} = r_{\varepsilon}(\sigma) {\rm d}\beta \wedge {\rm d}\sigma.
\end{gather}
We identify the surfaces $W_{\varepsilon}$ with the same surface $W$ by using these global coordinates, and simplify the notation by setting $H := H_0$ and $\omega:= \omega_0$, which is consistent with the fact that~${r_0=r}$.

For every $\varepsilon\in (-1,1)$, the component $\Gamma_{g_{\varepsilon}}^+$ of $\Gamma_{g_{\varepsilon}}$ on which $\beta=0$ projects to the point $\Gamma^+:= \left(\frac{m}{2},0\right)$ in $W$.

Since we are assuming that the curvature of the $\Z_2$-symmetric analytic metric $g$ vanishes at the equator, \eqref{e:curvature} implies that the profile function $r$ has a~degenerate maximum at $\frac{m}{2}$, and hence $r-r_{\max}$ vanishes with order $2k$, $k\geq 2$, at $\frac{m}{2}$. Standard results about normal forms of analytic functions show that the function $H(\sigma,\beta) = r(\sigma) \cos \beta$ has an $A_{2k-1}$ singularity at $\Gamma^+$. This is the content of the next lemma,
of which we provide a~self-contained proof for the sake of completeness.

\begin{Lemma}
\label{l:nf}
There exists an analytic diffeomorphism $\Psi$ from a~neighborhood of the origin in~$\R^2$ to a~neighborhood of $\Gamma^+$ in $W$ such that $\Psi(0) = \Gamma^+$ and
\begin{gather}
\label{e:singularity}
H(\Psi(x,y)) = r_{\max} - x^{2k} - y^2.
\end{gather}
Moreover, denoting by $\iota$ and $\jmath$ the involutions $\iota(x,y) = (-x,-y)$ and $\jmath(\sigma,\beta) = (m-\sigma,-\beta)$ on~$\R^2$ and $W$, we have $\Psi\circ \iota = \jmath \circ \Psi$.
\end{Lemma}

\begin{proof}
Since the non-negative analytic function $r_{\max} - r$ vanishes with order $2k$ at $\frac{m}{2}$, we have
\[
r(\sigma) = r_{\max} - a(\sigma) \left(\sigma - { \frac{m}{2}} \right)^{2k},
\]
for some positive analytic function $a$ near $\frac{m}{2}$. By the $\Z_2$-symmetry of $g$, we have $a(m-\sigma)=a(\sigma)$. Similarly
$
\cos \beta = 1 - b(\beta) \beta^2$,
near $0$, where the function $b$ is analytic, even, and $b(0)= \frac{1}{2}$. Therefore,
\[
H(\sigma,\beta) = r(\sigma) \cos \beta = r_{\max} - a(\sigma) \left(\sigma - { \frac{m}{2}} \right)^{2k} - r_{\max} b(\beta) \beta^2 + a(\sigma) b(\beta) \left(\sigma - { \frac{m}{2}} \right)^{2k}\beta^2.
\]
We can regroup terms and write
\[
H(\sigma,\beta) = r_{\max} - A(\sigma,\beta) \left(\sigma - { \frac{m}{2}} \right)^{2k} - B(\sigma,\beta) \beta^2,
\]
where
$
A(\sigma,\beta) := a(\sigma) \bigl( 1 - b(\beta)\beta^2 \bigr)$, $ B(\sigma,\beta) := r_{\max} b(\beta)$.
Since the analytic functions $A$ and $B$ are positive near $\Gamma^+$, the map
\[
\Phi(\sigma,\beta) := \left( A(\sigma,\beta)^{\frac{1}{2k}} \left(\sigma - { \frac{m}{2}} \right), B(\sigma,\beta)^{\frac{1}{2}} \beta \right)
\]
is analytic in a~neighborhood of $\Gamma^+$. It maps $\Gamma^+$ to the origin in $\R^2$ and its differential at $\Gamma^+$ is readily seen to be invertible. Therefore, $\Phi$ is a~local diffeomorphism at $\Gamma^+$ and, if we define $\Psi$ to be its inverse on a~neighborhood of the origin in $\R^2$, \eqref{e:singularity} holds. The symmetry properties of $a$ and $b$ imply that $\Phi\circ \jmath = \iota \circ \Phi$, and hence $\Psi\circ \iota = \jmath \circ \Psi$.
\end{proof}

After these preliminaries, we can prove Proposition~\ref{p:notsmooth}. Fix some $\varepsilon\in (-1,1) \setminus \{0\}$ and assume by contradiction that
$
h\colon T_g^1 S^2 \rightarrow T_{g_{\varepsilon}}^1 S^2
$
is a~smooth orientation-preserving $S^1$-equivariant conjugacy from the geodesic flow of $g$ to the one of $g_{\varepsilon}$. By a~standard application of Moser's homotopy argument, we can assume that $h^* \alpha_{g_{\varepsilon}} = \alpha_g$. Then $h^* C_{g_{\varepsilon}} = C_g$, and hence $h$ maps the component $\Gamma^+_g$ of $\Gamma_g$ to the component $\Gamma^+_{g_{\varepsilon}}$ of $\Gamma_{g_{\varepsilon}}$. Recall that we are identifying the quotients of the different unit tangent bundles by the $S^1$-action with the same surface $W$, and that the above components project to the same point $\Gamma^+=\left(\frac{m}{2},0\right)$ in $W$. Being an $S^1$-equivariant diffeomorphism such that $h^* \alpha_{g_{\varepsilon}} = \alpha_g$, $h$ induces an orientation-preserving diffeomorphism \smash{$\widehat{h}\colon W \rightarrow W$} fixing $\Gamma^+$ and such that
\smash{$
\widehat{h}^* H_{\varepsilon} = H$}, \smash{$ \widehat{h}^* \omega_{\varepsilon} = \omega$}.
Consider the orientation-preserving diffeomorphism
\[
\nu\colon \ W \rightarrow W, \qquad \nu (\sigma,\beta) = \bigl(\varphi_{\varepsilon}^{-1}(\sigma),\beta\bigr) = (\sigma + \varepsilon \psi(\sigma), \beta).
\]
By~\eqref{e:down}, we have
\[
\nu^* H_{\varepsilon} = H, \qquad \nu^* \omega_{\varepsilon} = \tilde{\omega}:= r(\sigma) ( 1 + \varepsilon \psi'(\sigma)) {\rm d}\beta \wedge {\rm d}\sigma,
\]
and the diffeomorphism $f:= \nu^{-1} \circ \widehat{h}$ satisfies
$
f^* H =H$, $ f^* \tilde{\omega} = \omega$.
By Lemma~\ref{l:nf}, $H$ has an $A_{2k-1}$ singularity at $\Gamma^+$ and hence Lemma~\ref{l:marvu} can be applied to the orientation-preserving diffeomorphism~$f$. Since the local chart $\Psi$ from Lemma~\ref{l:nf} intertwines the involutions~$\iota$ and~$\jmath$, and since $\jmath^* \omega = \omega$, both alternatives of Lemma~\ref{l:marvu} give us a~smooth local solution $u$ near~$\Gamma^+$ of the same equation
$
{\rm d}H \wedge {\rm d}u = \tilde\omega - \omega$.
We will reach a~contradiction by showing that the above equation does not have a~solution which is twice differentiable at $\Gamma^+$. In the coordinates $(\sigma,\beta)$, the above equation reads
\begin{align*}
\partial_{\sigma} H (\sigma,\beta) \partial_{\beta} u (\sigma,\beta) -
\partial_{\beta} H (\sigma,\beta) \partial_{\sigma} u (\sigma,\beta)
= - \varepsilon r(\sigma) \psi'(\sigma),
\end{align*}
and, substituting the expression~\eqref{e:down} of the Hamiltonian $H=H_0$,
\begin{align*}
r'(\sigma) \cos \beta \partial_{\beta} u (\sigma,\beta) + r(\sigma) \sin \beta \partial_{\sigma} u (\sigma,\beta)
= - \varepsilon r(\sigma) \psi'(\sigma).
\end{align*}
Evaluation at $\beta=0$ gives us
$
r'(\sigma) \partial_{\beta} u (\sigma,0) = - \varepsilon r(\sigma) \psi'(\sigma)$,
and differentiation in $\sigma$ produces
\begin{align*}
r''(\sigma) \partial_{\beta} u (\sigma,0) + r'(\sigma) \partial_{\sigma\beta} u (\sigma,0) = - \varepsilon r'(\sigma) \psi'(\sigma) - \varepsilon r(\sigma) \psi''(\sigma).
\end{align*}
Since $\frac{m}{2}$ is a~degenerate critical point of $r$, evaluation at $\sigma = \frac{m}{2}$ gives us
$
0 = - \varepsilon r_{\max} \psi'' \left( { \frac{m}{2}} \right)$.
The above identity contradicts~\eqref{e:condonpsi} and concludes the proof of Proposition~\ref{p:notsmooth}.

\section{Unstable equators}
\label{s:unstable}

Let $g\in \mathcal{G}$ be a~metric with associated unimodal profile function $r\colon [0,m] \rightarrow [0,+\infty)$ such that $\max r = r(\sigma_{\max}) = 1$. Given $\beta\in\R/2\pi \Z$ with $\beta\neq\pm\pi/2$, we consider a~unit tangent vector~${v_{\beta}=(\sigma_{\max},\theta_0,\beta)}$ based at the equator and forming an angle $\beta$ with its positively oriented tangent vector, and the corresponding geodesic
$
\gamma(t) = (\sigma(t),\theta(t))
$
such that $\dot\gamma(0)=v_\beta$. The angular discrepancy function
$
\Delta(t,\beta) := \theta(t) - \theta(0) - t \cos \beta
$
satisfies $\Delta (t,0)=\Delta(t,\pi)=0$ for every $t\in \R$.
In Lemma~\ref{l:stability}, we proved that if the curvature of~$g$ is non-negative near the equator then
\[
\lim_{\beta\downarrow 0} \max_{t\in [0,\tau(v_\beta)]} | \Delta(t,\beta)| = 0.
\]
In this section, we show that if the curvature is allowed to be somewhere negative near the equator -- something which can happen only if the curvature has infinite order of vanishing at the equator -- then the above stability property of the equator over possibly large time intervals may fail. Indeed, we shall construct examples for which
\[
\limsup_{\beta\downarrow 0} \Delta(\tau_1(\beta),\beta) = +\infty,
\]
where, as in Section~\ref{s:continuous}, $\tau_1(\beta)$ denotes the first $t>0$ such that $\dot\gamma(t)$ is tangent to a~parallel (note that $2\tau_1(\beta)$ is the first positive time at which $\gamma$ meets the equator, and hence $2\tau_1(\beta)< \tau(v_{\beta})$). The same construction will also give us metrics for which $\tau(v_{\beta})$ diverges arbitrarily fast along suitable sequences $\beta_k\downarrow 0$.

We start by considering two smooth non-negative functions $f$ and $h$ on $[0,\infty)$ such that
\begin{enumerate}[(i)]\itemsep=0pt
\item $f'(s) \geq 0$ and $h'(s)>0$ for all $s>0$;
\item $f^{(n)} (0) = h^{(n)} (0) = 0$ for all integers $n\geq 0$;
\item $f'=0$ on the interval $\bigl[2^{-2k-1},2^{-2k}\bigr]$, for all integers $k\geq 1$;
\item \smash{$\lim_{k \rightarrow \infty} \frac{2^{2k} \sqrt{h(2^{-2k})}}{f(2^{-2k})} = 0$}.
\end{enumerate}

Here is a~possible construction of two functions with these properties. Let $\chi\in C^{\infty}(\R)$ be a~non-negative function with support in $(0,1)$ and integral 1. The non-negative function
\begin{gather}
\label{e:f}
f(s) := \int_0^s \sum_{j=1}^{\infty} j^{-j} \chi\bigl( 2^{2j} \bigl(r-2^{-2j}\bigr) \bigr){\rm d}r
\end{gather}
is smooth on $[0,+\infty)$ and satisfies properties (i), (ii), and (iii). Moreover,
\[
f\bigl(2^{-2k}\bigr) = \sum_{j=k+1}^{\infty} 2^{-2j} j^{-j} \geq 2^{-2(k+1)} (k+1)^{-(k+1)},
\]
and the function
$
h(s) := e^{-\frac{1}{s}}$ for $ s>0$, $ h(0):=0$,
which is smooth on $[0,+\infty)$ and satisfies (i), tends to zero so fast for $s\rightarrow 0$ that (iv) holds.

\begin{Proposition}
\label{p:esplode}
Let $r\colon [0,m] \rightarrow [0,+\infty)$ be a~unimodal profile function with maximum at~${\sigma_{\max}\in (0,m)}$ and such that
\begin{gather}
\label{e:adx}
r(\sigma_{\max} + s) = 1 - f(s) - h(s) \qquad \forall s\in [0,s_0],
\end{gather}
where $f$ and $h$ satisfy conditions $(i)$--$(iv)$ above and $0< s_0 < m-\sigma_{\max}$. Then
\[
\limsup_{\beta \downarrow 0} \Delta(\tau_1(\beta),\beta) = +\infty.
\]
\end{Proposition}

\begin{proof}
By the computations in the proof of Lemma~\ref{l:calcoli}, we have
\begin{align*}
 \Delta(\tau_1(\beta),\beta) &= \cos \beta \int_{\sigma_{\max}}^{\sigma(\tau_1(\beta))} \left( \frac{1}{r(s)} - r(s) \right) \frac{1}{\sqrt{r(s)^2 - \cos^2 \beta}}{\rm d}s \\ &= \cos \beta \int_{\sigma_{\max}}^{\sigma(\tau_1(\beta))} \frac{1+r(s)}{r(s) \sqrt{r(s) + \cos \beta}} \frac{1-r(s)}{\sqrt{r(s)-\cos \beta}}{\rm d}s.
\end{align*}
By the conservation of Clairaut's first integral and by the fact that $\dot\gamma(\tau_1(\beta))$ is tangent to a~parallel, we have
\begin{gather}
\label{e:chiaro}
r(\sigma(\tau_1(\beta))) = \cos \beta,
\end{gather}
and in the above integral the function $r$ takes values in the interval $[c,1]$ with $c:= \cos \beta$. As we are interested in the limit for $\beta\rightarrow 0$, we can assume that $c>0$ and from the estimate
\[
\frac{1+\rho}{\rho \sqrt{\rho+c}} \geq \frac{2}{ \sqrt{1+c}} \geq \sqrt{2} > 1 \qquad \forall \rho\in [c,1]
\]
we deduce
\begin{align*}
 \Delta(\tau_1(\beta),\beta) &\geq \cos \beta \int_{\sigma_{\max}}^{\sigma(\tau_1(\beta))} \frac{1-r(s)}{\sqrt{r(s)-\cos \beta}}{\rm d}s
 = \cos \beta \int_{\sigma_{\max}}^{\sigma(\tau_1(\beta))} \frac{1-r(s)}{\sqrt{r(s)- r(\sigma(\tau_1(\beta)))}}{\rm d}s,
 \end{align*}
where we used~\eqref{e:chiaro}.
Note that the positive function $S(\beta):= \sigma(\tau_1(\beta)) - \sigma_{\max}$ tends to zero for~${\beta\downarrow 0}$. If $\beta>0$ is so small that $S(\beta)\leq s_0$, then by~\eqref{e:adx} we have
\[
 \int_{\sigma_{\max}}^{\sigma(\tau_1(\beta))} \frac{1-r(s)}{\sqrt{r(s)- r(\sigma(\tau_1(\beta)))}}{\rm d}s =
\int_0^{S(\beta)} \frac{f(s) + h(s)}{\sqrt{f(S(\beta)) + h(S(\beta))-f(s)-h(s)}}{\rm d}s.
\]
If $k\in \N$ is large enough, we can find $\beta_k>0$ forming an infinitesimal sequence such that $S(\beta_k) = 2^{-2k}$. Using the fact that the function $f$ is constant on the interval $\bigl[2^{-2k-1},2^{-2k}\bigr]$, we obtain the lower bound
\begin{gather*}
\int_0^{S(\beta_k)} \frac{f(s) + h(s)}{\sqrt{f(S(\beta_k)) + h(S(\beta_k))-f(s)-h(s)}}{\rm d}s \\
\qquad\geq \int_{2^{-2k-1}}^{2^{-2k}} \frac{f(s)+h(s)}{\sqrt{h\bigl(2^{-2k}\bigr)-h(s)}}{\rm d}s \geq \int_{2^{-2k-1}}^{2^{-2k}} \frac{f(s)}{\sqrt{h\bigl(2^{-2k}\bigr)}}{\rm d}s = 2^{-2k-1} \frac{f\bigl(2^{-2k}\bigr)}{\sqrt{h\bigl(2^{-2k}\bigr)}}.
\end{gather*}
Therefore,
\[
\Delta(\tau_1(\beta_k),\beta_k) \geq \frac{\cos \beta_k}{2} \frac{f\bigl(2^{-2k}\bigr)}{2^{2k} \sqrt{h\bigl(2^{-2k}\bigr)}},
\]
and the latter quantity diverges to $+\infty$ for $k\rightarrow \infty$ because $f$ and $h$ are assumed to satisfy condition~(iv).
\end{proof}

\begin{Remark}
\label{r:nocontrol}
In Remark~\ref{r:control}, we proved a~universal asymptotic upper bound for~${\tau(v_{\beta})}$ for~${\beta\downarrow 0}$, which holds whenever the curvature near the equator is non-negative. The above profile functions can be built with $\tau(v_{\beta_k})$ diverging arbitrarily fast along suitable sequences $\beta_k\downarrow 0$. Indeed, let~${\omega\colon \left(0,\frac{\pi}{2}\right) \rightarrow (0,+\infty)}$ be an arbitrary monotonically decreasing function such that $\omega(\beta) \rightarrow +\infty$ for $\beta\downarrow 0$. Assume that $f$ and $h$ are smooth non-negative functions on $[0,+\infty)$ satisfying (i), (ii), (iii), and
\begin{enumerate}[(i)]\itemsep=0pt
\setcounter{enumi}{4}
\item $\lim_{k\rightarrow \infty} 2^{2k} \omega\bigl(\arccos \bigl(1-f\bigl(2^{-2k}\bigr)\bigr)\bigr) \sqrt{h\bigl(2^{-2k}\bigr)} = 0$.
\end{enumerate}
Once $f$ has been chosen -- for instance, as in~\eqref{e:f} -- we can indeed find an $h$ satisfying (ii) and converging to 0 so fast for $s\rightarrow 0$ that (v) holds. Choosing $\beta_k$ as in the proof of Proposition~\ref{p:esplode}, \eqref{e:chiaro} implies
\[
1- f\bigl(2^{-2k}\bigr) - h\bigl(2^{-2k}\bigr) = \cos \beta_k,
\]
and hence
\[
\cos \beta_k \leq 1 - f\bigl(2^{-2k}\bigr) \quad \Rightarrow \quad \beta_k \geq \arccos \bigl(1-f\bigl(2^{-2k}\bigr)\bigr),\]
which by the monotonicity of $\omega$ implies
$ \omega(\beta_k) \leq \omega\bigl( \arccos \bigl(1-f\bigl(2^{-2k}\bigr)\bigr)\bigr)$.
By~\eqref{e:tempo}, arguing as in the above proof we find for $k$ large enough
\begin{align*}
\tau_1(\beta_k) = \int_{\sigma_{\max}}^{\sigma(\tau_1(\beta_k))} \frac{r(s)}{\sqrt{r(s)^2 - \cos^2 \beta_k}}{\rm d}s &\geq \frac{1}{2} \int_{2^{-2k-1}}^{2^{-2k}} \frac{{\rm d}s}{\sqrt{h\bigl(2^{-2k}\bigr) - h(s)}} \geq \frac{2^{-2(k+1)}}{\sqrt{h\bigl(2^{-2k}\bigr)}},
\end{align*}
and hence
\[
\frac{\omega(\beta_k)}{\tau_1(\beta_k)} \leq 2^{2(k+1)} \omega\bigl( \arccos \bigl(1-f\bigl(2^{-2k}\bigr)\bigr)\bigr) \sqrt{h\bigl(2^{-2k}\bigr)}.
\]
By (v), the latter quantity is infinitesimal and hence
\[
\frac{\tau(v_{\beta_k})}{\omega(\beta_k)} \geq \frac{\tau_1(\beta_k)}{\omega(\beta_k)} \rightarrow +\infty
\]
for $k\rightarrow \infty$.
\end{Remark}

\renewcommand{\appendixname}{}
\appendix

\section{Determining a~function from the tangent lines to its graph}
\label{s:appA}

Given a~differentiable function $f\colon I\to\R$ on an interval $I\subset \R$ with non-empty interior, we denote by
\[
\Gamma(f):=\bigl\{(x,y)\in \R^2\ |\ x\in I, \, y=f(x)\bigr\}
\]
the graph of $f$ and by $\TT(f)$ the set of affine lines $\ell\subset\R^2$ that are tangent to $\Gamma(f)$. This appendix is devoted to the question of recovering the function $f$ from the set $\TT(f)$. Here, we are seeing~$\TT(f)$ as a~subset of the affine Grassmannian of the plane, without any further structure, and in particular we are not keeping track of the point at which every tangent line is tangent to~$\Gamma(f)$, the knowledge of which would make the question trivial.

The earliest reference to this question which we could find is~\cite{hor89}, where Horwitz proves that~$\TT(f)$ does indeed determine $f$ when $f$ is a~twice differentiable function on a~compact interval such that $f''$ has finitely many zeroes. In~\cite[Theorem 4.8]{ric97}, Richardson proves the same conclusion for $C^{1,1}$ functions on compact intervals, using some results from geometric measure theory. In this appendix, we provide an elementary argument for $C^2$ functions on arbitrary intervals.

\begin{Theorem}\label{t:T(f)}
Let $I\subset\R$ be an interval with non-empty interior. The graphs of two $C^2$ functions $f_1\colon I\to\R$ and $f_2\colon I\to\R$ have the same tangent lines, i.e., $\TT(f_1)=\TT(f_2)$ if and only if $f_1= f_2$.
\end{Theorem}

Our proof uses some properties of the Legendre transform in one variable, which we now recall. Given a~$C^1$ function $f\colon I \rightarrow \R$ such that $f'$ is strictly monotone (i.e., $f$ is either strictly convex or strictly concave), the Legendre dual of $f$ is the function
\[
f^* \colon\ J:= f'(I) \rightarrow \R, \qquad f^*(y) = xy - f(x),\qquad \mbox{where}\ f'(x)=y.
\]
The properties listed in the following proposition are well known, but we include their short proof for the sake of completeness.

\begin{Proposition}
If $f$ is $C^1$ with strictly monotone derivative, then $f^*$ is $C^1$, its derivative is the strictly monotone function $\bigl(f^*\bigr)' = (f')^{-1}$, and $f^{**} = f$.
\end{Proposition}

\begin{proof}
We assume that $f'$ is monotone increasing, the monotone decreasing case being analogous.
Let $h\colon I\times J\to\R$ be the function
$h(x,y):=yx-f(x)$.
Since $\partial_xh(x,y)=y-f'(x)$, the strictly concave function $x\mapsto h(x,y)$ has a~strict maximum at $(f')^{-1}(y)$. This, together with~${f^*(f'(x))=h(x,f'(x))}$, implies
\begin{align*}
 f^*(y)\geq h(x,y) \qquad\forall x\in I,\ y\in J.
\end{align*}
For each $y_1,y_2\in J$ with $y_1<y_2$ and $x_i:=(f')^{-1}(y_i)$, we have
\begin{align*}
 f^*(y_2)
 -
 f^*(y_1)
 & \leq
 h(x_2,y_2)- h(x_2,y_1)
 =
 (y_2-y_1)x_2,
\end{align*}
and analogously
\begin{align*}
 f^*(y_2)
 -
 f^*(y_1)
 & \geq
 h(x_1,y_2)- h(x_1,y_1)
 =
 (y_2-y_1)x_1.
\end{align*}
Therefore,
\[
\frac{f^*(y_2)-f^*(y_1)}{y_2-y_1}
\in
 [x_1,x_2]=(f')^{-1}([y_1,y_2]).
\]
This shows that $f^*$ is differentiable, with derivative $\bigl(f^*\bigr)'(y)=(f')^{-1}(y)$. Since $f'$ is a~continuous monotone function, its inverse is continuous, and therefore $f^*$ is $C^1$. Finally, since $f(x)=xy-f^*(y)$ whenever $y=f'(x)$, and thus whenever $x=\bigl(f^*\bigr)'(y)$, we conclude that $f^{**}=f$.
\end{proof}

We identify $\R^2$ with the space of affine lines in $\R^2$ that are not vertical, where $(m,q)\in\R^2$ corresponds to the affine line given by the equation
$
y = mx - q$.
By this identification, the set~$\TT(f)$ of tangent lines to the graph of the differentiable function $f\colon I \rightarrow \R$ is given by
\begin{align*}
 \TT(f):=\bigl\{ (f'(x),x f'(x)-f(x))\ \big|\ x\in I\bigr\}.
\end{align*}
Comparing the above expression with the definition of the Legendre transform, we see that if~${f\colon I \rightarrow \R}$ is $C^1$ and its derivative is strictly monotone, then
\begin{gather}
\label{e:legendre}
\mathcal{T}(f) = \Gamma\bigl(f^*\bigr).
\end{gather}

Now let $f\colon I\to\R$ be a~$C^2$ function on an arbitrary interval with non-empty interior. Let
$
\{I_{\alpha}\}_{\alpha\in A}
$
be the set of maximal intervals which are open in $I$ and on which $f'$ is strictly monotone. These intervals can be constructed by starting with some $x\in I$ at which $f''(x)\neq 0$ and then considering the union of all relatively open subintervals of $I$ containing $x$ on which $f'$ is strictly monotone. The intervals $I_{\alpha}$ are pairwise disjoint and the index set $A$ is at most countable, and it is empty if and only if $f$ is affine. Any $x\in I$ with $f''(x)\neq 0$ belongs to some interval $I_{\alpha}$.

Denote by $K$ the complement in $I$ of the union of all these intervals. Being closed in $I$, the set $K$ is a~countable union of compact sets, and so is $f'(K)$. Since $f''=0$ on $K$, by Sard's theorem the set $f'(K)$ has zero measure. We conclude that $f'(K)$ is a~countable union of closed sets with empty interior.

We write $A=A_+ \sqcup A_-$, where $A_+$ (resp.\ $A_-$) is the set of indices $\alpha\in A$ such that $f'$ is strictly increasing (resp.\ strictly decreasing) on $I_{\alpha}$. Let $I_{\alpha}$ and $I_{\beta}$ be distinct intervals with $\alpha,\beta\in A_+$ and $x<x'$ for every $x\in I_{\alpha}$ and $x'\in I_{\beta}$. Let $[a,b]\subset I$ be the closed interval separating $I_{\alpha}$ from~$I_{\beta}$, i.e., $a=\sup I_{\alpha}$, $b=\inf I_{\beta}$.

We claim that $a<b$ and $[a,b]$ contains an open interval $U$ which is either in $K$ or is of the form~${U=I_{\gamma}}$ for some $\gamma\in A_-$. Indeed, if by contradiction this is not the case then~$f''|_{[a,b]}$ must be non-negative (as any point at which~$f''$ is negative belongs to some interval $I_{\gamma}$ with~$\gamma\in A_-$) and its set of zeroes must have empty interior (as we are assuming that $K\cap [a,b]$ has empty interior). This implies that $f'$ is strictly increasing on $[a,b]$, and hence on the interval $I_{\alpha} \cup [a,b] \cup I_{\beta}$, contradicting the fact that $I_{\alpha}$ and $I_{\beta}$ are maximal intervals on which $f'$ is strictly increasing.

Similarly, if $\alpha$ and $\beta$ are both in $A_-$, then $a<b$ and $[a,b]$ contains an open interval $U$ which is either in $K$ or is of the form $U=I_{\gamma}$ for some $\gamma\in A_+$.

For every $\alpha\in A$, let $f_{\alpha}$ be the restriction of $f$ to $I_{\alpha}$, set $J_{\alpha}:= f_{\alpha}'(I_{\alpha})$, and denote by
\[
g_{\alpha} := f_{\alpha}^* \colon\ J_{\alpha} \rightarrow \R
\]
the Legendre dual of $f_{\alpha}$. By~\eqref{e:legendre}, we have
\begin{gather}
\label{e:legendre2}
\mathcal{T}(f_{\alpha}) = \Gamma(g_{\alpha}).
\end{gather}
The following lemma is the main ingredient in the proof of Theorem~\ref{t:T(f)}.

\begin{Lemma}\label{l:app}
Let $J\subset\R$ be an open interval and $g\colon J\to\R$ a~$C^1$ function with strictly monotone derivative, whose graph satisfies $\Gamma(g)\subset \TT(f)$. Then $J\subset J_\alpha$ and $g=g_\alpha|_J$ for some $\alpha\in A$.
\end{Lemma}

\begin{proof}
We provide the proof for $g'$ strictly increasing, the other case being analogous. For each~${\alpha\in A}$, we define
\begin{align*}
V_\alpha:=\bigl\{ y\in J\cap J_\alpha\ \big|\ g(y)=g_\alpha(y) \bigr\}.
\end{align*}
By~\eqref{e:legendre2}, this subset consists of those $y\in J$ such that $(y,g(y))\in\TT(f_\alpha)$.

We claim that there exists $\alpha\in A$ such that the set $V_\alpha$ has non-empty interior, and since $g'$ is strictly increasing, we must have $\alpha\in A_+$. Indeed, the inclusion $\Gamma(g) \subset \TT(f)$ implies that for every $y\in J$ there exists $x\in I$ such that
\[
(y,g(y)) = (f'(x),x f'(x)- f(x)).
\]
If $x$ is in $I_{\alpha}$ for some $\alpha\in A$, then $y=f_{\alpha}'(x)$ and $g(y)= x f_{\alpha}'(x) - f_{\alpha}(x) = g_{\alpha}(y)$, so $y$ is in $V_{\alpha}$. If $x$ is not in any of the $I_{\alpha}$, then it is in $K$ and hence $y=f'(x)$ is in $f'(K)$. Therefore,
\[
J \subset f'(K) \cup \bigcup_{\alpha \in A} V_{\alpha}.
\]
As seen above, $f'(K)$ is a~countable union of closed sets with empty interior. By Baire's theorem, the above inclusion implies that there exists $\alpha\in A$ for which $\overline{V_\alpha}$ has non-empty interior. Since~$V_{\alpha}$ is closed in the interval $J\cap J_{\alpha}$, $V_{\alpha}$ has non-empty interior as well. This proves our claim.

It remains to prove that $V_\alpha=J$. Assume by contradiction that the inclusion $V_\alpha\subset J$ is strict.

We claim that this implies that $J\setminus \overline{V}_{\alpha}$ is non-empty. Indeed, considering the interval $J':= J \cap J_{\alpha}$ we have $V_\alpha \subset J' \subset J$ and since the inclusion $V_\alpha\subset J$ is strict we may pick $x\in J\setminus V_{\alpha}$. If~${x\notin \overline{V}_{\alpha}}$ we are done, so we can assume that $x\in \overline{V}_{\alpha}$. The relative closedness of $V_\alpha$ in $J'$ gives~${V_{\alpha} = \overline{V}_{\alpha} \cap J'}$; since $x\in \overline{V}_\alpha$ but $x\notin V_{\alpha}$, this forces $x\notin J'$, and as $x\in J$ we conclude that~${x\notin J_{\alpha}}$. Therefore, $x$ lies strictly to one side of the interval $J_{\alpha}$, and after replacing everything by its image under the map $y\mapsto -y$ we may assume that $x< y$ for every $y\in J_{\alpha}$, and in particular for every $y\in V_{\alpha}$, since $V_{\alpha} \subset J_{\alpha}$. Hence $\overline{V}_{\alpha} \subset [x,+\infty)$. Finally, since $x$ belongs to the open set~$J$ we can find $x'\in J$ with $x'<x$. Then $x'\notin \overline{V}_{\alpha}$, i.e., $x'\in J\setminus \overline{V}_{\alpha}$. This proves the claim that $J\setminus \overline{V}_{\alpha}$ is non-empty.

By applying once again the claim of the second paragraph to the restriction of $g$ to an open interval in $J\setminus V_\alpha$, we find another $V_\beta$ with non-empty interior for some $\beta\in A_+\setminus\{\alpha\}$. Since both the indices $\alpha$ and $\beta$ belong to $A_+$, by our previous discussion the closed interval separating~$I_{\alpha}$ and $I_{\beta}$ contains an open interval $U$ which is either in $K$ or is of the form $U=I_\gamma$ for some~${\gamma\in A_-}$.

Since $g'(V_\alpha) = g_{\alpha}'(V_{\alpha}) \subset I_\alpha$ and $g'(V_\beta)= g_{\beta}'(V_{\beta}) \subset I_\beta$, by the intermediate value theorem we can find an open interval $W\subset J$ such that $g'(W)\subset U$. We now apply, for a~third time, the claim of the second paragraph to the restriction $g|_W$, and obtain an open interval $W'\subset W\cap J_{\delta}$ such that $g|_{W'}= g_\delta|_{W'}$ for some $\delta\in A_+$.
This implies that $g'(W') \subset I_\delta$, which contradicts the fact that
\[
g'(W') \subset g'(W) \subset U \subset K \cup \bigcup_{\gamma\in A_-} I_{\gamma}.
\]
This contradiction concludes the proof of the lemma.
\end{proof}

\begin{proof}[Proof of Theorem~\ref{t:T(f)}]
Let $\TT:=\TT(f_1)=\TT(f_2)$. Notice that this set is a~singleton $\TT=\{(m,q)\}$ if and only if $f_1(x)=f_2(x)=mx-q$ for all $x\in I$. Assume now that $\TT$ is not a~singleton. For each $i\in\{1,2\}$, we denote by
$
\{ I_{i,\alpha} \}_{\alpha\in A_i}
$
the family of maximal intervals which are open in $I$ and on which $f_i'$ is strictly monotone. Given~${\alpha\in A_i}$, we denote by
\[
g_{i,\alpha}\colon\ J_{i,\alpha}:= f'_i(I_{i,\alpha}) \rightarrow \R
\]
the Legendre dual of the restriction $f_i$ to $I_{i,\alpha}$.

For each open interval $(a,b)\subset I$ such that $f_1''|_{(a,b)}$ is nowhere vanishing, we have $(a,b)\subset I_{1,\alpha_1}$ for some $\alpha_1\in A_1$. The dual function $g_{1,\alpha_1}$ is $C^1$ and has strictly monotone derivative on the open interval $(c,d):=f'_1((a,b))$, and by~\eqref{e:legendre2} we have
\[
\Gamma(g_{1,\alpha_1}) = \TT(f_{1,\alpha_1}) \subset\TT=\TT(f_2).
\]
Therefore, Lemma~\ref{l:app} implies that $(c,d)\subset J_{2,\alpha_2}$ and $g_{1,\alpha_1}|_{(c,d)}=g_{2,\alpha_2}|_{(c,d)}$ for some $\alpha_2\in A_2$. Since the Legendre transform is an involution, we obtain
\[
f_1|_{(a,b)} = (g_{1,\alpha_1}|_{(c,d)})^* = (g_{2,\alpha_2}|_{(c,d)})^* = f_2|_{(a,b)}.
\]
By repeating the whole argument switching the roles of $f_1$ and $f_2$, we infer that $f_1=f_2$ on the open subset
\begin{align*}
 U := \bigl\{ x\in I\ |\ f_1''(x)\neq0 \bigr\} \cup \bigl\{ x\in I\ |\ f_2''(x)\neq0 \bigr\},
\end{align*}
and thus on its relative closure $\overline U\cap I$. Therefore, $f_1' = f_2'$ on $U$ and thus on $\overline U\cap I$.
On each maximal open interval $(a,b)\subset I\setminus\overline U$, we have $f_1''|_{(a,b)}=f_2''|_{(a,b)}= 0$, and therefore
\begin{align*}
 f_1\,(x)=f_1(y)+f'_1(y)(x-y)=f_2(y)+f'_2(y)(x-y)=f_2\,(x),\qquad\forall x\in(a,b),
\end{align*}
where $y\in\{a,b\}\cap \overline U\cap I$ (the latter set is not empty because we are assuming that $\TT$ is not a~singleton). This shows that $f_1=f_2$ on the whole $I$.
\end{proof}

\begin{Remark}
We do not know whether Theorem~\ref{t:T(f)} is true also for functions which are just differentiable, or even $C^1$.
\end{Remark}

\section[S\^{}1-invariant contact forms on three-manifolds]{$\boldsymbol{S^1}$-invariant contact forms on three-manifolds}
\label{s:appB}

Let $M$ be a~smooth (resp.\ analytic) closed oriented 3-manifold equipped with a~free $S^1$-action generated by the smooth (resp.\ analytic) vector field $V$. We denote this action by $(\theta,x) \mapsto \theta\cdot x$, with $\theta\in S^1$ and $x\in M$. The quotient $W:=M/S^1$ is a~smooth (resp.\ analytic) closed oriented surface, and we denote by $\pi\colon M \to W$ the quotient projection. Equivalently, $M$ is the total space of a~smooth (resp.\ analytic) principal $S^1$-bundle over $W$.

Let $\alpha$ be a~smooth (resp.\ analytic) $S^1$-invariant contact form on $M$. We assume that $\alpha$ is compatible with the orientation of $M$, meaning that $\alpha \wedge {\rm d}\alpha$ is a~positive volume form. We denote by $R=R_{\alpha}$ the associated Reeb vector field, defined by~$\alpha(R)=1$ and $\imath_{R} {\rm d}\alpha=0$, and by~$\phi_R^t$ its flow. The function $\alpha(V)$ is invariant under the flow of $V$, and therefore is of the form~${
 \alpha(V)=H\circ\pi}$,
for some smooth (resp.\ analytic) function $H\colon W\to\R$. Being $S^1$-equivariant, the Reeb flow of $\alpha$ descends to a~flow on $W$, which by the next proposition is the Hamiltonian flow of $H$ with respect to a~suitable symplectic form. Here, our sign convention is that the Hamiltonian vector field $X_H$ induced by a~Hamiltonian $H$ on the symplectic manifold~$(W,\omega)$ is defined by
$\imath_{X_H} \omega = {\rm d}H$,
and the flow of $X_H$ is denoted by $\phi_{X_H}^t$.

\begin{Proposition}\label{p:symplectic_quotient}
There exists a~unique symplectic form $\omega$ on $W$ such that
$\pi^*\omega=\imath_V (\alpha\wedge {\rm d}\alpha)$.
The Hamiltonian vector field $X_H$ is the projection of the Reeb vector field~$R$, i.e., \begin{align}
\label{e:RversusX}
 {\rm d}\pi(x)R(x)=X_H(\pi(x))\qquad\forall x\in M.
\end{align}
\end{Proposition}

\begin{proof}
We consider the $S^1$-invariant 2-form $\Omega:=\imath_{V} (\alpha\wedge {\rm d}\alpha)$. Since $V$ is nowhere vanishing, $\Omega$ is nowhere vanishing as well. Moreover $\imath_V \Omega=0$. This allows us to define a~2-form $\omega$ on~$W$ by
$
\omega(v_1,v_2):=\Omega(\tilde v_1,\tilde v_2)$, $\forall v_1,v_2\in T_y W$, $\forall y\in W$,
where $\tilde{v}_1$ and $\tilde{v_2}$ are arbitrary vectors in $TM$ based at the same point $x\in \pi^{-1}(y)$ and projecting to $v_1$ and $v_2$ by ${\rm d}\pi(x)$. The value~${\omega(v_1,v_2)}$ is well defined independently of the choices of the preimage $x$ and of the lifts $\tilde v_1$, $\tilde v_2$.
Notice that~$\omega$ is the only 2-form on $W$ satisfying $\pi^*\omega=\Omega$. The latter identity implies that
 $\omega$ is nowhere vanishing (i.e., symplectic).
Since $0=\Lie_V\alpha=\imath_V {\rm d}\alpha+{\rm d}(\alpha(V))$, we have for every $x\in M$, $v\in T_{\pi(x)} W$, and a~lift $\tilde{v}\in T_x M$ of $v$
\begin{align*}
\omega({\rm d}\pi(x)R(x),v)
&=
\Omega(R(x),\tilde v)
=
\alpha\wedge {\rm d}\alpha(V(x),R(x),\tilde v)
=
-{\rm d}\alpha(V(x),\tilde v)\\
&=
{\rm d}(\alpha(V))(x)\tilde v
=
{\rm d}H(\pi(x))v=\omega(X_H(\pi(x)),v),
\end{align*}
which implies that ${\rm d}\pi(x)R(x)=X_H(\pi(x))$.
\end{proof}

Note that the orientation induced by $\omega$ agrees with the orientation of $W$ as base of an $S^1$-bundle with oriented total space.

{\textbf{The main statement.}} We now specialize our discussion to the case in which $W$ is diffeomorphic to $S^2$ and $H$ is a~perfect Morse function, meaning that it has only two critical points, necessarily the maximizer $w^+$ and the minimizer $w^-$, and that these are non-degenerate. Then each level set $H^{-1}(c)$ with $c\in J := (\min H,\max H)$ is a~non-constant periodic orbit of $X_H$, and we denote by $T(c)$ its period. We orient the embedded loop $H^{-1}(c)$ by $X_H$.

The interval $J$ and the function $T$ are clearly invariant under $S^1$-equivariant conjugacies of the contact form. More precisely, let $\pi_1 \colon M_1 \rightarrow W_1$ and $\pi_2 \colon M_2 \rightarrow W_2$ be oriented principal $S^1$-bundles over closed surfaces diffeomorphic to the two-sphere, and $\alpha_1$ and $\alpha_2$ two orientation-compatible $S^1$-invariant contact forms on $M_1$ and $M_2$ with Hamiltonians $H_1$ and~$H_2$ which are perfect Morse functions on $W_1$ and $W_2$. If there exists an $S^1$-equivariant diffeomorphism~${\varphi\colon M_1 \rightarrow M_2}$ such that $\varphi^* \alpha_2 = \alpha_1$, then, setting $J_j := (\min H_j,\max H_j)$ and denoting by $T_j\colon J_j \rightarrow (0,+\infty)$ the functions defined above, for $j=1,2$, we have
\begin{gather}
\label{e:coincide}
J_1 = J_2, \qquad T_1 = T_2.
\end{gather}
The aim of this appendix is to show that, conversely, condition~\eqref{e:coincide}, together with the assumption that the principal bundles $\pi_1$ and $\pi_2$ have the same Euler number, implies that the contact forms $\alpha_1$ and $\alpha_2$ are $S^1$-equivariantly conjugate.

\begin{Theorem}
\label{mainappB}
Consider smooth $($resp.\ analytic$)$ oriented principal $S^1$-bundles $\pi_1 \colon M_1 \rightarrow W_1$ and $\pi_2 \colon M_2 \rightarrow W_2$ with the same Euler number over surfaces $W_1$ and $W_2$ which are diffeomorphic to $S^2$.
Let $\alpha_1$ and $\alpha_2$ be smooth $($resp.\ analytic$)$ orientation-compatible $S^1$-invariant contact forms inducing Hamiltonians $H_1$ and $H_2$ on $W_1$ and $W_2$ which are perfect Morse functions. Denote by $J_j$ and $T_j$, $j=1,2$, the interval $(\min H_j,\max H_j)$ and the period function of the orbits~$H_j^{-1}(c)$. Then there exists an $S^1$-equivariant smooth $($resp.\ analytic$)$ diffeomorphism~${\varphi\colon M_1 \rightarrow M_2}$ such that $\varphi^* \alpha_2 = \alpha_1$ if and only if \eqref{e:coincide} holds.
\end{Theorem}

The remaining part of this appendix is devoted to the proof of the ``if'' part of Theorem~\ref{mainappB}.

{\textbf{Moser's argument.}} Assume that $\alpha$ and $\beta$ are smooth (resp.\ analytic) $S^1$-invariant contact forms on $M$ inducing the same Hamiltonian $H$ and the same symplectic form $\omega$ on the quotient manifold $W$, namely
\begin{gather}
\label{e:ass1}
\alpha(V)=\beta(V)=\pi^*H, \\
\label{e:ass2}
\imath_V (\alpha\wedge {\rm d}\alpha)=\imath_V (\beta\wedge {\rm d}\beta)=\pi^*\omega.
\end{gather}
For now, we do not need further assumptions on the surface $W$ or the Hamiltonian $H$. The fact that $\alpha-\beta$ is $S^1$-invariant and vanishes on the vector field $V$ generating the $S^1$-action implies that
\begin{gather}
\label{e:defn-eta}
\alpha - \beta = \pi^* \eta,
\end{gather}
for a~unique smooth (resp.\ analytic) one-form $\eta$ on $W$. In the next proposition, we show that if a~suitable cohomological equation has a~solution then $\alpha$ is the pullback of $\beta$ by a~suitable $S^1$-equivariant diffeomorphism.

\begin{Proposition}
\label{p:Moser}
Let $\alpha$ and $\beta$ be $S^1$-invariant smooth $($resp.\ analytic$)$ contact forms on $M$ satisfying~\eqref{e:ass1} and~\eqref{e:ass2}. If there exists a~smooth $($resp.\ analytic$)$ function $h\colon W \rightarrow \R$ such that
\begin{gather}
\label{e:cohomological}
{\rm d}h(X_H)= \eta(X_H),
\end{gather}
where the one-form $\eta$ on $W$ is defined by~\eqref{e:defn-eta}, then there exists a~smooth $($resp.\ analytic$)$ $S^1$-equivariant diffeomorphism $\psi\colon M\to M$ isotopic to the identity such that $\psi^*\beta=\alpha$.
\end{Proposition}

\begin{proof}
The proof is based on Moser's isotopy argument. Since $R_{\alpha}$ and $R_{\beta}$ project to the same vector field $X_H$, they differ by a~vector field which is parallel to $V$. Taking also the $S^1$-invariance into account, we deduce that
$
R_{\beta} - R_{\alpha} = \bigl(\pi^*f\bigr) V$,
for some smooth (resp.\ analytic) function $f\colon W \rightarrow \R$.

For $t\in[0,1]$, we set $\alpha_t:=t\beta+(1-t)\alpha$. We claim each $\alpha_t$ is a~contact form with Reeb vector field
\begin{gather}
\label{e:Ralphat}
R_{\alpha_t}=tR_\beta+(1-t)R_\alpha=R_\alpha+t\bigl(\pi^*f\bigr)V.
\end{gather}
Indeed, by $\Lie_V\alpha=\Lie_V\beta=0$ and~\eqref{e:ass1}, we have $\imath_{V} {\rm d}\alpha=\imath_V {\rm d}\beta=-\pi^*({\rm d}H)$.
Therefore, the identities
$\alpha_t(V)=\pi^*H$, $\imath_V {\rm d}\alpha_t=-\pi^*({\rm d}H)$, and~\eqref{e:ass2} imply
\begin{align*}
\imath_V (\alpha_t\wedge {\rm d}\alpha_t)
=
\bigl(\pi^*H\bigr) {\rm d}\alpha_t+\alpha_t\wedge \pi^*({\rm d}H)
=
t \imath_V (\beta\wedge {\rm d}\beta)+(1-t) \imath_V (\alpha\wedge {\rm d}\alpha)
=
\pi^*\omega.
\end{align*}
In particular, $\alpha_t\wedge {\rm d}\alpha_t$ is nowhere vanishing, i.e., $\alpha_t$ is a~contact form. Since
\begin{gather*}
\alpha(R_\beta) =1+\pi^*(fH),\qquad
\beta(R_\alpha) =1-\pi^*(fH),\\
\imath_{R_{\alpha}} {\rm d}\beta =-\pi^*(f) \imath_V {\rm d}\beta=-\pi^*(f) \imath_V {\rm d}\alpha = -\imath_{R_\beta} {\rm d}\alpha,
\end{gather*}
we infer
\begin{gather*}
 \alpha_t(R_{\alpha_t}) =t^2\beta(R_\beta)+t(1-t)(\beta(R_\alpha)+\alpha(R_\beta))+(1-t)^2\alpha(R_\alpha)=1,\\
\imath_{R_{\alpha_t}} {\rm d}\alpha_t
=
t^2 \imath_{R_\beta} {\rm d}\beta +t(1-t)(\imath_{R_\alpha} {\rm d}\beta+\imath_{R_\beta} {\rm d}\alpha)+(1-t)^2 \imath_{R_\alpha} {\rm d}\alpha=0,
\end{gather*}
which confirms that $R_{\alpha_t}$ is the Reeb vector field of $\alpha_t$.

We now look for an $S^1$-equivariant isotopy $\psi_t\colon M \to M$ such that $\psi_0=\id$ and $\psi_t^*\alpha_t=\alpha$. An $S^1$-equivariant time-dependent vector field $Y_t$ generates such an isotopy if and only if
\begin{align}
\label{e:Moser}
 \imath_{Y_t} {\rm d}\alpha_t+{\rm d}(\alpha_t(Y_t))+\beta-\alpha=0.
\end{align}
We can uniquely write any time-dependent vector field as $Y_t=h_tR_{\alpha_t}+Z_t$ for some time-dependent function $h_t\colon M \to\R$ and for some time-dependent vector field $Z_t\in\ker(\alpha_t)$. Since $\imath_{Y_t} {\rm d}\alpha_t
=
\imath_{Z_t} {\rm d}\alpha_t$ and ${\rm d}(\alpha_t(Y_t))={\rm d}h_t$, we infer that~\eqref{e:Moser} is satisfied if and only if
$\imath_{Z_t} {\rm d}\alpha_t=\alpha-\beta-{\rm d}h_t$, namely if and only if
\begin{align}\label{e:Moser2}
\imath_{Z_t} {\rm d}\alpha_t|_{\ker(\alpha_t)}
=
(\alpha-\beta-{\rm d}h_t)|_{\ker(\alpha_t)},
\qquad
(\alpha-\beta)(R_{\alpha_t})={\rm d}h_t(R_{\alpha_t}).
\end{align}
Since ${\rm d}\alpha_t|_{\ker(\alpha_t)}$ is non-degenerate, for any choice of $h_t$, the first of these two equations is satisfied by a~unique vector field $Z_t$, which is smooth (resp.\ analytic). By~\eqref{e:defn-eta}, \eqref{e:Ralphat}, and~\eqref{e:RversusX}, we have
\[
(\alpha - \beta) (R_{\alpha_t}) = \bigl(\pi^* \eta\bigr)(R_{\alpha_t}) = \eta(X_H).
\]
By~\eqref{e:Ralphat} and~\eqref{e:RversusX},
\[
{\rm d}\bigl(\pi^* h\bigr)(R_{\alpha_t}) = {\rm d}\bigl(\pi^* h\bigr)(R_{\alpha}) = {\rm d}h(X_H),
\]
where $h\colon W\rightarrow \R$ solves the cohomological equation~\eqref{e:cohomological}. We conclude that the second equation in~\eqref{e:Moser2} is satisfied by the time-independent smooth (resp.\ analytic) function $h_t:= \pi^* h$.
Finally, the first equation in~\eqref{e:Moser2} implies that $Z_t$ is $S^1$-equivariant, and so are $Y_t$ and $\psi_t$.
This concludes the proof.
\end{proof}

{\textbf{Perfect Hamiltonians on the 2-sphere.}} We now assume that $H\colon S^2 \rightarrow \R$ is a~smooth (resp.\ analytic) perfect Morse function, with critical points $w^+$ and $w^-$ such that
\[
H(w^-) = \min H, \qquad H\bigl(w^+\bigr) = \max H,
\]
and $\omega$ is a~smooth (resp.\ analytic) symplectic form on $S^2$. At $w^+$ and $w^-$ we have smooth (resp.\ analytic) coordinates $(x,y)$ with respect to which $\omega$ and $H$ take the form
\begin{gather}
\label{e:dm}
\omega ={\rm d}x \wedge dy = r{\rm d}r\wedge {\rm d}\theta \qquad \mbox{and} \quad H = h\bigl(r^2\bigr),
\end{gather}
where $(r,\theta)$ are polar coordinates on the $xy$-plane and $h\colon [0,\epsilon) \rightarrow \R$ is a~smooth (resp.\ analytic) function such that $h'$ is nowhere vanishing on $[0,\epsilon)$. The existence of these coordinates can be seen as the two-dimensional case of the isochore Morse Lemma of Colin de Verdière and Vey (see~\cite{cdvv79} for the smooth category and~\cite{vey77} for the analytic one), or as a~special case of the symplectic Morse Lemma of Eliasson and Vey (see~\cite{eli90} for the smooth category and~\cite{vey78} for the analytic category). Here, we will refer to such coordinates as Darboux--Morse coordinates.

In the next lemma, we show that the function $T$ mapping $c$ to the period of the orbit $H^{-1}(c)$ has a~smooth (resp.\ analytic) extension to the closure of the interval $J=(\min H, \max H)$, and that it is the derivative of the area of sublevels of $H$.

\begin{Lemma}\label{l:extT}
\emergencystretch=2em 
The function $T$ has a~smooth $($resp.\ analytic$)$ positive extension to $\overline{J} = [ \min H,\max H]$ and
\begin{gather}
\label{e:action}
\int_{H^{-1}([\min H,c])} \omega = \int_{\min H}^c T(s){\rm d}s
\end{gather}
for every $c\in \overline{J}$.
\end{Lemma}

\begin{proof}
The smoothness (resp.\ analyticity) of $T$ on $J$ and the validity of~\eqref{e:action} on $J$ are readily seen by using action-angle coordinates on the annulus $S^2 \setminus \{w^+,w^-\}$.
In Darboux--Morse coordinates~\eqref{e:dm} near $w^+$ or $w^-$, we have $H=h\bigl(r^2\bigr)$ and the Hamiltonian vector field takes the form
$
X_H = - 2 h'\bigl(r^2\bigr) \partial_{\theta}$,
where $h\colon [0,\epsilon) \rightarrow \R$ is smooth (resp.\ analytic) with $h'(0)\neq 0$.
Therefore, the orbit with $H=h\bigl(r^2\bigr)=c$, $0<r<\sqrt{\epsilon}$, has period
\[
T(c) = \frac{2\pi r}{ 2 r |h'\bigl(r^2\bigr)| } = \pi |f'(c)|,
\]
where $f$ is the inverse of $h$, which is smooth (resp.\ analytic) near $H\bigl(w^{\pm}\bigr)$ and has non-vanishing derivative at this point because $h'(0)\neq 0$. Therefore, $T$ has a~smooth (resp.\ analytic) positive extension to $\overline{J}$. By the continuity of both sides, \eqref{e:action} holds on $\overline{J}$.
\end{proof}

The next result is well-known, but we include a~proof for the sake of completeness.

\begin{Proposition}
\label{l:solve-coho}
Let $\eta$ be a~smooth $($resp.\ analytic$)$ one-form on $S^2$. Then the cohomological equation
\begin{gather}
\label{e:coho}
{\rm d}h(X_H) = \eta(X_H)
\end{gather}
has a~smooth $($resp.\ analytic$)$ solution $h\colon S^2 \rightarrow \R$ if and only if
\begin{gather}
\label{e:cohoass}
\int_{H^{-1}(c)} \eta = 0 \qquad \forall c\in J.
\end{gather}
\end{Proposition}

\begin{proof}
The existence of a~solution $h$ to~\eqref{e:coho} implies~\eqref{e:cohoass} because in this case the left-hand side of~\eqref{e:cohoass} is the integral over $H^{-1}(c)$ of the exact one-form ${\rm d}h$. We prove the converse.
By Lemma~\ref{l:extT}, the function
\[
A(c):= \int_{H^{-1}([\min H,c])} \omega
\]
is smooth (resp.\ analytic) on $\overline{J}$ and satisfies $A'=T$. The function $K:= \frac{1}{2\pi} A\circ H$ is smooth (resp.\ analytic) on $S^2$ and its Hamiltonian vector field is
\[
X_K = \frac{T\circ H}{2\pi} X_H.
\]
Therefore, the flow $\phi^t_{X_K}$ of $X_K$ generates a~smooth (resp.\ analytic) $S^1$-action on $S^2$ having the same orbits of the flow of $X_H$. The function
\[
E(t,w) := \int_{\phi_{X_K}^{[0,t]} (w)} \eta = \int_0^t \eta\bigl(X_K\bigl(\phi^s_{X_K}(w)\bigr)\bigr){\rm d}s
\]
is smooth (resp.\ analytic) on $\R\times S^2$ and satisfies the cocycle identity
\begin{gather}
\label{e:cocycle}
E(t+s,w)=E(t,w) + E\bigl(s,\phi^t_{X_K}(w)\bigr) \qquad \forall t,s\in \R, \ \forall w\in S^2.
\end{gather}
By~\eqref{e:cohoass}, $E$ is $2\pi$-periodic in the first variable. We claim that the smooth (resp.\ analytic) function $h\colon S^2 \rightarrow \R$ defined by
\[
h(w):= - \frac{1}{2\pi} \int_0^{2\pi} E(s,w){\rm d}s
\]
solves the cohomological equation~\eqref{e:coho}. Indeed, by~\eqref{e:cocycle} and the $2\pi$-periodicity of $E$ we have for every $w\in S^2$ and $t\in \R$
\begin{align*}
h\bigl(\phi^t_{X_K} (w)\bigr)&=-\frac{1}{2\pi}\int_0^{2\pi}\bigl(E(t+s,w)-E(t,w)\bigr){\rm d}s
\\ &= E(t,w)-\frac{1}{2\pi}\int_t^{t+2\pi} E(s,w){\rm d}s = E(t,w) +h(w),
\end{align*}
and hence
\[
{\rm d}h(X_K(w)) =\frac{{\rm d}}{{\rm d}t}\Bigl|_{t=0}h\bigl(\phi^t_{X_K} (w)\bigr)= \frac{{\rm d}}{{\rm d}t}\Bigl|_{t=0} E(t,w) = \eta(X_K(w)).
\]
Since the vector fields $X_H$ and $X_K$ are parallel, we conclude that~\eqref{e:coho} holds.
\end{proof}

The next lemma (which is essentially trivial in the smooth setting) shows that the two critical points of the Hamiltonian $H$ can be joined by a~smooth (resp.\ analytic) arc which is transverse to the level sets and whose restriction to Darboux--Morse charts near the critical points is the germ of a~centrally symmetric curve.

\begin{Lemma}
\label{l:arc}
There exists a~smooth $($resp.\ analytic$)$ connected compact 1-di\-men\-sion\-al submanifold $A \subset S^2$ such that $\partial A =\{w^+,w^-\}$, $A \setminus \partial A$ is transverse to the flow of $X_H$ and meets each non-constant orbit exactly once, and in Darboux--Morse coordinates on neighborhoods~$U^+$ and~$U^-$ of $w^+$ and $w^-$ we have
\[
A \cap U^{\pm} = \bigl\{\bigl(r,\theta^{\pm}\bigl(r^2\bigr)\bigr) \mid 0\leq r < \sqrt{\epsilon}\bigr\},
\]
for suitable smooth $($resp.\ analytic$)$ functions $\theta^{\pm}\colon [0,\epsilon) \rightarrow \R$.
\end{Lemma}

\begin{proof}
In the smooth case, we can choose $\theta^{\pm}$ to be constant and join the corresponding arcs smoothly by a~curve which is transverse to $X_H$. Here, we prove the analytic case, which is more delicate. By Lemma~\ref{l:extT}, the involution
\[
S\colon\ S^2 \rightarrow S^2, \qquad z\mapsto \phi_{X_H}^{T(H(z))/2}(z)
\]
is an analytic diffeomorphism with fixed point set $\big\{w^+,w^-\big\}$. We choose an analytic Riemannian metric $g$ on $S^2$ such that
\begin{gather}
g(u,v) ={\rm d}^2 H\bigl(w^-\bigr)[u,v] \qquad \forall u,v\in T_{w^-} S^2,\nonumber \\
g(u,v) = -{\rm d}^2 H\bigl(w^+\bigr)[u,v] \qquad \forall u,v\in T_{w^+} S^2.\label{e:diff2}
\end{gather}
Here is an explicit construction of such an analytic metric: viewing $S^2$ as unit sphere in $\R^3$, we choose $g_-$ and $g_+$ positive bilinear symmetric forms on $\R^3$ such that $g_-$ restricts to ${\rm d}^2 H\bigl(w^-\bigr)$ on $T_{w^-} S^2$ and $g_+$ restricts to $-{\rm d}^2 H\bigl(w^+\bigr)$ on $T_{w^+} S^2$, and for every $w\in S^2$ set
\[
g_w := ( \chi(w) g_- + (1-\chi(w)) g_+ )|_{T_w S^2\times T_w S^2}, \qquad \mbox{where}\ \chi(w) := \frac{|w-w^+|^2}{|w-w^-|^2 + |w-w^+|^2}.
\]
Up to the replacement of $g$ by $\frac{1}{2} \bigl(g+S^*g\bigr)$, we may assume that $g$ is $S$-symmetric. Let $\operatorname{grad} H$ be the gradient vector field of $H$ with respect to the metric $g$, and choose $A$ to be the union of a~non-constant flow line of $\operatorname{grad} H$ and $\bigl\{w^+,w^-\bigr\}$. Then $A \setminus \bigl\{w^+,w^-\bigr\}$ is an analytic connected 1-dimensional submanifold of $S^2$ transverse to the flow of $X_H$ and meeting each non-constant orbit exactly once. We must check that $A$ is analytic up to the boundary and has the required form in $U^+$ and $U^-$.

In the Darboux--Morse coordinates near $w^{\pm}$, $S$ coincides with the central symmetry $S_0$ mapping $(x,y)$ to $-(x,y)$. In these coordinates, both $g$ and $H$ are invariant under $S_0$ and so is $\operatorname{grad} H$. By~\eqref{e:diff2}, the $g$-Hessian of $H$ at the critical point $w^{\pm}$ is $\mp \mathrm{id}$, so the linearization of the gradient flow of $H$ at $w^{\pm}$ is the linear flow $(t,w) \mapsto e^{\mp t} w$, $w\in \R^2$. By Poincar\'e's theorem on the linearization of analytic vector fields near a~hyperbolic equilibrium, near $w^+$ and $w^-$ the flow of $\operatorname{grad} H$ is conjugate to the linearized flow by a~unique analytic conjugacy which is tangent to the identity. This conjugacy is equivariant with respect to $S_0$ -- this follows from its uniqueness -- and we conclude that in Darboux--Morse coordinates the union of each gradient flow line with its image by $S_0$ and with the origin is a~centrally symmetric analytic 1-dimensional submanifold containing the origin. This proves that $A$ is analytic up to the boundary and that the arcs~${A\cap U^+}$ and $A\cap U^-$ have the desired form.
\end{proof}

The next proposition shows that the period function completely determines the Hamiltonian system on the 2-sphere induced by a~perfect Morse function.

\begin{Proposition}
\label{p:aac-glob}
Let $H_1$, $H_2$ be smooth $($resp.\ analytic$)$ perfect Morse functions on $S^2$ such that
\[
(\min H_1,\max H_1) = (\min H_2,\max H_2) = J.
\]
Let $\omega_1$ and $\omega_2$ be smooth $($resp.\ analytic$)$ symplectic forms on $S^2$ and denote by $T_1\colon J \rightarrow (0,+\infty)$ and $T_2\colon J \rightarrow (0,+\infty)$ the period functions induced by the Hamiltonian systems $(H_1,\omega_1)$ and~${(H_2,\omega_2)}$. If $T_1=T_2$ then there exists a~smooth $($resp.\ analytic$)$ diffeomorphism $\zeta\colon S^2 \rightarrow S^2$ such that
$
\zeta^* \omega_2 = \omega_1$, $ \zeta^* H_2 = H_1$.
\end{Proposition}

\begin{proof}
Denote by $w^+_1$ and $w^-_1$ the critical points of $H_1$, and by $w^+_2$ and $w^-_2$ the critical points of $H_2$, where $H_1\bigl(w_1^+\bigr)=H_2\bigl(w_2^+\bigr)=\sup J$ and $H_1\bigl(w_1^-\bigr)=H_2\bigl(w_2^-\bigr)=\inf J$. Let $A_1$ and $A_2$ be smooth (resp.\ analytic) connected 1-dimensional submanifolds of $S^2$ with $\partial A_1 = \{w^+_1,w^-_1\}$, $\partial A_2 = \bigl\{w^+_2,w^-_2\bigr\}$, and satisfying the other conditions listed in Lemma~\ref{l:arc} with respect to the vector fields $X_{H_1}$ and $X_{H_2}$. Denote by $\zeta\colon A_1 \rightarrow A_2$ the map sending the intersection of $A_1$ with~$H_1^{-1}(c)$ to the intersection of $A_2$ with $H_2^{-1}(c)$, for every $c\in \overline{J}$. Since $A_1$ and $A_2$ are transverse to the level sets of $H_1$ and $H_2$ away from their boundary, the map $\zeta$ restricts to a~smooth (resp.\ analytic) diffeomorphism from $A_1\setminus \partial A_1$ to $A_2\setminus \partial A_2$.

The fact that the non-constant orbits of $X_{H_1}$ and $X_{H_2}$ meet $A_1$ and $A_2$ exactly once, together with the assumption $T_1=T_2$,
implies that the map $\zeta$ extends uniquely to a~map from $S^2$ to $S^2$ which is a~conjugacy from the flow of $X_{H_1}$ to the one of $X_{H_2}$. We denote this extension by the same symbol $\zeta$. Since Hamiltonian flows preserve the Hamiltonian and since $X_{H_1}$ and $X_{H_2}$ are transverse to $A_1\setminus \partial A_1$ and $A_2\setminus \partial A_2$, the map $\zeta$ satisfies $\zeta^* H_2 = H_1$ and restricts to a~smooth (resp.\ analytic) diffeomorphism from $S^2\setminus \bigl\{w^+_1,w^-_1\bigr\}$ to $S^2\setminus \bigl\{w^+_2,w^-_2\bigr\}$ such that $\zeta^* X_{H_2} = X_{H_1}$.

Moreover, $\zeta\bigl(w_1^{\pm}\bigr) = w_2^{\pm}$ and it remains to check that $\zeta$ is a~smooth (resp.\ analytic) diffeomorphism near $w_1^+$ and $w_1^-$. Indeed, once this is proven, the identity $\zeta^* X_{H_2} = X_{H_1}$ and the fact that we are in dimension two imply that $\zeta^* \omega_2 = \omega_1$.

Let $(w_1,w_2)$ be either $\bigl(w_1^+,w_2^+\bigr)$ or $(w_1^-,w_2^-)$ and consider Darboux--Morse coordinates on neighborhoods $U_1$ of $w_1$ and $U_2$ of $w_2$ such that
$
\omega_1 = r{\rm d}r \wedge {\rm d}\theta$, $ H_1 = h_1\bigl(r^2\bigr)$, $ \omega_2 = r{\rm d}r \wedge {\rm d}\theta$, $ H_2 = h_2\bigl(r^2\bigr)$,
for some smooth (resp.\ analytic) functions $h_1\colon [0,\epsilon) \rightarrow \R$ and $h_2\colon [0,\epsilon) \rightarrow \R$ with~$h_1'$ and $h_2'$ never vanishing, and
\[
A_1 \cap U_1 = \bigl\{\bigl(r,\theta_1\bigl(r^2\bigr)\bigr) \mid 0\leq r < \sqrt{\epsilon}\bigr\}, \qquad
A_2 \cap U_2 = \bigl\{\bigl(r,\theta_2\bigl(r^2\bigr)\bigr) \mid 0\leq r < \sqrt{\epsilon}\bigr\},
\]
for suitable smooth (resp.\ analytic) functions $\theta_1\colon [0,\epsilon) \rightarrow \R$ and $\theta_2\colon [0,\epsilon) \rightarrow \R$ (see Lemma~\ref{l:arc}). Denote by $f_1$ and $f_2$ the inverses of $h_1$ and $h_2$, which are smooth (resp.\ analytic). Then
\[
H_1(r,\theta) = H_2\bigl(\sqrt{f_2 \circ h_1\bigl(r^2\bigr)},\theta'\bigr) \qquad \forall r\in [0,\sqrt{\epsilon}), \ \forall \theta,\theta'\in \R/2\pi \Z,
\]
and hence $\zeta$ maps the point $\bigl(r,\theta_1\bigl(r^2\bigr)\bigr)\in A_1 \cap U_1$ to the point $\bigl(s,\theta_2\bigl(s^2\bigr)\bigr)\in A_2 \cap U_2$ with
$
s^2 = f_2\circ h_1\bigl(r^2\bigr)$.
The assumption $T_1=T_2$ implies that the flows of $X_{H_1}$ and $X_{H_2}$ move the points on the circles of radius $r$ and $s$ with the same constant angular velocity. Therefore, the expression of the map~$\zeta$ in these local coordinates is
\[
\zeta(r,\theta) = \bigl( \sqrt{ f_2\circ h_1\bigl(r^2\bigr) }, \theta + \theta_2\bigl(f_2\circ h_1\bigl(r^2\bigr)\bigr) - \theta_1\bigl(r^2\bigr)\bigr).
\]
The above expression shows that $\zeta$ is smooth (resp.\ analytic) on $U_1$. The expression for the inverse of $\zeta$ is obtained from the above one by exchanging the role of the indices $1$ and $2$, so~$\zeta^{-1}$ is smooth (resp.\ analytic) on $U_2$. We conclude that $\zeta\colon S^2 \rightarrow S^2$ is a~smooth (resp.\ analytic) diffeomorphism.
\end{proof}

\begin{proof}[Proof of Theorem~\ref{mainappB}] Since $J_1=J_2$ and $T_1=T_2$, Proposition~\ref{p:aac-glob} gives us a~smooth (resp.\ analytic) diffeomorphism $\zeta\colon W_1 \rightarrow W_2$ such that
$
\zeta^* \omega_2 = \omega_1$, $ \zeta^* H_2 = H_1$.
The pullback bundle $\zeta^* \pi_2$ is a~principal $S^1$-bundle over~$W_1$ with the same Euler number as the principal $S^1$-bundle~$\pi_1$. Since the Euler number classifies principal $S^1$\nobreakdash-bundles, these bundles are isomorphic and by composition we obtain a~smooth (resp.\ analytic) bundle-isomorphism $\tilde\zeta: M_1 \rightarrow M_2$ lifting the map $\zeta$.

Then $\alpha:= \alpha_1$ and $\beta:= \tilde\zeta^* \alpha_2$ are $S^1$-invariant contact forms on the same 3-manifold $M:= M_1$ inducing the same Hamiltonian $H:= H_1$ and the same symplectic form $\omega:= \omega_1$ on the quotient manifold $W:= W_1 = M/S^1\cong S^2$. We set $\pi:= \pi_1$, and denote by $V$ the vector field generating the $S^1$-action on $M$ and by $w^+$ and $w^-$ the critical points of the perfect Morse function $H$ on~$W$. By construction
\begin{gather}
\label{e:summarize}
\imath_V \alpha = \imath_V \beta = \pi^*H, \qquad \imath_V (\alpha\wedge {\rm d}\alpha) = \imath_V (\beta \wedge {\rm d}\beta) = \pi^* \omega,
\end{gather}
The first of these identities, together with $\mathcal{L}_V \alpha = \mathcal{L}_V \beta = 0$, implies
\begin{gather}
\label{e:uguali}
\imath_V {\rm d}\alpha = \imath_V {\rm d}\beta = - \pi^* {\rm d}H.
\end{gather}
Let $\eta$ be the unique smooth (resp.\ analytic) one-form on $W$ such that
\begin{gather}
\label{e:defneta2}
\alpha-\beta = \pi^* \eta.
\end{gather}
By Proposition~\ref{p:Moser}, it suffices to show that the cohomological equation
\begin{gather}
\label{e:cohoeq}
{\rm d}h(X_H) = \eta(X_H)
\end{gather}
has a~smooth (resp.\ analytic) solution $h\colon W \rightarrow \R$. Indeed, if this is the case then Proposition~\ref{p:Moser} gives us a~smooth (resp.\ analytic) $S^1$-equivariant diffeomorphism $\psi\colon M\to M$ such that $\psi^*\beta=\alpha$, and the composition $\varphi:= \tilde\zeta\circ \psi$ is a~smooth (resp.\ analytic) $S^1$-equivariant diffeomorphism from~$M_1$ to $M_2$ such that $\varphi^* \alpha_2 = \alpha_1$.

First note that~\eqref{e:summarize}, \eqref{e:uguali}, and~\eqref{e:defneta2} imply the chain of identities
\begin{align*}
0 &= \imath_V (\alpha\wedge {\rm d}\alpha) - \imath_V (\beta\wedge {\rm d}\beta) = \bigl(\pi^* H\bigr) {\rm d}\alpha - \alpha \wedge \imath_V {\rm d}\alpha - \bigl(\pi^* H\bigr) {\rm d}\beta + \beta \wedge \imath_V {\rm d}\beta \\
&= \bigl(\pi^* H\bigr) {\rm d} (\alpha - \beta) - (\alpha - \beta) \wedge \imath_V {\rm d}\alpha = \bigl(\pi^* H\bigr) {\rm d}\bigl(\pi^* \eta \bigr) + \bigl(\pi^* \eta\bigr)\wedge \bigl(\pi^* {\rm d}H\bigr) \\
&= \pi^* ( H {\rm d}\eta + \eta\wedge {\rm d}H),
\end{align*}
and since $\pi$ is a~submersion we deduce that
\begin{gather}
\label{e:newide}
H {\rm d}\eta + \eta \wedge {\rm d}H = 0.
\end{gather}
Let $u\colon J:= (\min H, \max H) \rightarrow \R$ be the smooth function
\[
u(c) := \int_{H^{-1}(c)} \eta,
\]
where the circles $H^{-1}(c)$ are oriented by $X_H$. This function has a~continuous extension to $\overline{J}$ with $u(\min H) = u(\max H)=0$. If $Y$ is a~smooth vector field on $W\setminus \{w^-,w^+\}$ such that~${{\rm d}H(Y)=1}$, then~\eqref{e:newide} implies
\begin{align*}
c u'(c) &= c \int_{H^{-1}(c)} \imath_Y {\rm d}\eta = \int_{H^{-1}(c)} \imath_Y (H {\rm d}\eta) = - \int_{H^{-1}(c)} \imath_Y ( \eta \wedge {\rm d}H) \\ &= - \int_{H^{-1}(c)} (\imath_Y \eta) {\rm d}H + \int_{H^{-1}(c)} \eta = u(c),
\end{align*}
and hence
$
c u'(c) - u(c) = 0 $, $ \forall c\in J$.
This equation implies that $u$ vanishes at all its critical points in $J$, and together with the fact that $u$ is continuous on the compact interval $\overline{J}$ and vanishes at its end-points, we conclude that $u$ is identically zero. Therefore, Proposition~\ref{l:solve-coho} implies that the cohomological equation~\eqref{e:cohoeq} has a~smooth (resp.\ analytic) solution $h$.
 This concludes the proof of Theorem~\ref{mainappB}.
 \end{proof}

\subsection*{Acknowledgments}

The authors wish to thank Pietro Majer and Bernd Stratmann for useful conversations on symmetric rearrangements.
We also thank Dan Cristofaro-Gardiner for spending an afternoon with one of us chatting with Chat GPT-5.5 Pro, which eventually produced the example which is mentioned in Remark~\ref{r:smooth?}.
We are indebted to an anonymous referee, who identified a~gap in the proof of Proposition~\ref{p:coniso} in an earlier version and suggested a~way to prove Theorem~\ref{mainappB} in its present form, avoiding an unnecessary hypothesis.

The authors completed part of this work while in residence at the Institute for Advanced Study during the Special Year on Conformally Symplectic Dynamics and Geometry, supported by an anonymous donor.
The second author is partially supported by the ANR grants CoSyDy (ANR-21-CE40-0014) and QuanQual (ANR-26-CE40-0676).

\pdfbookmark[1]{References}{ref}
\LastPageEnding

\end{document}